\documentclass{amsart}

\usepackage{amsmath,amsthm, txfonts, 
eucal, amscd}
\usepackage{graphicx, color}
\usepackage{overpic}

\newtheorem{Theorem}{Theorem}[section]
\newtheorem{Proposition}[Theorem]{Proposition}
\newtheorem{Lemma}[Theorem]{Lemma}

\newtheorem{Corollary}[Theorem]{Corollary}

\theoremstyle{definition}
\newtheorem{Definition}[Theorem]{Definition}
\newtheorem{Remark}[Theorem]{Remark}
\begin{document}
\def\revisionColor{black}
\def\real{\mathbb{R}}
\def\integer{\mathbb{Z}}
\def\supp{\mathrm{supp}\,}
\newcommand{\pomega}{\langle\omega\rangle}
\newcommand{\torsion}{\mathrm{Tor}}
\newcommand{\Bargmann}{{\mathcal{B}}}
\newcommand{\BargmannP}{\mathcal{P}}
\newcommand{\pBargmann}{{\mathfrak{B}}}  
\newcommand{\pBargmannP}{{\mathfrak{P}}}
\newcommand{\complex}{{\mathbb{C}}}
\newcommand{\vol}{m}
\newcommand{\konst}{\varkappa_\sharp}
\newcommand{\lift}{{\mathrm{lift}}}
\newcommand{\bj}{\mathbf{j}}
\newcommand{\bi}{\mathbf{i}}
\newcommand{\trho}{\tilde{\rho}}
\newcommand{\cJ}{\mathcal{J}}
\newcommand{\cL}{\mathcal{L}}
\newcommand{\bbL}{\mathbb{L}}
\newcommand{\bL}{\mathbf{L}}
\newcommand{\bbH}{\mathbb{H}}
\newcommand{\bH}{\mathbf{H}}
\newcommand{\bbI}{\mathbb{I}}
\newcommand{\bI}{\mathbf{I}}
\newcommand{\f}{\bar{f}}
\newcommand{\fF}{\mathfrak{F}}
\newcommand{\bt}{\mathbf{t}}
\newcommand{\bDelta}{\mathbf{\Delta}}
\newcommand{\per}{\mathrm{per}}
\newcommand{\template}{\mathcal{T}}

\title[Exponential mixing for Anosov flows] {Exponential mixing for\\ generic volume-preserving Anosov flows\\ in dimension three }
\author[M.~Tsujii]{Masato TSUJII}
\address{Department of Mathematics, Kyushu University, Fukuoka, 819-0395}
\email{tsujii@math.kyushu-u.ac.jp}
\thanks{The author would like to thank Michihiro Hirayama, Viviane Baladi and Damien Thomine for comments on the manuscript of this paper in earlier stages. This work is partially supported by KAKENHI 15H03627.}
\begin{abstract}
Let $M$ be a closed $3$-dimensional Riemann manifold and let $3\le r\le \infty$. We prove that there exists an open dense subset in the space of $C^r$ volume-preserving Anosov flows on $M$ such that all the flows in it are exponentially mixing. 
\end{abstract}
\date{\today}
\maketitle

\section{Introduction}
In this paper, we study mixing properties of volume-preserving Anosov flows on a \hbox{$3$-dimensional} closed $C^{\infty}$ Riemann manifold $M$. 
The main result is as follows. 
Let $\fF^r_A$ be the space of $C^r$ Anosov flows on $M$ preserving the Riemann volume~$\vol$, equipped with the $C^r$ compact-open topology as a subspace of $C^r(M\times \real, M)$. 
A flow $f^t\in \fF^r_A$ is said to be \emph{exponentially mixing} with respect to the volume $m$ if, for some\footnote{Once the decay estimate \eqref{eq:cor} holds for some $\alpha>0$, we can prove it for any $\alpha>0$ by approximation, possibly with different constants $c_\alpha$ and $C_\alpha$. See \cite[p.1046]{MR1779392}. It is therefore enough to consider  \eqref{eq:cor} for some fixed $\alpha>0$. }  $\alpha>0$, we have 
\begin{equation}\label{eq:cor}
\left|\int \varphi\cdot  (\psi\circ f^t) \,d\vol\right| \le C_{\alpha}\,\|\varphi\|_{C^\alpha}\,\|\psi\|_{C^\alpha}\,\exp(-c_{\alpha}t)
\end{equation}
for any $t>0$ and any $\varphi,\psi\in C^{\alpha}(M)$ satisfying $\int \varphi \, d\vol=0$,
where $c_\alpha$ and $C_\alpha>0$ are constants independent of $\varphi$, $\psi$ and $t$. 
 
\begin{Theorem}\label{th:main} For $3\le r\le \infty$, there exists a $C^3$-open and $C^r$-dense subset\/ $\mathcal{U}\subset \fF^r_A$ such that 
all the flows in $\mathcal{U}$ are exponentially mixing. 
Further there is a $C^3$-open neighborhood of each $f^t\in \mathcal{U}$ such that the decay estimate \eqref{eq:cor} holds  for all the flows in it with uniform constants $c_\alpha$ and $C_{\alpha}$.
\end{Theorem}

It is well known as Anosov alternative\footnote{This is proved in the paper \cite[Theorem 14]{MR0242194}. Actually, in \cite[Theorem 14]{MR0242194}, the first alternative is that the stable and unstable foliations are metrically transitive. 
As explained in \cite[\S 5 in the appendix]{MR0242194}, the last condition implies that $f^t$ is Kolmogorov and hence mixing, whereas $f^t$ is obviously not mixing in the first alternative.} that a volume-preserving Anosov flow is either mixing or $C^1$ conjugate to a suspension flow of an Anosov diffeomorphism by a constant roof function. Also it is not difficult to see that the former alternate holds for almost all cases, say,  for an open dense subset in the space of volume-preserving Anosov flows.
But, for the questions how often the flows are \emph{exponentially} mixing, our knowledge is rather incomplete. 
An ultimate conjecture, known as Bowen-Ruelle conjecture, states that mixing Anosov flows will always be exponentially mixing. But this conjecture is widely open at present. 
In this paper, we investigate a related problem: whether exponential mixing is an open dense property for volume-preserving Anosov flows. 

A few important progresses on the rate of mixing for Anosov flows were made by Chernov\cite{MR1626741} and Dolgopyat\cite{MR1626749,MR1653299,MR1779392} in late 1990's.
In \cite{MR1626749}, Dolgopyat proved that a volume-preserving Anosov flow is exponentially mixing if the stable and unstable foliations are $C^1$ and are not jointly integrable. In particular, it is proved in \cite{MR1626749} that the geodesic flows on negatively curved surfaces are exponentially mixing. (Later this result is extended to higher dimensional contact Anosov flows \cite{MR2113022, MR2652469, MR2995886} and also to those with singularities such as Sinai billiard flows \cite{1506.02836, MR2964773}.)
In \cite{MR1653299} and \cite{MR1779392}, he also studied exponential and rapid ({\it i.e.} super-polynomial)  mixing for suspension flows of subshifts of finite type, which abstracts Axiom A flow, and gave several criteria for such flows to be rapid or exponential mixing. 
Based on the argument in \cite{MR1653299}, Field, Melbourne and T\"or\"ok  proved more recently  in \cite{MR2342697} that  rapid mixing is an open dense property for Axiom~A flows and, in particular, for volume-preserving Anosov flows.    

However, to the author's knowledge, the problem on exponential mixing mentioned above remains open.
The aim of this paper is to study the problem in the simplest possible setting of dimension $3$ and present an affirmative answer in Theorem~\ref{th:main}. This also provides an example of a non-empty open set of volume-preserving Anosov flows which stably exhibit exponential mixing. (Rather surprisingly, no such example has been known.)
%But see \cite{MR3487229} for such an example of Axiom A flows.  

In the following sections, 
we first investigate the geometry of the stable and unstable foliations and introduce the notion of \emph{$s$-template} which describes how the stable subbundle twists along unstable manifolds. 
In Definition \ref{def:ni}, we formulate the non-integrability condition $(NI)_\rho$ for $\rho>0$ in terms of $s$-templates.
In Theorem \ref{th:od}, we show that the condition $(NI)_\rho$ for sufficiently small $\rho>0$ holds for a dense subset in $\fF^r_A$ for any $r\ge 3$. 
Then we prove Theorem \ref{th:main} by showing, in Theorem~\ref{th:exp}, that, if $f_0^t\in \fF^3_A$ satisfies  $(NI)_\rho$ for some $\rho>0$, there is a $C^3$ open neighborhood of $f^t_0$ in $\fF^3_A$ in which all the flows are exponentially mixing with uniform constants $c_\alpha$ and $C_\alpha$ in the decay estimate \eqref{eq:cor}. 

The main novelty in our argument consists in the argument related to $s$-templates in Section~\ref{sec:ni}. 
The idea is quite simple and explained in the following few pages. 
Also the perturbation argument in the proof of Theorem \ref{th:od} in Section \ref{sec:pr_od} may be of some interest, where we consider deformation families of a flow with huge number of parameters and apply large deviation argument in the parameter spaces. The proof of Theorem \ref{th:exp} is obtained by modifying the argument in the author's previous papers \cite{MR2652469, MR2995886}. 
Unfortunately this part is rather long and occupies the remaining two-thirds of this paper, though the core of the argument is presented in a few pages.  
This is because some objects we consider are not smooth and require careful treatment. Still our argument is basically elementary and straightforward. If the reader is familiar with  estimates on non-linearity of hyperbolic flows and/or Fourier analysis, one will be able to  skip good part of the argument and/or find better ways to prove the propositions in this part.   

\begin{Remark}
The argument presented in this paper depends crucially on the assumptions that $M$ is three dimensional and that $f^t$ preserves a smooth volume. So it will not extend to more general cases directly. But the author would like to emphasize that the following idea behind the argument will be useful in much  more general cases of partially hyperbolic dynamical systems: {\it 
Twist of the stable subbundle along pieces of unstable manifolds viewed in the unit scale will be ``random" and ``rough" in generic cases and such  ``random" twist will not be cancelled completely in the process where the flow $f^t$ contracts the piece of unstable manifold to microscopic scale as $t\to -\infty$ (if we view things in an appropriate scaling), because the contraction is exponential and therefore  only Taylor approximation of $f^t$ up to some finite order will be effective.}
See also Remark~\ref{Rem:contact}. 
\end{Remark}

%\begin{Remark}
%Recently, the author was informed by 
%Oliver Butterly (ICTP, Italy) that the main theorem, Theorem \ref{th:main}, of this paper is true also for $4$-dimensional case, that is, for volume-preserving Anosov flows on $4$-dimensional manifolds. 
%\end{Remark}
%%%%%%%%%%%%%%%%%%%%%%%%
\section{The non-integrability condition}
\label{sec:ni}
%%%%%%%%%%%%%%%%%%%%%%%% 

Let $3\le r\le \infty$ and consider a $C^r$ Anosov flow $f^t:M\to M$. We suppose that the flow $f^t$ preserves a $C^{r}$ volume $\mu$ on $M$. Here, by a technical reason, we do \emph{not} assume that $\mu$ is the Riemann volume $\vol$. 
Let $v$ be the $C^{r-1}$ vector field generating the flow $f^t$. 
We suppose that $\|v\|\equiv 1$ for the Riemann metric $\|\cdot\|$ on $M$.  
Since the argument below does not depend on the Riemann metric essentially, this does not cause any  loss of generality. 

In some points in the argument below, we will need to check that some constants can be taken uniformly for the flows in a sufficiently small $C^r$ neighborhood of $f^t$ that preserve $C^r$ volume forms sufficiently close to $\mu$. 
In order to distinguish such constants, we put the subscript $*$ to the symbols of them. 
We use $C_*$ as a generic symbol for such class of constants and write $C_*(\cdot)$ when we emphasize their dependence on some quantity in the parentheses. Also we write $\mathcal{O}_*(\cdot)$ for a term which is bounded in absolute value by the quantity inside the parenthesis multiplied by some constant $C_*$. 
\subsection{Anosov flows}
From the definition of Anosov flow, there is an $f^t$-invariant continuous decomposition of the tangent bundle 
\begin{equation}\label{eq:hyperbolic_splitting}
TM=E_0\oplus E_s\oplus E_u \quad\text{with}\quad \dim E_0=\dim E_s=\dim E_u=1
\end{equation}
such that $E_0=\langle v\rangle$ and that, for some  positive constants $C_*>0$ and $\chi_*>0$, 
\begin{equation}\label{eq:hyperbolicity}
\|Df_x^t|_{E_s}\|\le C_*e^{-\chi_* t},\quad \|Df_x^{t}|_{E_u}\|\ge C_*^{-1} e^{\chi_* t}\quad \mbox{for all }t\ge 0.
\end{equation}
The decomposition  dual to \eqref{eq:hyperbolic_splitting} is $
T^*M=E_0^*\oplus E_s^*\oplus E_u^*$ where
\[
E_0^*=(E_s\oplus E_u)^\perp,\quad E_s^*=(E_u\oplus E_0)^\perp,\quad E_u^*=(E_s\oplus E_0)^\perp.
\]
The subbundle $E_0$ is $C^{r-1}$, but $E_s$ and $E_u$ are not even $C^1$ in general. 
However we have  
\begin{equation}\label{eq:irregularity_distribution}
\angle(E_s(p),E_s(q))\le C_* \|p-q\|\cdot \langle \log \|p-q\|\rangle
\end{equation}
in local charts\footnote{It is of course possible to formulate \eqref{eq:irregularity_distribution} without using local charts by introducing a parallel transport. }, where (and henceforth) $\langle s\rangle$ denotes\footnote{The notation $\langle s\rangle$ is used often in analysis and sometimes called Japanese bracket. By a technical reason, our definition is slightly different from the usual one $\langle s\rangle=\sqrt{1+s^2}$. But the difference is not essential at all. }  some fixed $C^\infty$ function of $s$ such that
\[
\langle s\rangle=
|s|\quad\text{if $|s|\ge 2$}\qquad \text{and} \qquad \langle s\rangle\ge 1 \quad \text{for any $s$}.
\] 
The estimate \eqref{eq:irregularity_distribution} holds also for the subbundles $E_u$ and  $E_0^*$. 
\begin{Remark}\label{rem:eseu} The non-smoothness of $E_s$ and $E_u$ mentioned above is caused mainly by their variation in the flow direction. In fact, 
the sums $E_u\oplus E_0$ and $E_s\oplus E_0$ are $C^1$ and so are  their normals $E_s^*$ and  $E_u^*$.  Especially we have 
\begin{equation}\label{eq:irregularity_distribution2}
\angle(E_s^*(p),E_s^*(q))\le C_* \|p-q\|,\qquad
\angle(E_u^*(p),E_u^*(q))\le C_* \|p-q\|
\end{equation}
in local charts.  
\end{Remark}

%%%%%%%%%%%%%%%%%%%%%%%
\subsection{The intrinsic metric on stable and unstable manifolds}\label{ss:intrinsic_metric}
%%%%%%%%%%%%%%%%%%%%%%%
Let $W^s(p)$ and $W^u(p)$ be the stable and unstable manifolds passing through a point $p\in M$. 
Below we discuss about twist of the stable subbundle $E_s$ along  $W^u(p)$. But note that, by considering the time-reversed flow ~$f^{-t}$, we can (and will) argue  about twist of the unstable subbundle $E_u$ along $W^s(p)$ in parallel. 
To begin with, let us introduce a $C^{r-1}$ metric on $W^u(p)$ by
\begin{equation}\label{eq:intrinsice_metric}
|v|_{W^u(p)}
=
\lim_{t\to -\infty} \frac{\|Df_q^t(v)\|}{\|Df_p^t|_{E_u}\|}\qquad \text{for $v\in T_qW^u(p)$ at $q\in W^u(p)$. }
\end{equation}
The following lemma is an immediate consequence of this definition.
\begin{Lemma}\label{eq:intrinsic_parameter}
 If $f^t$ sends $W^u(p)$ to $W^u(p')$, it brings the  metric $|\cdot|_{W^u(p)}$ to  $|\cdot|_{W^u(p')}$ up to multiplication by a positive constant. If $f^t(p)=p'$, the multiplier is just $\|Df^t_p|_{E^u}\|$. 
\end{Lemma}

Let $w^u_p:\real\to M$ be the $C^r$ parametrization of $W^u(p)$ by the arc length with respect to the metric $|\cdot|_{W^u(p)}$ such that $w^u_p(0)=p$. (We do not care about the direction of parametrization.) For an interval $J\subset \real$, we set $W^u_{J}(p):=w^u_p(J)\subset W^u(p)$.

 %%%%%%%%%%%%%%%%%%%%%%%%%%%%%
 \subsection{Some sections of the normal bundles of unstable manifolds}\label{ss:cotan}
 %%%%%%%%%%%%%%%%%%%%%%%%%%%%%
For a point $p\in M$ and an interval $J\subset \real$,  
let $\Gamma^{u}(p,J)$ be the space of continuous sections $
\gamma:W^u_{J}(p)\to T^*M$ such that $\gamma(q)\in T_q^*M$  at each $q\in W^u_J(p)$ is normal to the tangent space $T_qW^u(p)=E_u(q)$.
Let $\Gamma_1^{u}(p,J)\subset \Gamma^{u}(p,J)$ be the subset that consists  $\gamma\in \Gamma^{u}(p,J)$ satisfying  $
\langle \gamma(q),v(q)\rangle\equiv 1$ where $v(\cdot)$ denotes the generating vector field of the flow $f^t$.

We write $\gamma^\perp_{p,J}\in \Gamma^{u}(p,J)$ for either of the two $C^{r-1}$ sections satisfying 
\[
\langle \gamma^\perp_{p,J}(q), u\rangle
=\pm \mu(v(q),(w^u_p)'(\tau),u)
\quad \text{for any $u\in T_q M$ at $q=w^u_p(\tau)$ and for any $\tau \in J$,}
\]
where $\mu$ is the volume preserved by $f^t$. 
We tentatively fix a $C^{r-1}$ section $\gamma^{0}_{p,J}$ in $ \Gamma_1^{u}(p,J)$. 
For the moment  we assume only that the sections $\gamma^{0}_{p,J}$ are bounded in $C^{r-1}$ sense uniformly for $p\in M$ and $J\subset (-1,1)$.
We may then express each section $\gamma\in \Gamma^u_1(p,J)$ as 
\begin{equation}\label{eq:expression0}
\gamma(q)=\gamma^0_{p,J}(q)+\psi_\gamma(\tau)\cdot \gamma^\perp_{p,J}(q)\quad\text{for $q=w^u_p(\tau)$ with $\tau\in J$,}
\end{equation}
where $\psi_\gamma:J\to \real$ is a continuous function and called the  representation function of~$\gamma$. For a $C^{r-1}$ section $\gamma\in \Gamma^u_1(p,J)$, we define its (maximum) curvature $\kappa(\gamma)$  by 
\[
\kappa(\gamma)=\sup\{\,|\psi''_{\gamma}(\tau)|\mid \tau\in J\,\}.
\]

For a $C^{r-1}$ section $\gamma\in \Gamma^u_1(p,J)$ and $t\in \real$, 
there is a unique section 
$\gamma_t\in \Gamma^u_1(f^t(p),J(t))$ with $J(t)=\pm \|Df^t|_{E_u}(p)\|\cdot J$ so that   
\begin{equation}\label{eq:gammat}
\gamma(q)=(Df^t)^*\gamma_t(f^t(q)).
\end{equation} 
Observe that the curvature $\kappa(\gamma_t)$ of $\gamma_t$ tends to infinity as $t\to -\infty$ in most cases, but may be bounded for some $\gamma$. This motivates  the following definition.
\begin{Definition}
A $C^{r-1}$ section $\gamma\in \Gamma^u_{1}(p,J)$ is said to be \emph{straight} if $\kappa(\gamma_t)$ is bounded uniformly for all $t\le 0$.
\end{Definition}
Note that this definition does not depend on the choice of $\gamma^{0}_{p,J}$ provided that they are uniformly bounded in  $C^{r-1}$ sense (as we are assuming) and is therefore intrinsic to the flow~$f^t$. In the next lemma,  we describe the space of straight sections.
\begin{Definition}
Two functions $\psi_0,\psi_1:J\to \real$ are said to be equivalent modulo affine functions if $\psi_0(\tau)=\psi_1(\tau)+\alpha \tau+\beta$  for $\tau\in J$ with some $\alpha,\beta\in \real$. Two sections $\gamma_0,\gamma_1\in \Gamma^u_1(p,J)$ are said to be 
\emph{affine equivalent} if their representation functions $\psi_{\gamma_0}$ and $\psi_{\gamma_1}$ (defined  in \eqref{eq:expression0}) are equivalent modulo affine functions.   
\end{Definition}  
\begin{Lemma}\label{lm:straight} For any point $p\in M$ and any interval $J\subset \real$, 
there exists a straight section $\gamma_0\in \Gamma^{u}_1(p,J)$. A $C^{r-1}$ section $\gamma \in \Gamma_1^{u}(p,J)$ is  straight  if and only if it is affine equivalent to $\gamma_0$. 
If a $C^{r-1}$ section $\gamma\in \Gamma^u_1(p,J)$ is straight, then $\gamma_t\in \Gamma^u_1(f^t(p), J(t))$ for $t\in \real$ is again straight. 
\end{Lemma}

\begin{proof}
If $f^t$ sends a $C^{r-1}$ section $\gamma\in \Gamma_1^u(p,J)$ to $\gamma_t\in \Gamma^u_1(p(t), J(t))$, the representation functions of $\gamma_t$ is related to that of $\gamma$ by the formula
\begin{equation}\label{eq:recur}
\psi_{\gamma_t}(\tau)=a(t)\cdot \psi_{\gamma}(a(t)^{-1} \tau)+\varphi_{p,t}(\tau),\qquad a(t)=\pm \|Df^t_p|_{E_u}\|
\end{equation}
where $\varphi_{p,t}$ is a $C^{r-1}$ function that stems from the difference between $\gamma^0_{f^t(p),J(t)}$ and the push-forward of $\gamma^0_{p,J}$ by $f^t$. 
Note that, for any $t_0>0$, the $C^{r-1}$ norm of the function $\varphi_{p,t}$ is bounded uniformly for $p\in M$ and $t\in \real$ with $|t|\le t_0$. 
\begin{Remark}\label{rm:sign}
The indefiniteness of the sign of $a(t)$ in \eqref{eq:recur} is due to that in the direction of the parametrization $w^u_p$. 
In most part of the argument below, we will ignore the related indefiniteness of sings because they are not essential at all and easy to fix if one likes. 
\end{Remark}
Differentiating the both sides of \eqref{eq:recur} with respect to $\tau$ twice and changing the variable $\tau$ to $a(t)^{-1}\tau$, we obtain the relation
\begin{equation}\label{eq:recurdd}
\psi''_{\gamma}(\tau)=a(t)\, \psi''_{\gamma_t}(a(t)\tau)-a(t)\,\varphi''_{p,t}(a(t)\tau).
\end{equation}
The claims of the lemma are consequences of this relation. 
Let us consider  
 a sequence 
\[
t(0)=0>t(1)>t(2)>\dots \to -\infty\quad \text{ such that }t_0/2\le t(i)-t(i+1)\le t_0.
\]  
By letting $t_0$ larger if necessary, we may and do assume
\[
a_i:= \|Df^{t(i+1)-t(i)}_{p(i)}|_{E_u}\|=a(t(i+1))/a(t(i))\, \in [-1/2,1/2].
\]
If we set $t=t(i+1)-t(i)$ and replace $p$ with $p(i)=f^{t(i)}(p)$ in \eqref{eq:recurdd}, we find the formula
\begin{equation}\label{eq:recursive_rel2}
\psi''_{\gamma_{t(i)}}(\tau)=a_i\, \psi''_{\gamma_{t(i+1)}}(a_i\cdot\tau)-a_i\,\varphi''_{p(i),t(i+1)-t(i)}(a_i\cdot \tau)\quad \text{for $\tau\in [-1,1]$.}
\end{equation}
Recursive application of this formula yields, for any integer $N>0$ and $\tau\in [-1,1]$,  
\[
\psi''_{\gamma}(\tau)=a(t(N))\cdot  \psi''_{\gamma_{t(N)}}(a(t(N))\cdot \tau)-\sum_{i=0}^{N-1} a(t(i+1))\,\varphi''_{p(i),t(i+1)-t(i)}(a(t(i+1))\cdot \tau).
\]
Since $|a(t)|\le C_* e^{\chi_* t}$ for $t\le 0$ by \eqref{eq:hyperbolicity}, the right-hand side converges to a unique continuous function as $N\to \infty$ in $C^{0}$ sense, provided that $\kappa(\gamma_t)=\|\psi''_{\gamma_t}\|_\infty$ is uniformly bounded for $t\le 0$. That is to say, if $\gamma\in \Gamma_1^u(p,J)$ is straight, its representation function satisfies 
\begin{equation}\label{eq:rep_exp}
\psi''_{\gamma}(\tau)=-\sum_{i=0}^{\infty} a(t(i+1))\,\varphi''_{p(i),t(i+1)-t(i)}(a(t(i+1))\cdot \tau)\quad \text{for $\tau\in [-1,1]$.}
\end{equation}

Conversely, suppose that $\gamma\in \Gamma_1^u(p,J)$ satisfies the last condition \eqref{eq:rep_exp}. Then $\psi''_\gamma$ is of class $C^{r-3}$ and  $\gamma$ is of class $C^{r-1}$. The relation \eqref{eq:recursive_rel2} gives
\[
\psi''_{\gamma_{t(j)}}(\tau)=-\sum_{i=j}^{\infty} 
\frac{a(t(i+1))}{a(t(j))}\,\varphi''_{p(i),t(i+1)-t(i)}\left(\frac{a(t(i+1))}{a(t(j))}\cdot \tau\right) \quad \text{for $\tau\in [-a(t(j)),a(t(j))]$}
\]
and the right-hand side is bounded uniformly in $j\ge 0$, so that $\gamma$ must be straight. 

Clearly the former two statements of the lemma follow from the argument above. The last statement is an immediate consequence of the definition of straight section. 
\end{proof}
Since the choice of the sections $\gamma^0_{p,J}$ was rather arbitrary, we henceforth assume without loss of generality that the sections $\gamma^0_{p,J}$  are straight sections. (This is just for avoiding a new notation.) Further, in the case $J=(-1,1)$, we specify $\gamma^0_{p,J}$ as the unique straight section satisfying the following conditions at the end points:
\begin{equation}\label{eq:specyfy_gamma0}
\lim_{\tau\to \sigma} \gamma^0_{p,(-1,1)}(w^u_p(\tau)) \in E_0^*(w^u_p(\sigma))\quad \text{for $\sigma=\pm 1$.} 
\end{equation}

%%%%%%%%%%%%%%%%%%%%%%%%%%%%%%%%%%%%%%%
\subsection{The definition of $s$-templates} 
%%%%%%%%%%%%%%%%%%%%%%%%%%%%%%%%%%%%%%%
We next consider how the direction of the stable subspace $E_s$ twists along the local unstable manifold $W^u_J(p)$ when we observe it in the frame given in the last subsection. 
Let $\gamma^s_{p,J}\in\Gamma_1^{u}(p,J)$ be the unique continuous section such that $\gamma^s_{p,J}(q)\in E_0^*(q)=(E_u(q)\oplus E_s(q))^\perp$ for $q\in W^u_{J}(p)$ and let $\psi^s_{p,J}:J\to \real$ be its representation function. 
\begin{Remark}   
The function $\psi^s_{p,J}$ captures  the variation of the stable subbundle $E_s$ along $W^u_J(p)$. But note that we consider only the component normal to $W^u_{J}(p)$. The variation of the other component will turn out to be negligible. (See Remark \ref{rem:eseu}.)
\end{Remark}
The function $\psi^s_{p,J}$ is not even $C^1$ in general but satisfies
\begin{equation}\label{eq:regularity_psi_s}
|\psi^s_{p,J}(\tau')-\psi^s_{p,J}(\tau)|\le C_* |\tau'-\tau| \cdot \langle \log |\tau'-\tau|\rangle\quad\text{for }\tau,\tau'\in J 
\end{equation}
as a consequence of \eqref{eq:irregularity_distribution}. Now we introduce

\begin{Definition}[$s$-templates]
The functions $\psi^s_{p,(-1,1)}$ for $p\in M$ are called the \emph{$s$-templates} for the flow $f^t$. We write $
\template=\template(f^t)=\{\psi^s_{p,(-1,1)}\mid p\in M\}$ for the set of  all $s$-templates for the flow $f^t$.
\end{Definition}
Note that, from the condition \eqref{eq:specyfy_gamma0} in the choice of $\gamma^0_{p,(-1,1)}$, the $s$-templates satisfy 
\begin{equation}\label{eq:boundary_cond}
\psi^s_{p,(-1,1)}(\pm 1):=\lim_{\tau\to \pm 1} \psi^s_{p,(-1,1)}(\tau)=0.
\end{equation}
The next lemma tells that the twist of $E_s$ along the unstable manifolds in a microscopic scale is a miniature of an $s$-template up to affine equivalence. 
(This is the reason for the name ``template".) 
\begin{Lemma}\label{lm:psis}
For any $q\in M$ and any  $\delta\in (0,1)$, there exist $t> 0$ such that  
\begin{equation}\label{eq:psis}
\psi^s_{q,(-\delta,\delta)}(\tau)=\delta\cdot \psi_{p,(-1,1)}^s(\delta^{-1}\tau)+\alpha \tau +\beta\qquad \text{with} \quad p=f^t(q)
\end{equation}
where $|\alpha|\le C_*\langle \log \delta\rangle$ and $|\beta|\le C_*$.
\end{Lemma}
\begin{proof}
Let $q\in M$ and $0<\delta<\delta'\le 1$. We
take $t>0$ satisfying $\|Df^t_q|_{E_u}\|=\delta'/\delta$, so that  $f^t(W^u_{(-\delta,\delta)}(q))=W^u_{(-\delta',\delta')}(p)$ with $p=f^t(q)$. 
Let us recall the relation \eqref{eq:recur} and find that
\begin{equation}\label{eq:recur2}
\psi_{\gamma_t}(\tau)=(\delta'/\delta) \cdot \psi_{\gamma}( (\delta/\delta') \tau)+\varphi_{q,t}(\tau)
\end{equation}
for any section $\gamma \in \Gamma^u_1(q,(-\delta,\delta))$ and its image $\gamma_t\in  \Gamma^u_1(p,(-\delta',\delta'))$ by $f^t$ that is defined by the relation \eqref{eq:gammat}. (See also Remark \ref{rm:sign}.)

We show that the function $\varphi_{q,t}(\tau)$ in \eqref{eq:recur2} is an affine function. 
To see this, we let $\gamma \in \Gamma^u_1(q,(-\delta,\delta))$ be the pull-back of the section $\gamma^0_{p,(-\delta',\delta')}$ by $f^t$, so that $\gamma_t=\gamma^0_{p,(-\delta',\delta')}$. Then the representation function of $\gamma_t=\gamma^0_{p,(-\delta',\delta')}$ on the left-hand side is null by definition. Also that of $\gamma$ on the right-hand side is an affine function because it is a straight section from Lemma~\ref{lm:straight}. Therefore $\varphi_{q,t}(\tau)$ is also an affine function. 
 
In order to get the conclusion of the lemma, we set $\gamma=\gamma^s_{q,(-\delta,\delta)}\in \Gamma^u_1(q,(-\delta,\delta))$ in \eqref{eq:recur2}. Then, from invariance of $E_0^*$, we see
\[
\psi^s_{p,(-\delta',\delta')}(\tau)=(\delta'/\delta) \cdot \psi^s_{q,(-\delta,\delta)}( (\delta/\delta') \tau)+\varphi_{q,t}(\tau).
\]
Changing the variable $\tau$ to $(\delta'/\delta)\tau$, we rewrite it as 
\begin{equation}\label{eq:eq:psis_recursive}
\psi_{q,(-\delta,\delta)}^s(\tau)= (\delta/\delta') \cdot \psi^s_{p,(-\delta',\delta')}((\delta'/\delta)\tau)- (\delta/\delta')\cdot \varphi_{q,t}((\delta'/\delta)\tau).
\end{equation}
We obtain the formula \eqref{eq:psis} as the case $\delta'=1$. The required estimate on $\alpha$ is obtained by applying \eqref{eq:eq:psis_recursive} with $\delta'/\delta$ bounded recursively (as in the proof of Lemma \ref{lm:straight}) and by using the fact that the affine function $\varphi_{q,t}(\tau)$ is bounded provided that  $\delta'/\delta$ (or $t$) is bounded. The required estimate on $\beta$ should be obvious. 
\end{proof}

\begin{Remark}
\label{Rem:time-reversal}
As we noted in the beginning of Subsection \ref{ss:intrinsic_metric}, we can develop the argument above for the time-reversed flow $f^{-t}$ in parallel. The objects corresponding to
\begin{equation}\label{eq:sense1}
|\cdot|_{W^u(p)},\quad W^u_J(p),\quad w^u_p(\cdot),\quad \Gamma^u(p,J),\quad\Gamma^u_1(p,J),\quad\gamma^{\perp}_{p,J},\quad \gamma^0_{p,J},\quad \gamma^s_{p,J}, \quad \psi^s_{p,J} 
\end{equation}
in such argument will be denoted respectively by 
\begin{equation}\label{eq:sense2}
|\cdot|_{W^s(p)},\quad W^s_J(p),\quad w^s_p(\cdot),\quad \Gamma^s(p,J),\quad \Gamma^s_1(p,J), \quad\hat{\gamma}^{\perp}_{p,J}, \quad\hat{\gamma}^0_{p,J}, \quad \gamma^u_{p,J},\quad \psi^u_{p,J}. 
\end{equation}
\end{Remark}

%%%%%%%%%%%%%%%%%%%%%%%%%%%%%%%%%%%%%%%%%%%%
\subsection{The non-integrability condition}
%%%%%%%%%%%%%%%%%%%%%%%%%%%%%%%%%%%%%%%%%%%%
Now we put the following definition.
\begin{Definition} \label{def:ni}
Let $0<\rho<1$.
We say that 
a $C^3$ Anosov flow $f^t$ on $M$ preserving a smooth volume $\mu$ satisfies the non-integrability condition $(NI)_{\rho}$  if, for sufficiently large $b>0$, the estimate 
\begin{equation}\label{eq:tv}
\left|\int_{-1}^1 \exp\left(ib \left(\psi(\tau)+\alpha \tau\right)\right) d\tau\right|<
b^{-\rho}
\end{equation}
holds for all $s$-templates $\psi\in \template(f^t)$ and $\alpha\in \real$. 
\end{Definition}
\begin{Remark}\label{Rem:largealpha}
From \eqref{eq:regularity_psi_s} (or \eqref{eq:irregularity_distribution}), the $s$-templates $\psi\in \template(f^t)$ are H\"older continuous with any exponent $0<\beta<1$ and the corresponding  H\"older coefficients are bounded by a uniform constant $C_{\beta,*}$. Hence, for each $0<\rho<1$, the condition \eqref{eq:tv} holds for free  if $|\alpha|>b$ and $b$ is sufficiently large. (For instance, we can check this by using ``regularized" integration by parts given in Lemma \ref{lm:reg_int_part}.) 
\end{Remark}
\begin{Remark}\label{Rem:indep_NI_metric} 
From Lemma \ref{lm:psis}, the non-integrability condition $(NI)_{\rho}$ remains unchanged even if  we replace the Riemann metric on $M$ by another Riemann metric and the volume $\mu$ by its scalar multiple.  
\end{Remark}
\begin{Remark}\label{Rem:contact}
In the case of contact Anosov flows (see \cite{MR2113022} for the definition), the set $\template$ of $s$-templates consists of a single trivial equivalence class $[0]$ modulo affine functions. (To check this, observe that $\gamma^s(q,J)$ is given by the contact form restricted to $W^u_{J}(p)$ and is straight because  the contact form is preserved by the flow.)
Therefore our non-integrability condition $(NI)_{\rho}$ excludes the case of contact Anosov flows!  
\end{Remark}

The main theorem, Theorem \ref{th:main}, follows if we prove the following two theorems.
\begin{Theorem}\label{th:od}
Let $3\le r< \infty$. 
If  we let $0<\rho<1$ be sufficiently small depending only on~$r$, the set of flows that satisfy the non-integrability condition $(NI)_{\rho}$ is dense in $\fF^r_A$. 
\end{Theorem}

\begin{Theorem}\label{th:exp}
If a flow $f^t_0\in \fF^3_{A}$ satisfies the non-integrability condition $(NI)_{\rho}$ for some $0<\rho<1$, there exists an open neighborhood $\mathcal{V}$ of $f_0^t$ in  $\fF^3_A$ such that all $f^t\in \mathcal{V}$ are exponentially mixing and further that  the decay estimate \eqref{eq:cor} holds for all $f^t\in \mathcal{V}$ with uniform
  constants $c_{\alpha}$ and $C_\alpha$. 
\end{Theorem}
We prove Theorem \ref{th:od} in the next section, Section \ref{sec:pr_od}. 
We prove Theorem \ref{th:exp} in Section \ref{sec:proof1}, after preparations in Section \ref{sec:lc} and Section \ref{sec:anisoH}.

\subsection{Approximate  non-integrability}
\label{ss:approxNI}
We finish this section by a discussion on another important idea about joint non-integrability of the stable and unstable foliation, which is closer to the ideas of Frobenius non-integrability and the uniform non-integrability condition introduced by Chernov\cite{MR1626741}.   
Let us consider how the flow $f^t$ twists the tangent bundle along local unstable (resp. stable) manifolds (in a more literal sense). 
Consider a point $q\in M$ and a positive number $0<\delta<1$. 
Recall that we have specified the straight sections $\gamma^0_{q,J}$ uniquely when $J=(-1,1)$ by the condition \eqref{eq:specyfy_gamma0}, but not yet for the case $J= (-\delta,\delta)$ with $0<\delta<1$. 
There are two natural but different ways to choose a straight section in $\Gamma_1^{u}(q,(-\delta,\delta))$:
\begin{itemize}
\item[(a)] we take it as a restriction of $\gamma^0_{q,(-1,1)}$ to $W^u_{(-\delta,\delta)}(q)\subset W^u_{(-1,1)}(q)$, or 
\item[(b)] we take $t> 0$ such that $
f^t(W^u_{(-\delta, \delta)}(q))=W^u_{(-1,1)}(p)$ with $p=f^t(q)$ and let it be the pull-back of $\gamma^0_{p,(-1,1)}\in \Gamma^u_1(p,(-1,1))$ by $f^t$. 
\end{itemize}
Let us denote the straight sections obtained in (a) and (b) by $\gamma^{0}_{q,(-\delta,\delta)}$ and $\gamma^{\dag}_{q,(-\delta,\delta)}$ respectively. 
They are both straight sections and hence affine equivalent, that is, 
\begin{equation}\label{eq:exp_gamma_ast}
\gamma^\dag_{q,(-\delta,\delta)}(\tau)=
\gamma^0_{q,(-\delta,\delta)}(\tau)+\psi^\dag_{q,(-\delta,\delta)}(\tau)\cdot \gamma^\perp_{q,(-\delta,\delta)}(\tau)
\end{equation}
for an affine function $\psi^\dag_{q,(-\delta,\delta)}(\tau)$. 
The linear part of $\psi^\dag_{q,(-\delta,\delta)}(\tau)$ may be understood as the torsion that $f^t$ (with $t$ in (b) above) makes along $W^u_{(-\delta, \delta)}(q)$. This motivate us to define 
\begin{equation}\label{def:tors}
\torsion^s(q,\delta):=(\psi^\dag_{q,\delta})'(0).
\end{equation}

As we noted in Remark \ref{Rem:time-reversal}, we can apply the parallel argument to the time-reversed flow $f^{-t}$ and define 
\begin{equation}\label{eq:sense3}
\hat{\gamma}^\dag_{q,(-\delta,\delta)}, \quad\hat{\psi}^\dag_{q,\delta},\quad \torsion^u(q,\delta)
\end{equation}
as the objects corresponding to 
\begin{equation}\label{eq:sense4}
{\gamma}^\dag_{q,(-\delta,\delta)},\quad  {\psi}^\dag_{q,\delta},\quad  \torsion^s(q,\delta).
\end{equation}
These extends the correspondence\footnote{The correspondence between \eqref{eq:sense1}, \eqref{eq:sense3} and \eqref{eq:sense2}, \eqref{eq:sense4} may look a bit confusing. But, since we will not use the notation for the time-reversed system very often, the readers do not have to care too much about it.} between \eqref{eq:sense1} and \eqref{eq:sense2}.

Recall that there are options of choosing signs in the definitions of $\gamma^\perp_{q,(-\delta,\delta)}$ and $\hat{\gamma}^\perp_{q,(-\delta,\delta)}$. In the following definition, we suppose that the signs are chosen so that
\[
\langle \gamma^\perp_{q,(-\delta,\delta)}(0), (w^s_{q,(-\delta,\delta)})'(0)\rangle>0,\qquad 
\langle \hat{\gamma}^\perp_{q,(-\delta,\delta)}(0), (w^u_{q,(-\delta,\delta)})'(0)\rangle>0.
\]

\begin{Definition}\label{def:approx_ni}
For $q\in M$ and $0<\delta<1$, we set 
\begin{align}
\Delta(q,\delta)&=\torsion^u(q,\delta)-\torsion^s(q,\delta)
\end{align}
and call it the \emph{approximate non-integrability} at $q\in M$ in the scale $\delta$.
\end{Definition}

\begin{Remark}
We regard the quantity $\Delta(q,\delta)$ as an approximation of Frobenius non-integrability between stable and unstable foliation at $q\in M$ viewed  in the scale $\delta>0$. 
It is natural to expect that $\Delta(q,\delta)$ will take large values for most of small $\delta>0$ and most points $q\in M$ in generic cases. 
However the problem with the quantity $\Delta(q,\delta)$ is its dependence on the scale $\delta$. When we perturb the flow, it is difficult to see how the quantity $\Delta(q,\delta)$ varies for small $\delta>0$. 
On the contrary, our non-integrability condition $(NI)_{\rho}$ is formulated in terms of  \hbox{$s$-templates} and does not involve the scale $\delta$. This is the main technical advantage of  the notion of $s$-templates. Note however that we will make use of the quantity $\Delta(q,\delta)$ in the proof of Theorem \ref{th:exp}. In fact, we will  use 
the non-integrability condition $(NI)_{\rho}$ only in the situation where $|\Delta(q,\delta)|$ is not large enough. 
\end{Remark}

The next lemma gives a few basic properties of the quantities we have introduced.
\begin{Lemma}\label{lm:Tor} For $\sigma=s,u$, $0<\delta,\delta'<1$, $q\in M$ and $t\in \real$, we have 
\begin{align}\label{eq:torsion_bound}
&|\torsion^\sigma(q,\delta)-\torsion^\sigma(q,\delta')|< C_* \langle\log (\delta'/\delta)\rangle\quad\text{and hence}\quad
|\torsion^\sigma(q,\delta)|< C_* \langle\log \delta\rangle,\\
&|\torsion^\sigma(f^t(q),\delta)-\torsion^\sigma(q,\delta)|\le C_*\langle t\rangle,\quad\text{and hence } |\Delta(f^t(q),\delta)-\Delta(q,\delta)|\le C_*\langle t\rangle
\label{eq:torsion_bound2}
\intertext{and}
&|\torsion^\sigma(q',\delta)-\torsion^\sigma(q,\delta)|\le C_*\quad\text{if $d(q,q')<\delta$.} \label{eq:torsion_bound3}
\end{align} 
\end{Lemma}
\begin{proof} We prove the claims in the case $\sigma=s$. We can prove those in the case $\sigma=u$ in the parallel manner considering the time reversed flow $f^{-t}$. Note first of all that $\torsion^s(q,1)=0$ by definition. From the definition of $\torsion^s(q,\delta)$, it is easy to see that, for fixed $t_0>0$, there exist $C_*=C_*(t_0)>0$ satisfying
\[
|\torsion^s(q,\delta)-\torsion^s(f^t(q),\|Df^t_q|_{E_u}\|\cdot \delta)|\le C_* \quad  \text{for $0\le t\le t_0$, $q\in M$ and $0<\delta\le \|Df^t_q|_{E_u}\|^{-1}$.}
\]
Recursive application of this estimate yields 
\begin{equation}\label{eq:torsion1}
|\torsion^s(q,\delta)-\torsion^s(f^t(q),\|Df^t_q|_{E_u}\|\cdot \delta)|\le C_*\langle t\rangle  \quad  \text{for $t>0$ and $0<\delta\le \|Df^t_q|_{E_u}\|^{-1}$.}
\end{equation}
In particular, letting $t\ge 0$ be such that $\|Df^t_q|_{E_u}\|\cdot  \delta=1$, we get
\[
|\torsion^s(q,\delta)|\le C_*\langle \log \delta \rangle  \quad  \text{for  $0<\delta\le 1$.}
\]
For $0<\delta\le \delta'\le 1$ and $t\ge 0$ such that $\|Df^t_q|_{E_u}\|\cdot \delta'=1$, we have 
\begin{equation}\label{eq:torsion2}
|\torsion^s(q,\delta)-\torsion^s(q,\delta')|=
|\torsion^s(f^t(q),\delta/\delta')-\torsion^s(f^t(q),1)|=|\torsion^s(f^t(q),\delta/\delta')|
\end{equation}
where the first equality is a consequence of the definition of $\torsion^s(q,\delta)$ and Lemma \ref{lm:straight}. 
The last two estimates yield \eqref{eq:torsion_bound}. Then \eqref{eq:torsion_bound2} follows from \eqref{eq:torsion_bound} and \eqref{eq:torsion1}. 
To prove the last claim \eqref{eq:torsion_bound3}, it is enough to show the claim   
\[
|\torsion^s(q',\delta)-\torsion^s(q,\delta)|\le C_*
\]
for sufficiently small  $\delta>0$ and for any  $q,q'\in M$ satisfying either 
\begin{center} (i) $q'\in W^u_{(-\delta,\delta)}(q)$, \quad (ii) 
$q'\in W^s_{(-\delta,\delta)}(q)$\;\; and \;\;(iii) $q'\in f^\tau(q)$ for some $\tau\in (-\delta, \delta)$.
\end{center}
In the case (i), each straight section on $W^u_{(-\delta,\delta)}(q)$ extends uniquely to that on $W^u_{(-\delta,\delta)}(q')$. 
In the case (iii), the flow $f^\tau$ brings each straight section on $W^u_{(-\delta,\delta)}(q)$ to that on $W^u_{(-\delta',\delta')}(q')$ with $\delta'=\|Df^\tau_q|_{E_u}\|\cdot  \delta$, which extends (or restricts) uniquely to that on
$W^u_{(-\delta,\delta)}(q')$. Therefore we obtain the claim in these two cases by comparing the the definitions of $\torsion^s(q,\delta)$ and $\torsion^s(q',\delta)$ through such correspondences. 
In the case (ii), the claim is obvious because the distance between $f^t(q)$ and $f^t(q')$ decreases exponentially as $t\to \infty$.
 \end{proof}

\section{Proof of Theorem \ref{th:od}}\label{sec:pr_od}
In this  section, we prove Theorem \ref{th:od}. We consider an arbitrary $C^\infty$ flow $f^t\in \fF^\infty_A$  and deform its $s$-templates, perturbing the flow  by \emph{time-changes}\footnote{After time changes, the flow will no longer preserve the Riemann volume $\vol$ though it will preserve a smooth volume close to $\vol$. 
We will resolve this problem by using a result of Moser \cite{MR0182927} on deformation of smooth volumes on a manifold. See the last part of Subsection \ref{ss:prob_func}.}. 
The problem is how often the condition \eqref{eq:tv} in the non-integrability condition $(NI)_{\rho}$ holds when $b$ is large. Roughly speaking, the point of the argument below is that, for some $R>0$, we can  change the values of $s$-templates $\psi^s_{p,(-1,1)}$ on each of disjoint subintervals in $(-1,1)$ with size $b^{-1/R}$ almost independently with amplitude proportional to $b^{-1}$. This together with the large deviation argument enable us to show that the condition \eqref{eq:tv} with large $b$ is violated only with very small possibility bounded by a  stretched exponential rate in $b$.   

\subsection{Exceptional set $\mathcal{E}(b)$}\label{ss:Exc}
In order to avoid technical problems caused by interference of perturbations, we will regard the $s$-templates for points in some subset $\mathcal{E}(b)$ as exceptions and treat them in a different manner.
The definition of the exceptional set $\mathcal{E}(b)$ in the next paragraph is motivated as follows. 
In order to modify the values of the $s$-template $\psi^s_{p,(-1,1)}$ for $p\in M$, we will perturb the flow $f^{t}$ in a small neighborhood $\mathcal{N}$ of the subset  $f^{t_*}(W^u_{(-1,1)}(p))$ with some $t_*>0$. 
In such perturbation, the problem of interference arises when $f^t(\mathcal{N})\cap \mathcal{N}\neq \emptyset$ for some $t>0$ that is \emph{not} large enough. 
(As we will see, the interference is negligible if $t>0$ is  large enough.) This is the situation that we would like to avoid. Note that, in such situation, there exists a periodic orbit with period $\lesssim t$ that passes through a neighborhood of $\mathcal{N}$, by the pseudo-orbit tracing property of Anosov flow (or Anosov closing lemma \cite[Corollary 18.1.8]{MR1326374}). 

Below we give the precise definition of the exceptional set $\mathcal{E}(b)$. We denote the prime period of a periodic point $w\in M$ by  $\per(w)$ and take  and fix a constant $\tau_*>0$ such that 
\begin{equation}\label{eq:tau_ast}
10 \tau_* <  \min\{\per(w)\mid \text{$w$ is a periodic point of $f^t$}.\}
\end{equation}
Then we set 
\[
\lambda_*:=e^{10\chi_* \tau_*}>1\quad\text{and}\quad
c_*=\frac{2}{1-\lambda_*^{-1}}.
\]
so that the modulus of hyperbolicity 
$\|Df^{\per(w)}_w|_{E_u}\|$ of a periodic point $w\in M$ is always greater than $\lambda_*$. Let $R>0$ be an integer constant that we will take soon in the next subsection. (One may suppose $R=r+11$ if one like.) For $b>1$ and $p\in M$, we set
\begin{equation}\label{eq:T}
T(p,b)=\inf\{ t\ge 0\mid \|Df^{-t}_p|_{E_u}\|\le b^{-1/(4R)}\}.
\end{equation}
\begin{Definition}[The exceptional set $\mathcal{E}(b)$]
The exceptional set $\mathcal{E}(b)\subset M$ for $b>1$ is the open set of points $p\in M$ such that there exits a periodic orbit $\gamma$ whose prime period is less than $T(p,b)$ and whose (minimum) distance from $W^u_{[-c_*,c_*]}(p)$ is less than $b^{-2/(3R)}$.
\end{Definition}
The next lemma tells that the periodic orbit $\gamma$ in the definition above is unique for each $p\in \mathcal{E}(b)$, provided that $b$ is sufficiently large.
\begin{Lemma}\label{lm:potp}
For any $C>1$, there exists $b_0>0$ such that, 
for $b\ge b_0$ and $p\in M$, there exists at most one periodic orbit $\gamma$ satisfying both of the following conditions:
\begin{enumerate}
\item the prime period of $\gamma$ is bounded by $T(p,b)+C$, and
\item the (minimum) distance from $\gamma$ to $W^u_{[-c_*,c_*]}(p)$ is less than $b^{-3/(5R)}$.
\end{enumerate} 
Further, if such periodic orbit $\gamma$ exists, it passes through the $b^{-3/(5R)}$-neighborhood of $W^u_{[-c_*,c_*]}(p)$ only once. More precisely, if there are two points $p_1,p_2$ on $\gamma$ that belongs to the $b^{-3/(5R)}$-neighborhood of $W^u_{[-c_*,c_*]}(p)$, there exists $t$ with $|\tau|<b^{-1/(3R)}$ such that $f^t(p_1)=p_2$. 
\end{Lemma} 
\begin{proof} 
Suppose that a periodic orbit $\gamma$ satisfies the conditions (1) and (2) in the lemma for some $b\ge b_0$ and $p\in M$. Let $p_1$ be a point on the periodic orbit $\gamma$ that belongs to the $b^{-3/(5R)}$-neighborhood of $W^u_{[-c_*,c_*]}(p)$. 
By a crude estimate we see that  the backward orbit $f^{-t}(p)$ of the point $p$ traces that of $p_1$ as far as 
\[
\|Df^{-t}_p\|\sim \|Df_p^{-t}|_{E_u}\|^{-1}\ll \textcolor{\revisionColor}{b^{3/(5R)}}.
\]
Hence, for any $\varepsilon>0$, we can take $t(\varepsilon)>0$ and let $b_0$ large so that the distance between $f^{-t}(p)$ and $f^{-t}(p_1)\in \gamma$ is bounded by $\varepsilon$ for $t\in [t(\varepsilon), t(\varepsilon)+T(p,b)+C]$. That is, the orbit $f^{-t}(p)$ for $t\in [t(\varepsilon), t(\varepsilon)+|\gamma|]$ traces the periodic orbit $|\gamma|$ within distance $\varepsilon$. 
By uniform hyperbolicity of the flow $f^t$,  this implies uniqueness of the periodic orbit $\gamma$ and also the last claim of the lemma because $b^{-1/(3R)}\gg b^{-3/(5R)}$.  
\end{proof}

\subsection{A probability measure on the space of functions} \label{ss:prob_func}
Let $3\le r <\infty$ and let $C^r(M)$ be the Banach space of $C^r$ functions. Let us consider the translations on $C^r(M)$:
\[
\tau_\varphi:C^r(M)\to C^r(M), \qquad 
\tau_\varphi(u)=u+\varphi
\]
In the following, we fix some $R>r$ and a Borel probability measure $\boldsymbol{\mu}$ on $C^r(M)$ such that
 \begin{equation}\label{eq:trans_inv}
\exp(-\|\varphi\|_{C^R})\le \left|\frac{d((\tau_\varphi)_*\boldsymbol{\mu})}{d\boldsymbol{\mu}}\right|\le \exp(\|\varphi\|_{C^R})\qquad\text{for any }\varphi\in C^R(M).
\end{equation}
From \cite[Lemma E]{MR1167374}, such a measure $\boldsymbol{\mu}$ exists provided that we let $R> r$ be sufficiently large, say $R=r+11$. Note that, from \eqref{eq:trans_inv}, we have $\boldsymbol{\mu}(U)>0$ for any non-empty open subset $U\subset C^r(M)$. 

We suppose that $\mathcal{W}$ is a small neighborhood of the origin $0$ in $C^r(M)$ and will let it be smaller in the course of the argument if necessary. Let $v$ be the generating vector field of the flow $f^t$. 
For $\varphi\in \mathcal{W}$, let $f^t_{\varphi}$  be the flow generated by the vector field $v_{\varphi}=(1+\varphi)\cdot v$. 
Note that, since the flow $f^t_{\varphi}$ preserves the $C^r$ volume $m_{\varphi}=(1+\varphi)^{-1}\cdot \vol$, we can apply the argument in Section \ref{sec:ni} to the flow $f^t_{\varphi}$ with setting $\mu=m_{\varphi}$. 
For $0<\rho<1$, $p\in M$, $\alpha\in \real$ and $b>0$, let 
\[
X_\rho(p,\alpha;b)\subset \mathcal{W}
\]
be the set of functions $\varphi\in \mathcal{W}$ such that the condition \eqref{eq:tv} with these $\rho$, $\alpha$ and $b$ \emph{fails} for the \hbox{$s$-template} at $p$ for  the flow $f^t_{\varphi}$. As the main step in the proof of  Theorem \ref{th:od}, we show the following proposition. 
\begin{Proposition}\label{Prop:mesure}
If $\rho>0$ is sufficiently small, then, for sufficiently large $b>0$, we have
\[
\boldsymbol{\mu}(X_\rho(p,\alpha;b))< \exp(-b^{\,\rho})
\]
for any  $p\in M\setminus \mathcal{E}(b)$ and $\alpha\in \real$. 
\end{Proposition}
The proof of this proposition will be given in the following subsections. Below we deduce Theorem  \ref{th:od} from this proposition. Note that we have $X_{\rho}(p,\alpha;b)=\emptyset$ for $\alpha$ with  $|\alpha|\ge b$ from Remark \ref{Rem:largealpha}, provided that $b$ is sufficiently large. 
Take small $\rho>0$ so that the conclusion of  Proposition \ref{Prop:mesure} holds  and let  $\rho'$ be a real number such that $0<\rho'<\rho$. 
\begin{Corollary}\label{Cor:ap}
For sufficiently large $b>0$, we have
\[
\boldsymbol{\mu}\left(\bigcup_{p\in M\setminus \mathcal{E}(b)} \bigcup_{\alpha\in \real} X_{\rho'}(p,\alpha;b)\right)< \exp(-b^{\,\rho}/2).
\]
\end{Corollary} 
\begin{proof}%[Proof of Corollary \ref{Cor:ap}]
If we take a finite but sufficiently dense subset of points $\{(p_i,\alpha_i)\}_{i=1}^I$ in 
\[
(M\setminus \mathcal{E}(b))\times \{\alpha\in \real\mid |\alpha|<b\}
\]
depending on $b$, then, by approximation, the union of the subsets $X_{\rho}(p_i, \alpha_i;b)$ will cover $\bigcup_{p\in M\setminus \mathcal{E}(b)} \bigcup_{\alpha\in \real} X_{\rho'}(p,\alpha;b)$. By crude estimate, we can see that the cardinality $I$ of the finite set necessary for this to be true is bounded by a polynomial order in $b$. Therefore we obtain the conclusion from Proposition \ref{Prop:mesure}.
\end{proof}

Next we prove the following lemma which tells roughly that if the condition \eqref{eq:tv} holds for all $p\notin \mathcal{E}(b)$, it also holds for $p\in \mathcal{E}(b)$ with slightly smaller $\rho$. Let us say that a flow $f^t_\varphi$ satisfies the condition $(NI)_{\rho,b}$ for $b>0$ and $\rho>0$ if the condition \eqref{eq:tv} with these $b$ and $\rho$ holds for all the \hbox{$s$-templates} (for $f^t_{\varphi}$) and $\alpha\in \real$.  Let  $\rho''$ be a real number such that  \[
0<\rho''<\rho'(1-\rho')<\rho'.
\]
\begin{Lemma}\label{Lm:ap}
If $b>0$ is sufficiently large and if $\varphi\in \mathcal{W}$ does not belong to the subset $\bigcup_{p\in M\setminus \mathcal{E}(b')} \bigcup_{\alpha\in \real} X_{\rho'}(p,\alpha;b')$ for any integer $b'$ with $ b^{1-\rho'}\le b'\le \lceil b\rceil$, then the flow $f^t_{\varphi}$ satisfies the condition $(NI)_{\rho'',b}$.  
\end{Lemma}
\begin{proof}
We henceforth write $W^u_{J}(p;\varphi)$, $w^u_p(\tau;\varphi)$ and $\psi^s_{J}(\tau;\varphi)$ respectively for $W^u_{J}(p)$, $w^u_p(\tau)$ and $\psi^s_{J}(\tau)$ defined for the flow $f^t_{\varphi}$, with setting $f^t=f^t_{\varphi}$ and $\mu=m_{\varphi}$ in the argument in Section \ref{sec:ni}. Note that, since $f^t_{\varphi}$ is a time change of the flow $f$, the projection of $W^u_{J}(p;\varphi)$ along the flow line (to some transversal section) coincides with that of $W^u_{J}(p)=W^u_{J}(p;0)$. Below we always suppose that $b$ is sufficiently large. 

From the assumptions, the condition \eqref{eq:tv} with $\rho$ replaced by $\rho'$ holds for the $s$-templates at any point in $ M\setminus \mathcal{E}(b)$ and for any $\alpha$. It is therefore enough to prove the condition \eqref{eq:tv} with $\rho$ replaced by $\rho''$ for the $s$-template at $p\in \mathcal{E}(b)$ and $\alpha\in \real$ with $|\alpha|<b$. For the proof, we will use the following simple relation that follows from \eqref{eq:psis}:  for $q\in M$ and $0<\delta<1$, let $t>0$ be such that $\|(Df_\varphi^t)_q|_{E_u}\|=\delta^{-1}$,  then  
\begin{equation}\label{eq:obs}
\frac{1}{2\delta}\left|\int^{\delta}_{-\delta} \exp(ib(\psi^s_{q,(-\delta,\delta)}(\tau;\varphi)+\alpha \tau)) d\tau\right|
= \frac{1}{2}\left|\int^{1}_{-1} \exp(i\delta b(\psi^s_{q',(-1,1)}(\tau;\varphi)+\alpha' \tau)) d\tau\right|
\end{equation}
for $q'=f^t_\varphi(q)$ and some $\alpha'\in \real$ depending on $\alpha$. 

Suppose that $p\in \mathcal{E}(b)$. 
From the definition of the set $\mathcal{E}(b)$, there exists a periodic orbit $\gamma$ for $f^t$ whose prime period is less than $T(p,b)$ and whose distance from $W^u_{[-c_*,c_*]}(p)$ is less than $b^{-2/(3R)}$.
Note that $\gamma$ is a periodic orbit also for the flow $f^t_\varphi$ and the distance between $\gamma$ and $W^u_{[-c_*,c_*]}(p;\varphi)$ is bounded by $C_*b^{-2/(3R)}$ provided that $\mathcal{W}$ is sufficiently small. 

%\begin{figure}
%\begin{overpic}[scale=0.6]{gamma.eps}
%\put(29,45){$p$}
%\put(75,47){Intersection with $\gamma$}
%\put(7,42){$W^u_{[-c_*,c_*]}(p;\varphi)$}
%\put(41,8){$W_0$}
%\put(10, 5){$W^u_{[-1,1]}(p;\varphi)$}
%\put(83,12){$q$}
%\end{overpic}
%\caption{A picture of the flow $f^t_\varphi$ in a section transversal to the flow that contains the unstable manifold $W^u_{[-c_*,c_*]}(p;\varphi)$. The shaded area and the dashed arrow indicate the intersection with $U$ and the move of points by the return map to the section. There are a few different cases for the relative position of $W^u_{[-1,1]}(p;\varphi)$ to the point $q$.  }
%\label{fig1}
%\end{figure}

We will estimate the integral on right-hand side of \eqref{eq:tv} by
estimating its restrictions to subintervals using \eqref{eq:obs}. To this end, we divide  $W^u_{[-1,1]}(p;\varphi)$ into finitely many pieces
\[
W_k=W^u_{[-\delta(k),\delta(k)]}(q(k);\varphi)\quad\text{for $0\le k\le k(p)$}
\]
with choosing $q(k)\in W^u_{[-1,1]}(p;\varphi)$ and $\delta(k)>0$ appropriately. 
(We allow the pieces $W_k$ to meet each other only at their end points.) Let us write $q=w^u_p(\tau_\gamma;\varphi)$ for the point in $W^u_{[-c_*,c_*]}(p;\varphi)$ that is closest to the periodic orbit $\gamma$.
We take $W_k$, $1\le k\le k(p)$, in the following manner: 
\begin{itemize}
\item 
 $W_0$ is taken as an exceptional piece that covers the intersection of $W^u_{[-1,1]}(p;\varphi)$ with a neighborhood of $q$ of size proportional to  $b^{-\rho'}$. More precisely, we take $q(0)$ and $\delta(0)$ so that $0\le \delta(0)\le 4c_*b^{-\rho'}$ and  that $W_0$ covers the intersection 
\[
w^u_{p}([\tau_\gamma-2c_* b^{-\rho'}, \tau_\gamma+2c_* b^{-\rho'}];\varphi)\cap W^u_{[-1,1]}(p;\varphi).
\]
\item 
For $1\le  k\le k(p)$, the length of $W_k$ is proportional to its distance from the point $q=w^u_p(\tau_\gamma;\varphi)$ in $W^u(p;\varphi)$. More precisely, we take  $q(k)=w^u_p(\tau_k;\varphi)$ and $\delta(k)$ so that 
\[
C_*^{-1}|\tau_k-\tau_\gamma|\le \delta(k)\le c_*^{-1}|\tau_k-\tau_\gamma|/2\quad\text{and}\quad b^{-\rho'}\le \delta(k)\le (c_*-1)/c_*.
\] 
\end{itemize}
Clearly such construction is possible. 
Further we may and do suppose that $\delta(k)b$ for $k\neq 0$ are integers, by adjusting the non-exceptional piece $W_k$ with $k\neq 0$ and incorporating the remnants in the exceptional piece $W_0$.

Below we are going to consider the non-exceptional piece $W_k$ with
 $1\le k\le k(p)$.  Take $t_k>0$ so that $\|(Df^{t_k}_\varphi)_{q(k)}|_{E_u}\|= \delta(k)^{-1}$ or, in other words,  that  
\[
 f_\varphi^{t_k}(W^u_{[-\delta(k),\delta(k)]}(q(k);\varphi))=W^u_{[-1,1]}(f^{t_k}_{\varphi}(q(k));\varphi).
 \]
We claim that $f^{t_k}_{\varphi}(q(k))$ for $1\le  k\le k(p)$ does not belong to $\mathcal{E}(\delta(k)b)$. Set $q'(k):=f^{t_k}_{\varphi}(q(k))$ for brevity. To prove the claim, let us suppose that $q'(k)$  belongs to $\mathcal{E}(\delta(k)b)$ and show that the periodic orbit $\gamma$ is too close to $W_k$, which contradicts the fact that $W_k$ is a non-exceptional piece. 
Note, first of all, that we may suppose the ratio $t_k/T(p,b)$ to be  close to $0$ because $\rho'<\rho$ are assumed to be small. From the definition of $\mathcal{E}(\delta(k)b)$, there exists a periodic orbit $\gamma'$ for the flow $f^t$ whose period $|\gamma'|$ is  less than $T(q'(k),\delta(k)b)$ and whose distance from 
$W^u_{[-c_*,c_*]}(q'(k))$ is bounded by $(\delta(k)b)^{-2/(3R)}$. 
The backward orbit $f^{-t}_{\varphi}(q'(k))$ of $q'(k)$ traces the periodic orbit $\gamma'$ as far as $\|Df^{-t}_{q'(k)}\|\sim \|Df^{-t}_{q'(k)}|_{E_u}\|^{-1}\ll (\delta(k)b)^{2/(3R)}$ and hence for $0\le t\le T(q'(k),\delta(k)b)$. Therefore 
we have that 
\begin{equation}\label{eq:periodqdd}
\|Df^{-|\gamma'|}_{q''}|_{E_u}\|>C_*^{-1}\|Df^{-T(q'(k),\delta(k)b)}_{q'(k)}|_{E_u}\|\ge C_*^{-1}(\delta(k) b)^{-1/(4R)}\gg  b^{-1/(4R)}
\end{equation}
for any point $q''$ on the orbit $\gamma'$. Since $f^{-t_k}_{\varphi}$ preserves the periodic orbit $\gamma'$,  
the distance between $\gamma'$ and $W_k=W^u_{[-c_*\delta(k), c_*\delta(k)]}(q(k);\varphi)\subset W^u_{[-c_*,c_*]}(p;\varphi)$ is less than 
\[
C_* \delta(k)^{-1}\cdot  (\delta(k)b)^{-2/(3R)}\ll \textcolor{\revisionColor}{b^{-1/(2R)}}.
\] 
This and the fact that $q(k)\in W^u_{[-1,1]}(p,\varphi)$ implies that the backward orbit $f^{-t}(p)$ of $p$ also traces $\gamma'$ as far as $\|Df^{-t}_p\|\sim \|Df^{-t}_p|_{E_u}\|^{-1}\ll \textcolor{\revisionColor}{b^{1/(2R)}}$. Hence, from \eqref{eq:periodqdd},  
the period $|\gamma'|$ of $\gamma'$ is bounded by $T(p,b)+C$.
This and Lemma \ref{lm:potp} imply $\gamma'=\gamma$. But the estimate above on the distance between $\gamma'=\gamma$ and $W_k$ clearly contradicts the fact that $W_k$ is a non-exceptional piece.   

Now we can apply \eqref{eq:obs} to the non-exceptional pieces $W_k$ for $1\le k\le k(p)$ and find
\[
\frac{1}{|I_k|}\left|\int_{I_k} \exp(ib(\psi^s_{p,(-1,1)}(\tau;\varphi)+\alpha \tau)) d\tau\right|
= \frac{1}{2}\left|\int^{1}_{-1} \exp(i\delta(k) b(\psi^s_{q'(k),(-1,1)}(\tau;\varphi)+\alpha' \tau)) d\tau\right|
\]
where $I_k\subset [-1,1]$ is the interval such that $W_k=W^u_{I_k}(p;\varphi)$.
Since $q'(k)=f^{t_k}_{\varphi}(q(k))$ for $1\le  k\le k(p)$ does not belong to $\mathcal{E}(\delta(k)b)$ as we have proved above, the assumption of the lemma tells that the right-hand side is bounded by $(\delta(k)b)^{-\rho'}\ll b^{-\rho''}$. Since the length of the exceptional piece $W_0$ is bounded by $C_* b^{-\rho'}\ll b^{-\rho''}$,  we obtain the required estimate \eqref{eq:tv} with $\rho$ replaced by $\rho''$. 
\end{proof}

From  Lemma \ref{Lm:ap}, we see that, if $\varphi\in \mathcal{W}$ does not belong to the subset
\[
\bigcap_{B} \bigcup_{b\ge B, b\in \mathbb{N}}\left(\bigcup_{p\in M\setminus \mathcal{E}(b)} \bigcup_{\alpha\in \real} X_{\rho}(p,\alpha;b)\right),
\]
the flow $f^t_{\varphi}$ satisfies the condition $(NI)_{\rho'}$ for $\rho'<\rho(1-\rho)$. Since the $\boldsymbol{\mu}$-measure of the set above is $0$ from Corollary \ref{Cor:ap} and Borel-Cantelli lemma, we can find $\varphi\in C^r(M)$ arbitrarily close to $0$  such that the flow $f^t_{\varphi}$ satisfies the non-integrability condition $(NI)_{\rho/2}$. By a theorem of Moser \cite{MR0182927}, there is a $C^r$ diffeomorphism $\Phi_{\varphi}:M\to M$ which transfers the volume $m_\varphi=(1+\varphi)^{-1}\vol$ to $\vol$, and $\Phi_{\varphi}$ converges to the identity in $C^r$ sense as $\varphi$ converges to $0$ in $C^r(M)$. Therefore, taking conjugation of $f^t_{\varphi}$ by such a diffeomorphism $\Phi_{\varphi}$ and recalling Remark \ref{Rem:indep_NI_metric}, we obtain a $C^r$ flow in $\fF^r_{A}$ which is arbitrarily close to $f^t$ in the $C^r$ sense and satisfies the non-integrability condition $(NI)_{\rho/2}$. We have finished the proof of Theorem~\ref{th:od}.

\subsection{Perturbation family}\label{ss:perturbation_family}
In this subsection, we explain the scheme of perturbation for the proof of Proposition \ref{Prop:mesure}. Recall that we write $v$ for the generating vector field of the flow $f^t$. Suppose that $b>0$ is large and that a point $p\in M\setminus \mathcal{E}(b)$ and $\alpha\in \real$ with $|\alpha|< b$ are given arbitrarily. 
Below we set up  functions $\varphi_j\in C^\infty(M)$ for $1\le j\le \lceil b^{1/R}\rceil$. Then,  for arbitrary $\varphi_0\in \mathcal{W}$ and a subset 
$\cJ\subset \{1,2,\dots,\lceil b^{1/R}\rceil\}$ of integers, we consider the family of vector fields
\begin{equation}\label{eq:vfamily}
v_{\bt}=
(1+\varphi_{\bt})\cdot v\quad
\text{ where }\quad \varphi_{\bt}=\varphi_0+
\sum_{j\in \cJ}t_{j}\cdot \varphi_{j}\quad \text{and}\quad
\bt=(t_{j})\in [-4,4]^{\cJ}.
\end{equation}
Once we fix such family of vector fields, we write  $f^t_{\bt}=f^t_{\varphi_{\bt}}$ for the flow generated by $v_\bt$. 

Recall the constant $\tau_*>0$ taken so that the condition \eqref{eq:tau_ast} holds and set $q=f^{4\tau_*}(p)$. 
We take the points 
\[
s(j)=-1+(2j-1)\cdot b^{-1/R}\quad \text{for $1\le j \le \lceil b^{1/R}\rceil$}
\]
so that 
\[
s(1)=-1+b^{-1/R}, \quad s(i+1)-s(i)=2b^{-1/R}\quad\text{ and }\quad s(\lceil b^{1/R}\rceil)\in [1,1+b^{-1/R}].
\] Then we put, for $1\le j \le \lceil b^{1/R}\rceil$, 
\[
p(j)=w^u_{p}(s(j)),\qquad  
q(j)=f^{4\tau_*}(p(j))=w^u_q(\lambda s(j))
\]
where we set  $\lambda=\|Df^{4\tau_*}_p|_{E_u}\|$. (See Figure \ref{fig2}.) 

To proceed, we take a $C^\infty$ local coordinate chart
\[
\kappa_{p}:U\to V\times [-\tau_*,6\tau_*]\subset \real^2\times \real,\quad \kappa_p(m)=(x,y,z)
\]
on a neighborhood $U$ of $W^u_{[-2,2]}(p)$, so that
\begin{itemize}
\item it is a flow box coordinate charts for the flow $f^t$, that is,  $(\kappa_p)_*(v)=\partial_z$,
\item it transfers the Riemann volume $m$ to the standard volume $dxdydz$ on $\real^3$, and 
\item $\kappa_p(w^u_{p}(\tau))=(\tau,0,0)$ if $\tau\in [-2,2]$.
\end{itemize}
We may and do assume that the coordinate charts $\kappa_p$ for $p\in M$ are bounded uniformly in $C^\infty$ sense. (Recall that we are assuming $f\in \fF^\infty_A$.) Note that we have 
\[
\kappa_{p}(p(j))=(s(j),0,0),\quad 
\kappa_{p}(q(j))=\kappa_{p}(f^{4\tau_*}(p(j)))=(s(j),0,4\tau_*).
\]

Next we take and fix a $C^\infty$ functions $h_0:\real^2\to [0,1]$ supported on the disk $|(x,y)|\le 3/2$ such that $
h_0(x,y)=y$ if $|(x,y)|\le 1$. Also let  $\chi:\real\to [0,1]$ be a $C^\infty$ function such that $\chi(s)=1$ if $|s|\le 1$ and $\chi(s)=0$ if $|s|\ge 3/2$. 
Then we define  
\[
\varphi_j:M\to [0,1]\quad\text{ for $1\le j\le \lceil b^{1/R}\rceil$}
\]
by 
\begin{equation}\label{eq:def_varphi_j}
\varphi_j(m)=- b^{-1-1/R}\cdot \chi\big((z-4\tau_*)/\tau_*\big)\cdot h_0\big(b^{1/R}(x-s(j)),b^{1/R}y\big)
\end{equation}
if $m\in U$, where $\kappa_p(m)=(x,y,z)$, and set $\varphi_j(m)=0$ if $m\notin U$. 
\begin{Remark}\label{Rem:adj}
The motivation for the choice of $\varphi_j$ above is explained as follows.
Consider the family $f^t_{\bt}$ defined by \eqref{eq:vfamily} with  $\cJ=\{j\}$ and $\varphi_0\equiv 0$. 
Then, in the local chart $\kappa_p$, the vector field $v_{\bt}$ looks
\[
(\kappa_{p})_*v_{\bt}(x,y,z)=
\bigg(1-t_j \cdot b^{-1-1/R} \cdot \chi((z-4\tau_*)/\tau_*) \cdot h_0\big(b^{1/R}(x-s(j)),b^{1/R}y\big)\bigg)\cdot \partial_z.
\]
A simple computation tells that, 
if $w=(x,y)$ satisfies $|w-(s(j),0)|\le b^{-1/R}$, the map $f^{-6\tau_*}_{\bt}$  takes the point $(w, 6\tau_*)$  to $
(w,  a_0 b^{-1} t_j  y +\mathcal{O}_*(|b^{-1}t_j y|^2))$ where $a_0=\int \chi(z/\tau_*) dz$. Therefore, by varying the parameter $t_j$, we will be able to rotate the stable subspace $E_s(w)$ around the unstable manifold by the rate proportional to $b^{-1}$.   
\end{Remark}

\begin{figure}
\begin{overpic}[scale=0.5]
{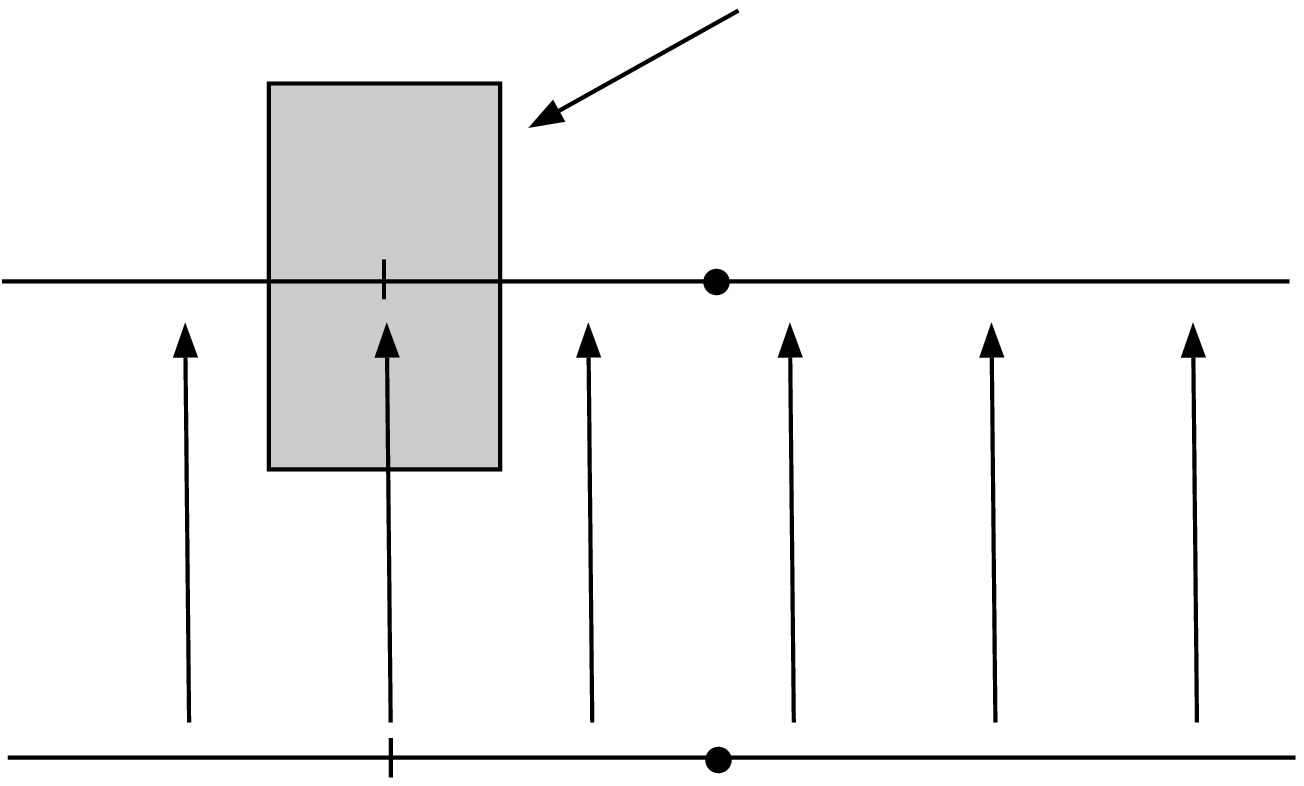}
\put(58,58){The support of $\varphi_j$}
\put(54,41){$q$}
\put(21,41){$q(j)$}
\put(20,4){$p(j)$}
\put(54,4){$p$}
\put(-10,6){$W^u_{(-1,1)}(p)$}
\end{overpic}
\caption{A picture of the flow $f^t$ in a section parallel to the flow that contains the unstable manifold $W^u_{(-1,1)}(p)$. }
\label{fig2}
\end{figure}

Below we discuss about the perturbation family $f^t_{\bt}=f^t_{\varphi_{\bt}}$ defined in \eqref{eq:vfamily}. First of all, note that the functions $\varphi_j$ satisfy 
\begin{equation}\label{eq:norm}
\|D^k\varphi_j\|_{\infty}\le C_* b^{((k-1)/R)-1}\quad\text{ for $0\le k\le R$.}
\end{equation}
The intersection multiplicity of  $\supp \varphi_{j}$ for $1\le j\le \lceil b^{1/R}\rceil$ is bounded by $2$. Hence, regardless of the choice of $\cJ$, we have that, for $\bt\in [-4,4]^{\cJ}$, 
\begin{equation}\label{eq:basic_est_family}
\|\varphi_{\bt}-\varphi_{\mathbf{0}}\|_{C^R}<C_* b^{-1/R},\;\; 
\|\partial_{\bt}\varphi_{\bt}\|_{C^1}\le C_* b^{-1}\;\;\text{and}\;\;
\|\varphi_{\bt}-\varphi_{\mathbf{0}}\|_{C^0}<C_* b^{-1-1/R}.
\end{equation}
As we explained in Remark \ref{Rem:adj}, 
our aim is to modify the $s$-template $\psi^s_{p,(-1,1)}$ on the intervals 
\[
J(j):=[s(j)-b^{-1/R}, s(j)+b^{-1/R}]\cap [-1,1]\quad \text{ for $j\in \cJ$}
\]
almost independently
by varying the parameters $t_j$ for $j\in \cJ$.  
To realize this, we have to choose the subset $\cJ$ carefully. 
Let us make two observations: The first one is that
\begin{itemize}
\item[(Ob1)] $
f^{t}(\supp \varphi_i) \cap \supp \varphi_j= \emptyset$ for $1\le i,j\le \lceil b^{1/R}\rceil$ and $4\tau_*\le t\le T(p,b)-4\tau_*$.
\end{itemize}
This is a consequence of the fact that $p$ does not belong to $\mathcal{E}(b)$.
Indeed, if this were not true, we could find a periodic orbit with period $< T(p,b)$ in the $C_* b^{-1/R}$-neighborhood of $W^u_{[-c_*,c_*]}(p)$ by the pseudo-orbit tracing property and hence the point $p$ would belong to $\mathcal{E}(b)$, provided that $b$ is  sufficiently large. 

The second observation below is somewhat similar to the first one but makes use of the fact that the curve $f^{-t}(W^u_{[-1,1]}(p))$ shrinks exponentially fast as $t$ increases:
\begin{itemize}
\item[(Ob2)] 
Let $\mathfrak{E}(p,b)$ be the set integers $1\le j\le \lceil b^{1/R}\rceil$ such that the support of $\varphi_j$ meets $
f^{-t}(W^u_{[-1,1]}(p))$  for some $t\ge 0$ with $\|Df^{-t}_p|_{E_u}\|\ge b^{-1}$. Then $\sharp \mathfrak{E}(p,b)\le C_* b^{3/(4R)}$.
\end{itemize}
Indeed, by the similar reason as that for (Ob1),  we can find a sequence of real numbers $t_0=0<t_1<t_2<\dots<t_k<\cdots$ such that the curve  $f^{-t}(W^u_{(-1,1)}(p))$ meets $\supp \varphi_j$ for some $1\le j\le \lceil b^{1/R}\rceil$ only if  $t\in \cup_{i=1}^\infty [t_i-2\tau_*, t_i+2\tau_*]$ and that  $t_{i+1}-t_i> T(p,b)-4\tau_*$ for $i\ge 0$, provided that $b$ is sufficiently small. 
Then (Ob2) follows since the length of $f^{-t_i}(W^u_{[-1,1]}(p))$ is bounded by $(C_* b^{-1/(4R)})^i$ from the definition of $T(p,b)$. 

In the following,  we let the subset $\cJ$ be either of $\cJ_{\mathrm{even}}$ and $\cJ_{\mathrm{odd}}$ that consist of even and odd integers $1\le j\le \lceil b^{1/R}\rceil$ respectively, \emph{but we exclude}
\begin{itemize}
\item those $j$ in $\mathfrak{E}(p,b)$ given in (Ob2), and
\item those $j$ for which $J(j-1)\cup J(j) \cup J(j+ 1)$ contains either of $-1$ or $1$.
\end{itemize}
\begin{Remark}\label{rem:excp_j}
The cardinality of $1\le j\le \lceil b^{1/R}\rceil$ that does not belong to $\cJ_{\mathrm{even}}\cup\cJ_{\mathrm{odd}}$ is bounded by $C_*b^{3/(4R)}$ from (Ob2) and hence the Lebesgue measure of the union of $J(j)$ for such $j$'s is bounded by $C_* b^{-1/(4R)}$. With this reason, these exceptions will turn out to be negligible when we prove the estimate \eqref{eq:tv} in the following subsections, provided that $\rho$ is so small that $\rho<1/(4R)$.
\end{Remark}

\subsection{Deformation of $s$-templates}
In this subsection, we suppose that $\cJ$ is either of the subsets $\cJ_{\mathrm{even}}$ and $\cJ_{\mathrm{odd}}$ defined above and observe how the $s$-template $\psi^s_{p,(-1,1)}$ at the point $p$ is deformed in the
perturbation family \eqref{eq:vfamily}  with arbitrary $\varphi_0\in \mathcal{W}$. 

Let us write $
T_qM=E_0(q;\bt)\oplus E_s(q;\bt)\oplus E_u(q;\bt)$
for the hyperbolic decomposition at $q\in M$ for the flow $f^t_\bt$, correspondingly to 
\eqref{eq:hyperbolic_splitting}. Let 
\[
\gamma^0_{p,\bt}, \;\gamma^\perp_{p,\bt}, \;\gamma^s_{p,\bt}\; :W^u_{(-1,1)}(p;\varphi_{\bt})\to T^*M
\]
be the sections $\gamma^0_{p,(-1,1)}, \;\gamma^\perp_{p,(-1,1)}, \;\gamma^s_{p,(-1,1)}$ considered in Section \ref{sec:ni} but defined for the perturbed flow  $f^t_{\bt}$. 
Then, by definition, the $s$-template $\psi^s_{p, (-1,1)}(\tau;\varphi_{\bt})$ for the flow $f^t_{\bt}$ at $p$ is the function satisfying 
\begin{equation}\label{eq:expression_psi}
\gamma^s_{p,\bt}(z)=\psi^s_{p, (-1,1)}(\tau;\varphi_{\bt})\cdot \gamma^\perp_{p,\bt}(z) +\gamma^0_{p,\bt}(z)\qquad\text{for }z=w^u_p(\tau;\varphi_{\bt}).
\end{equation}

Actually it is not a very simple task to observe how the $s$-template $\psi^s_{p, (-1,1)}(\cdot;\varphi_{\bt})$ depends on the parameter $\bt$, because the frames $\gamma^\perp_{p,\bt}$ and $\gamma^0_{p,\bt}$ also depend on $\bt$. In order to simplify the argument, we consider an intermediate approximation of $\psi^s_{p, (-1,1)}(\cdot;\varphi_{\bt})$, which is defined in a similar  manner as $\psi^s_{p, (-1,1)}(\tau;\varphi_{\bt})$  but with the fixed frames defined for $\mathbf{t}=\mathbf{0}$. Let us define $\tilde{\gamma}^s_{p,\bt}:W^u_{(-1,1)}(p;\varphi_{\mathbf{0}})\to T^*M$ as
the unique continuous section in $\Gamma_1^u(p,(-1,1))$ for the flow $f^t_{\mathbf{0}}$ (with $\bt=\mathbf{0}$!) such that $\tilde{\gamma}^s_{p,\bt}(z)$ is normal to 
$E_u(z;\mathbf{0})\oplus E_s(z;\bt)$ for each $z\in W^u_{(-1,1)}(p;\varphi_{\mathbf{0}})$. Notice that this is a section \emph{not} on $W^u_{(-1,1)}(p;\varphi_{\bt})$ but on $W^u_{(-1,1)}(p;\varphi_{\mathbf{0}})$. We can express it as 
\begin{equation}\label{eq:expression}
\tilde{\gamma}^s_{p,\bt}(z)=\tilde{\psi}^s_{p,(-1,1)}(\tau;\bt)\cdot \gamma^\perp_{p,\mathbf{0}}(z)+\gamma^0_{p,\mathbf{0}}(z)\quad\text{for $z=w^u_p(\tau;\varphi_{\mathbf{0}})$}
\end{equation}
using a unique continuous function $\tilde{\psi}^s_{p,(-1,1)}(\cdot;\bt):(-1,1)\to \real$. In the next lemma, we show  that  $\tilde{\psi}^s_{p,(-1,1)}(\cdot;\bt)$ is a nice approximation of ${\psi}^s_{p,(-1,1)}(\cdot;\varphi_{\bt})$. Before stating the lemma, we note that there exist unique functions
\[
h:[-4,4]^{\cJ}\to \real\quad\text{and}\quad
H:(-1,1)\times [-4,4]^{\cJ}\to \real
\]
such that $h(\mathbf{0})=1$, $H(\tau,\mathbf{0})=0$ , $H(0,\bt)=0$ and that 
\[
w^u_p(h(\bt)\cdot\tau;\varphi_{\bt})=f^{H(\tau,\bt)}(w^u_p( \tau;\varphi_{\mathbf{0}})) \quad\text{for $\tau\in (-1,1)$ and $\bt\in [-4,4]^{\cJ}$.}
\]
This is a consequence of the definition of the intrinsic metric \eqref{eq:intrinsice_metric} and the fact that our perturbation does not change the flow line. 
\begin{Lemma}\label{lm:approx_psis}
For $\bt\in [-4,4]^{\cJ}$, we have $|h(\bt)-1|\le C_* b^{-2}$ and  
\[
|\tilde{\psi}^s_{p,(-1,1)}(\tau;\bt)-
{\psi}^s_{p,(-1,1)}(h(\bt)\cdot \tau;\varphi_{\bt})|
<C_* b^{-1-(1/(4R))}\quad\text{for $\tau\in(-1,1)$}.
\] 
\end{Lemma}

\begin{proof}
From the observation (Ob2) in the last subsection and the definitions of $\cJ_{\mathrm{odd}}$ and $\cJ_{\mathrm{even}}$, the backward orbit $f^{-t}(W^u_{[-1,1]}(p))$ of $W^u_{[-1,1]}(p)$, which shrinks exponentially fast as $t$ increases,  does not meet the domain of perturbation $\cup_{j\in \cJ} \supp \varphi_j$ until the condition $\|Df^{-t}_p|_{E_u}\|< b^{-1}$ is fulfilled.  Therefore, from the basic estimate \eqref{eq:basic_est_family} and the construction of the unstable manifold by the graph transform method, we see that  the $C^1$ distance between $W^u_{(-1,1)}(p;\varphi_{\bt})$ and $W^u_{(-1,1)}(p;\varphi_{\mathbf{0}})$ is bounded by $C_* b^{-2}$ and consequently
\begin{equation}\label{eq:estmu}
|h(\bt)-1|\le C_* b^{-2}\quad\text{ and } 
\quad 
\|\partial_\tau H(\tau;\bt)\|_{\infty}\le C_* b^{-2}.
\end{equation}
Next we consider the sections $\gamma^s_{p,\bt}(\cdot)$, $\gamma^\perp_{p,\bt}(\cdot)$ and $\gamma^0_{p,\bt}(\cdot)$. For convenience, we look them in the flow box coordinate chart $\kappa_p$ considered in the last subsection and regard them as mappings from $(-1,1)$ to $\real^3$.
From \eqref{eq:estmu}, we have
\begin{equation}\label{eq1}
\gamma^\perp_{p,\bt}(w^u_p(h(\bt)\cdot \tau;\varphi_{\bt}))=\left(1+\mathcal{O}_*(b^{-2})\right)\cdot 
\gamma^\perp_{p,\mathbf{0}}(w^u_p(\tau;\varphi_{\mathbf{0}}))
\end{equation}
and 
\begin{equation}\label{eq2}
\gamma^s_{p,\bt}(w^u_p(h(\bt)\cdot \tau;\varphi_{\bt}))=
\tilde{\gamma}^s_{p,\bt}(w^u_p(\tau;\varphi_{\mathbf{0}}))+\mathcal{O}_*(b^{-2}).
\end{equation}
To compare $\gamma^0_{p,\bt}(\cdot)$ and $\gamma^0_{p,\mathbf{0}}(\cdot)$, 
we recall that they are specified by the condition \eqref{eq:specyfy_gamma0}. For the point $z=w_p^u(\pm 1;\varphi_{\mathbf{0}})$, we see, from the choice of the subsets $\cJ_{\mathrm{odd}}$ and $\cJ_{\mathrm{even}}$ and from the observation (Ob1), that its forward orbit $f^t(z)$ for $0\le t\le T(p, b)-4\tau_*$ does not pass through the domain of perturbation. Hence, from the construction of the stable subspace,   
\[
\|\gamma^s_{p,\bt}(w^u_p(\pm h(\mathbf{t});\varphi_{\bt}))-\gamma^s_{p,\mathbf{0}}(w^u_p(\pm 1;\varphi_{\mathbf{0}}))\|<C_* b^{-1-(1/(4R))}.
\]
This together with \eqref{eq:estmu} and the $C^2$ boundedness of $\gamma^0_{p,\bt}$ imply
\[
\|\gamma^0_{p,\bt}(w^u_p(\pm h(\mathbf{t});\varphi_{\bt}))-\gamma^0_{p,\mathbf{0}}(w^u_p(\pm 1;\varphi_{\mathbf{0}}))\|<C_* b^{-1-(1/(4R))}.
\]
Let us recall, from the proof of Lemma \ref{lm:straight}, how the straight sections are determined as  limits in terms of the time evolution along the orbit of $W^u_{(-1,1)}(p)$ in the \emph{negative} time direction. 
Then, from disjointness of the backward orbit of $W^u_{(-1,1)}(p)$ with the domain of perturbation, mentioned in the beginning of this proof, and also from the estimate above at the end points,  we see that
\begin{equation}\label{eq3}
\|\gamma^0_{p,\bt}(w^u_p(h(\bt)\cdot \tau;\varphi_{\bt}))-
\gamma^0_{p,\mathbf{0}}(w^u_p(\tau;\varphi_{\mathbf{0}}))\|<C_* b^{-1-(1/(4R))}\quad \text{for $\tau\in (-1,1)$}.
\end{equation}
We can now conclude the lemma by comparing    \eqref{eq:expression_psi} and \eqref{eq:expression} and using  \eqref{eq1}, \eqref{eq2}, \eqref{eq3}. 
\end{proof}

From Lemma \ref{lm:approx_psis}, we have 
\[
\left|\int_I \exp\left(ib \left(\tilde{\psi}^s_{p,(-1,1)}(\tau;\bt)+\alpha \tau\right)\right) d\tau
-\int_I \exp\left(ib \left(\psi^s_{p,(-1,1)}(\tau;\varphi_{\bt})+\alpha \tau\right)\right) d\tau\right|<C_* b^{-1/(4R)}
\]
for any Borel subset $I\subset [-1,1]$. 
Therefore, in proving \eqref{eq:tv},  we may consider $\tilde{\psi}^s_{p,(-1,1)}(\tau;\bt)$ in the place of ${\psi}^s_{p,(-1,1)}(\tau;\varphi_{\bt})$, provided that $0<\rho<1/(4R)$. 
 
To proceed, we introduce the mapping
\[
\Psi_{\tau}=\Psi_{\tau,p,\alpha,b}:[-4,4]^{\cJ}\to \real^{\cJ}
\]  
defined for each $\tau\in (-b^{-1/R}, b^{-1/R})$ by 
\[
\Psi_{\tau}(\bt)=b\cdot 
\bigg(\tilde{\psi}_{p,(-1,1)}^s(s(j)+\tau;\textcolor{\revisionColor}{a_0^{-1}}\bt)+\alpha(s(j)+\tau)\bigg)_{j\in \cJ}
\]
where $a_0=\int \chi(z/\tau_*) dz$. (The factor $a_0^{-1}$ before $\bt$ above is thrown-in just for normalization. See Remark \ref{Rem:adj}.)
For $\varphi_0\in \mathcal{W}$ in \eqref{eq:vfamily}, let $A(\varphi_0)$ be the  diagonal matrix of size $\sharp \cJ$ with the diagonal elements 
\[
{a}_j(\varphi_0)=\textcolor{\revisionColor}{a_0^{-1}}\cdot \int (1+\varphi_0(f^{t}(q(j))))^{-2}\cdot \chi(t/\tau_*) \, dt
\] 
where $\chi(\cdot)$ is the function that appeared in the definition \eqref{eq:def_varphi_j} of $\varphi_j$. 
Clearly we have $C_*^{-1}<{a}_j(\varphi)<C_*$ for $\varphi\in \mathcal{W}$ and 
\begin{equation}\label{eq:var_phi}
|{a}_j(\tilde{\varphi})-{a}_j(\varphi)|<C_* \|\tilde{\varphi}-\varphi\|_{\infty} \quad\text{for }\varphi,\tilde{\varphi}\in \mathcal{W}.
\end{equation}
The next lemma realize the idea explained in Remark \ref{Rem:adj}.
\begin{Proposition} \label{prop:Psi}
If $b>0$ is sufficiently large, we have   
\[
\left\|D\Psi_{\tau}(\bt)- A(\varphi_0):\real^{\cJ}\to \real^{\cJ}\,\right\|_{\max}\le C_{*} b^{-1/(4R)}
\]
for any $\bt\in [-4,4]^{\cJ}$, $p\in M\setminus \mathcal{E}(b)$, $\alpha\in \real$ and $\tau\in (-b^{-1/R}, b^{-1/R})$, 
where we consider the operator norm with respect to the  norm $\|(s_j)\|_{\max}=\max_j|s_j|$ on $\real^{\cJ}$. This implies in particular that the mapping $\Psi_\tau$ restricted to $[-4,4]^{\cJ}$ is a diffeomorphism onto its image. 
In addition, for the Jacobian determinant, we have
\begin{equation}\label{eq:logdet}
|\log \det D\Psi_{\tau}(\bt)-\log \det A(\varphi_0)|\le C_*  b^{-1/(4R)}\cdot \sharp \cJ.
\end{equation}
\end{Proposition}
\begin{proof}The stable subspace at  $w\in M$ for  $f^t_{\bt}$ is given as the limit
\[
E_s(w;\bt)=\lim_{t\to \infty} Df^{-t}_{\bt}(E(f^t_{\bt}(w)))
\] 
where $E(w)$ is the one dimensional subspace in $T_wM$  which is orthogonal to 
$E_u(w)\oplus E_0(w)$. Thus we can compute the differentials of $E_s(w;\bt)$ and $\tilde{\psi}^s_{p,(-1,1)}(w;\bt)$ with respect to the parameter $\bt$ as an integral of the infinitesimal contribution of the perturbation at $f^t_{\bt}(w)$ for $t\ge 0$. It is not difficult to see that the differential is given in the form\footnote{Here we use the fact that our perturbation does not change the flow lines. Though we can express $X_j(\tau, p,t,\bt)$ explicitly by preparing some definitions, this is not necessary in the following.
}
\[
\partial_{t_j}\tilde{\psi}^s_{p,(-1,1)}(\tau;\bt) =
\int_0^\infty |Df^t_{\mathbf{0}}|_{E^u}(w_p^u(\tau;\varphi_{\mathbf{0}}))|^{-1}\cdot X_j(\tau, p,t,\bt)\, dt
\]
where $X_j(\tau, p,t,\bt)$ 
satisfies $|X_j(\tau, p,t,\bt)|<C_*b^{-1}$ from \eqref{eq:basic_est_family} and vanishes if $f^t_{\mathbf{0}}(w^u_p(\tau;\varphi_{\mathbf{0}}))$ does not belong to $\supp \varphi_j$. From the construction of $\varphi_{j}$ and Remark \ref{Rem:adj}, we have 
\[
\int_0^{6\tau_*} |Df^t_{\mathbf{0}}|_{E_u}(w_p^u(\tau;\varphi_{\mathbf{0}}))|^{-1} X_j(\tau, p,t,\bt)\, dt=\begin{cases}
b^{-1} \cdot {a}_j(\varphi_0)+\mathcal{O}_*(b^{-2}),&\text{if $\tau\in J(j)$;}\\
0,&\text{if $\tau\in J(i)$ for $i\neq j$.}
\end{cases} 
\]
From the observation (Ob1) in the last subsection, we have
\begin{equation}\label{eq:cons_Ob1}
\int_{6\tau_*}^{\infty} |Df^t_{\mathbf{0}}|_{E^u}(w_p^u(\tau;\varphi_{\mathbf{0}}))|^{-1}\cdot X_j(\tau, p,t,\bt)\,  dt\le C_* b^{-1-(1/(4R))}\quad \text{for all $\tau\in (-1,1)$}.
\end{equation}
Further we can strengthen the last estimate \eqref{eq:cons_Ob1} for most of $j\in \cJ$. Indeed,  from the definition of $\cJ_{\mathrm{odd}}$ and $\cJ_{\mathrm{even}}$ and from the observation (Ob1), we see that
\begin{itemize}
\item each point $w\in M$  belongs to the support of $\varphi_j$ for at most one $j\in \cJ$, 
\item if $w\in \supp \varphi_j$ for some $j\in \cJ$, the orbit 
$f^t_{\mathbf{0}}(w)$ for $t\in [4\tau_*, T(p,b)-4\tau_*]$ does not meet $\bigcup_{i\in \cJ}\supp\varphi_i$, and
\item if $w\in \supp \varphi_j$ and $f^t_{\bt}(w)\in \supp \varphi_{j'}$ for some $j,j'\in \cJ$ and $t \ge  T(p,b)-4\tau_*$, we have $|Df^t_{\mathbf{0}}|_{E_u}(w)|\ge C_*^{-1} b^{1/(4R)}$.
\end{itemize}
Therefore the forward orbit of $w_p^u(\tau;\varphi_{\mathbf{0}})$ passes through the region $\bigcup_{i\in \cJ}\supp\varphi_i$ in some intervals in time and, in between such intervals, the factor $|Df^t_{\mathbf{0}}|_{E^u}(w_p^u(\tau;\varphi_{\mathbf{0}}))|$ grows by a rate greater than $C_*^{-1}b^{1/(4R)}$. 
In particular, within the time when $|Df^t_{\mathbf{0}}|_{E^u}(w_p^u(\tau;\varphi_{\mathbf{0}}))| \le b^2$, the number of such intervals in time is bounded by $4R+1$.
This implies that, for each $\tau\in \cup_{i\in \cJ} J(i)$, we have
\begin{equation}\label{eq:cons_Ob12}
\int_{6\tau_*}^{\infty}  |Df^t_{\mathbf{0}}|_{E^u}(w_p^u(\tau;\varphi_{\mathbf{0}}))|^{-1}\cdot X_j(\tau,p,t,\bt)  dt\le C_* b^{-2}
\end{equation}
\emph{except} for a subset of $j\in \cJ$ (depending on $\tau$) whose cardinality is bounded by $4R+1$. 
This and the estimates above give the first claim of the proposition. 

In order to prove the second claim, let us define the $|\cJ|\times |\cJ|$ matrix $M$ by 
\[
M= D\Psi_{\tau}(\bt) \cdot  A(\varphi_0)^{-1}-\mathbf{1}
\]
where $\mathbf{1}$ at the end denotes the $|\cJ|\times |\cJ|$ identity matrix. 
From \eqref{eq:cons_Ob1}, each component $m_{ij}$ of $M$ is bounded by $C_* b^{-1/(4R)}$ and, for each $i\in \cJ$, the cardinality of $j\in \cJ$ such that $|m_{ij}|\ge C_* b^{-1}$ is bounded by $4R+1$. From these estimates, we can deduce
\[
\|M\|_{\max}\le C_* b^{-1/(4R)}(4R+1)+C_*b^{-1} \cdot \sharp \cJ \le C_* b^{-1/(4R)}.
\]
From this, we obtain
\[
|\mathrm{Tr}\, M^n|\le  \left(C_* b^{-1/(4R)}\right)^{n}\cdot \sharp \cJ\quad \text{for }n\ge 1
\] 
by estimating the diagonal entries. 
Now we can deduce the latter claim of the proposition:
\begin{align*}
|\log \det D\Psi_{\tau}(\bt)-&\log \det A(\varphi_0)|
= |\log \det (\mathbf{1}+M)|=|\mathrm{Tr} \log (\mathbf{1}+M)|\\
&=
\left|\sum_{n=1}^\infty  \frac{1}{n} \mathrm{Tr} M^n\right|\le \sum_{n=1}^\infty \frac{C_* \sharp \cJ}{n}  (C_* b^{-1/(4R)})^{n}\le C_*b^{-1/(4R)}\cdot \sharp \cJ
\end{align*}
provided that  $b$ is sufficiently large. 
\end{proof} 

\subsection{Proof of Proposition \ref{Prop:mesure}}
Let  $\cJ$ be either of $\cJ_{\mathrm{even}}$ and $\cJ_{\mathrm{odd}}$. 
Below we follow a standard argument in the large (or moderate) deviation theory\cite{MR1739680}. First, by using the fact that $1+s\le \exp(s)\le 1+s+s^2$ when $|s|\ll 1$ and that 
$\int_{-\pi}^{\pi}\mathrm{Re}\,( \exp( i s))ds=0$, 
we see
\begin{align}\label{eq:LD}
&\int_{[-\pi,\pi]^{\cJ}}
\exp\bigg(b^{-1/(8R)}\cdot \mathrm{Re}\,\bigg(\sum_{j\in \cJ} \exp(i x_j)\bigg)\bigg) \prod_{j\in \cJ} \frac{dx_j}{2\pi} \\
&=\left(\int_{-\pi}^{\pi} \exp\bigg(b^{-1/(8R)}\cdot \mathrm{Re}\,\big( \exp(i x)\big)\bigg) \frac{dx}{2\pi}\right)^{\sharp \cJ}\notag\\
&\le 
\left(1+ b^{-1/(8R)} \int_{-\pi}^{\pi} \mathrm{Re}\,(\exp(ix)) \frac{dx}{2\pi}+\int_{-\pi}^{\pi}  b^{-1/(4R)}\frac{dx}{2\pi}\right)^{\sharp \cJ}< \exp( b^{-1/(4R)}\cdot  \sharp \cJ).\notag
\end{align}
Let us recall the mapping $\Psi_{\tau}:[-4,4]^{\cJ}\to \real^{\cJ}$.
From the former claim of Proposition \ref{prop:Psi},  we can take a constant $K_*>0$ so that the subset     
\begin{equation}\label{def:Y}
Y=Y(\varphi_0)=\left\{ (x_j)_{j\in \cJ}\in [-4,4]^{\cJ}\;\left|\; {a}_j(\varphi_0)\cdot |x_j|+K_* b^{-1/(4R)}<\pi,\; \forall j\in \cJ\right.\right\}
\end{equation}
satisfies  
\begin{equation}\label{eq:inclu}
\Psi_{\tau}(Y)\subset [-\pi,\pi]^{\cJ}\quad \text{for $\tau\in (-b^{-1/R}, b^{-1/R})$. }
\end{equation}
Hence, by the latter claim  of Proposition \ref{prop:Psi} and the inequality \eqref{eq:LD} above, we see 
\begin{align*}
&\int_{Y}
\exp\bigg(b^{-1/(8R)}\cdot \mathrm{Re}\,\sum_{j\in \cJ} \exp\left(i\,\Psi_{\tau}(\bt)_j\right)\bigg) d\bt\\
&=\int_{\Psi_\tau(Y)}
\exp\bigg(b^{-1/(8R)}\cdot \mathrm{Re}\,\bigg(\sum_{j\in \cJ} \exp(i x_j)\bigg)\bigg) 
\cdot |\det(D\Psi_\tau)(\Psi^{-1}_{\tau}((x_j)_{j\in \cJ}))|^{-1} \prod_{j\in \cJ} dx_j
\\
&\le \int_{[-\pi,\pi]^{\cJ}}
\exp\bigg(b^{-1/(8R)}\cdot \mathrm{Re}\,\bigg(\sum_{j\in \cJ} \exp(i x_j)\bigg)\bigg) 
\prod_{j\in \cJ} dx_j\times  \exp\big(C_*  b^{-1/(4R)}\cdot  \sharp \cJ \big)\cdot |\det  A(\varphi_0)|^{-1}
\\
&\le (2\pi)^{\# \cJ}\cdot \exp\big(C_*  b^{-1/(4R)}\cdot  \sharp \cJ \big)\cdot |\det  A(\varphi_0)|^{-1}\\
&\le  \exp\big(C_*  b^{-1/(4R)}\cdot  \sharp \cJ \big)\cdot \mathrm{Leb}(Y)
\end{align*}
where $\Psi_{\tau}(\bt)_j$ denotes the $j$-th component of $\Psi_{\tau}(\bt)$. (In the above, we have used \eqref{eq:inclu} and Proposition \ref{prop:Psi}  in the first inequality, the estimate \eqref{eq:LD} in the second and the definition (\ref{def:Y}) in the third.)
Taking average of the both sides of the inequality above with respect to $\tau$ on  $[-b^{-1/R},b^{-1/R}]$ and noting $\sharp J\le b^{1/R}$, we find 
\begin{align*}
 \int_{Y}  \bigg(\frac{1}{2b^{-1/R}} \int_{-b^{-1/R}}^{b^{1/R}}
\exp\bigg(b^{-1/(8R)}&\sum_{j\in \cJ} \mathrm{Re}\, \exp\left(i \big(
  \tilde{\psi}_{p,(-1,1)}^s(s(j)+\tau;a_0^{-1}\bt)+\alpha (s(j)+\tau) \big)
  \right)\bigg)\, d\tau\bigg) d\bt\\
  &\qquad\qquad <  \exp\big(C_* b^{3/(4R)}\big)\cdot \mathrm{Leb}(Y).
\end{align*}
Then, applying Jensen's inequality for the exponential function (outside)  to the integral with respect to $\tau$, we deduce
\begin{align*}
&\int_{Y}
\exp\bigg(\frac{b^{7/(8R)}}{2}\int_{I}\mathrm{Re}\, \exp\left(i \big(
  \tilde{\psi}_{p,(-1,1)}^s(\tau;\textcolor{\revisionColor}{a_0^{-1}}\bt)+\alpha \tau \big)
  \right)d\tau\bigg)\, d\bt< \exp\big(C_* b^{3/(4R)}\big)\cdot \mathrm{Leb}(Y)
\end{align*}
where $I(\cJ):=\cup_{j\in \cJ}J(j)$. This implies
\begin{align*}
\mathrm{Leb}&\left\{ 
\bt\in Y\left\vert
 \mathrm{Re}\, 
 \int_{I(\cJ)} 
\exp\left(
i \big(
  \tilde{\psi}_{p,(-1,1)}^s(\tau;\textcolor{\revisionColor}{a_0^{-1}}\bt)+\alpha \tau  \big)
  \right)
  d\tau >b^{-1/(16R)}\right.
  \right\}\\
  &\qquad < \exp\big(C_* b^{3/(4R)}
  -(1/2)b^{13/(16R)}\big)\cdot \mathrm{Leb}(Y)<\exp(-b^{3/(4R)})\cdot \mathrm{Leb}(Y)
 \end{align*}
provided that $b$ is large enough. 

In the argument above, we considered only the real part of  $\exp(i(\tilde{\psi}_{p,(-1,1)}^s(\tau;\bt)+\alpha \tau))$. But we can argue in parallel also for  either of
\[
-\mathrm{Re}\, \exp(i(\tilde{\psi}_{p,(-1,1)}^s(\tau;\bt)+\alpha \tau))\quad \text{or}\quad
\pm \mathrm{Im}\, \exp(i(\tilde{\psi}_{p,(-1,1)}^s(\tau;\bt)+\alpha \tau)).
\]
Therefore we obtain
\[
\frac{1}{\mathrm{Leb}(Y)} \cdot 
\mathrm{Leb}\left\{\; \bt\in Y\;\left|\;\; \left|\int_{I(\cJ)} 
\exp\left(
i \big(
  \tilde{\psi}_{p,(-1,1)}^s(\tau;a_0^{-1}\bt)+\alpha \tau  \big)
  \right)
  d\tau \right|>b^{-1/(16R)}\right.\right\}<\exp(-b^{3/(4R)})
\] 
Note that this estimate is uniform for $\varphi_0\in \mathcal{W}$.
Letting $\mathcal{W}$ be slightly smaller, we may and do suppose that this estimate is true for all $\varphi_0\in C^r(M)$  such that 
\[
\varphi_0+\sum_{j\in \cJ} t_j\varphi_j \in \mathcal{W} \quad \text{for some $\bt=(t_j)_{j\in \cJ} \in [-4,4]^{\cJ}$.}
\]
Therefore, if we define a measure $\boldsymbol{\nu}$ on $\mathcal{W}$ by setting
\[
\boldsymbol{\nu}(X)=\int_{C^r(M)} 
\left(\frac{1}{\mathrm{Leb}(Y(\varphi_0))}\int_{Y(\varphi_0)} \mathbf{1}_X
\left(\varphi_0+\sum_{j\in \cJ} \textcolor{\revisionColor}{a_0^{-1}}t_j\varphi_j\right) d\bt\right)
 d\boldsymbol{\mu}(\varphi_0)
\]
for measurable subsets $X\subset \mathcal{W}$, 
where $Y(\varphi_0)$ is that in \eqref{def:Y},
it holds
\begin{equation}\label{eq:req}
\boldsymbol{\nu}\left\{\; \varphi\in \mathcal{W} \;\left|\; \int_{I} 
\left|\exp\left(
i \big(
  \tilde{\psi}_{p,(-1,1)}^s(\tau;\varphi)+\alpha \tau  \big)
  \right)\right|
  d\tau >b^{-1/(16R)}\right.\right\}< \exp(-b^{3/(4R)}).
\end{equation}

For the proof of Proposition \ref{Prop:mesure}, we need to show the inequality similar to the above, but for the measure $\boldsymbol{\mu}$ in the place of $\boldsymbol{\nu}$. 
To this end, we show  that the measure $\boldsymbol{\mu}$ restricted to $\mathcal{W}$ is absolutely continuous with respect to $\boldsymbol{\nu}$ with bounded  Radon-Nikodym derivative. This is basically a simple consequence of  the property \eqref{eq:trans_inv} of the measure~$\boldsymbol{\mu}$. But, since $Y(\varphi)$ depends on the function $\varphi\in \mathcal{W}$, we need a little argument. Let us take the following approximations of $Y(\varphi)$ from inside and outside:   
\[
Y^\pm(\varphi)=\left\{ (x_j)_{j\in \cJ}\in [-4,4]^{\cJ}\;\left|\; {a}_j(\varphi)\cdot |x_j|+K_* b^{-1/(4R)}<\textcolor{\revisionColor}{\pi} \pm   b^{-1},\; \forall j\in \cJ\right.\right\}
\]
so that $Y^-(\varphi)\subset Y(\varphi)\subset Y^+(\varphi)$. 
If we write $\Phi(\bt)=\sum_{j\in \cJ} \textcolor{\revisionColor}{a_0^{-1}} t_j\varphi_j$ for $\bt=(t_j)_{j\in \cJ}$ for brevity, we have  $
\|\Phi(\bt)\|_\infty<C_* b^{-1-1/R}$ for  $\bt\in [-4,4]^{\cJ}$  from \eqref{eq:basic_est_family}. 
Hence, from \eqref{eq:var_phi}, we have
\[
Y^-(\varphi)\subset Y(\varphi_0)\subset Y^+(\varphi)\quad \text{ whenever }\varphi=\varphi_0+\Phi(\bt)\text{ for some $\bt=(t_j)_{j\in \cJ}\in Y(\varphi_0)$.}
\]
For any measurable subset $X\subset \mathcal{W}$, we have
\begin{align*}
\boldsymbol{\nu}(X)&= 
\int_{C^r(M)\times Y(\varphi_0)} \frac{\mathbf{1}_{X}(\varphi_0+\Phi(\bt))}{\mathrm{Leb}(Y(\varphi_0))}
 (d\boldsymbol{\mu}(\varphi_0)\times d\bt)\\
&\ge 
\int_{\mathcal{W}\times Y^{-}(\varphi)} \frac{\mathbf{1}_{X}(\varphi)}{\mathrm{Leb}(Y^+(\varphi))}
d\boldsymbol{\mu}(\varphi-\Phi(\bt)) d\bt
\end{align*}
where, in the second line, we changed the variable $(\varphi_0,\bt)$ to $(\varphi=\varphi_0+\Phi(\bt),\bt)$. Then, by using  the property \eqref{eq:trans_inv} of $\boldsymbol{\mu}$ and \eqref{eq:basic_est_family}, we deduce 
\begin{align*}
\boldsymbol{\nu}(X)&
\ge  
\int_{\mathcal{W}\times Y^-(\varphi)} \frac{(1- C_* b^{-1/R})\cdot \mathbf{1}_{X}(\varphi)}{\mathrm{Leb}(Y^+(\varphi))}
d\boldsymbol{\mu}(\varphi)  d\bt
\ge
 \frac{1}{2} \frac{\mathrm{Leb}(Y^-(\varphi))}{\mathrm{Leb}(Y^+(\varphi))}
\boldsymbol{\mu}(X) \ge 
\frac{\boldsymbol{\mu}(X)}{4}.
\end{align*}

By the claim proved in the last paragraph and \eqref{eq:req}, we deduce 
\[
\boldsymbol{\mu}\left\{\; \varphi\in \mathcal{W} \;\left|\; \int_{I(\cJ)} 
\exp\left(
i \big(
  \tilde{\psi}_{p,(-1,1)}^s(\tau;\varphi)+\alpha \tau  \big)
  \right)
  d\tau >b^{-1/(16R)}\right.\right\}< 4\exp(-b^{3/(4R)}).
\]
With this inequality for the cases $\cJ=\cJ_{\mathrm{even}},\cJ_{\mathrm{odd}}$ and Remark \ref{rem:excp_j}, we 
 complete the proof of  Proposition \ref{Prop:mesure}, letting $0<\rho<1/(16R)$. 
%%%%%%%%%%%%%%%%%%%%%%%%%%%
\section{Local charts}
\label{sec:lc}
In the proof of Theorem \ref{th:exp}, which will be carried out in Section \ref{sec:proof1}, we study the  semigroup of  transfer operators (or Perron-Frobenius operators) 
\begin{equation}\label{def:cL}
\cL^t: L^2(M)\to L^2(M),\quad \cL^t u=u\circ f^{-t}
\end{equation}
associated to a flow $f^t\in \fF^3_A$. For analysis of the  action of $\cL^t$, we will decompose functions on $M$ with respect to the frequency in the flow direction and then investigate the action on each of the components. 
Since the operator $\cL^t$ virtually preserves the frequency in the flow direction, this method is natural and works efficiently. (See \cite{MR2652469, MR2995886} for the corresponding argument in the case of contact Anosov flows.)

In this section, we present preliminary arguments for the proof of Theorem \ref{th:exp}. We henceforth consider a $C^3$ flow $f^t\in \fF^3_A$ generated by a vector field $v$. 
%We first set up a system of local charts and a partition of unity for each integer $\omega\in \integer$, which will be used to consider the action of $\cL^t$ on the components of functions with the frequency around $\omega$ in the flow direction. We will then investigate how the central bundle $E_0^{*}=(E_s\oplus E_u)^\perp$ and the flow $f^t$ itself look in those local charts. 
We consider a large constant $t_\sharp>0$ that will be specified later in the course of the argument. 
Roughly the constant $t_{\sharp}>0$ will be taken so that the flow $f^t$ with $t\ge t_\sharp$ exhibits sufficiently strong hyperbolicity. Also we set 
\begin{equation}\label{eq:konst_def}
\konst=\exp({t_{\sharp}}^2)
\end{equation} so that we may suppose $
\|Df^{t_{\sharp}}\|_{\infty}\le C_* \exp(C_* t_\sharp) \ll \konst$.
\begin{Remark}
The constant $t_\sharp$ (as well as the constants $\omega_\sharp$ and $m_\sharp$ which will be introduced later) will turn out to be the kind of constants that are denoted by symbols with  subscript~$*$. (See the beginning of Section \ref{sec:ni}.)
But we use symbols with subscript $\sharp$ for them because the choice will be made much later. 
\end{Remark}
\subsection{Local charts depending on $\omega\in \mathbb{Z}$}\label{ss:lc}  
We are going to take  a finite system of local charts on $M$  and a partition of unity subordinate to it,  depending on an integer parameter $\omega\in \mathbb{Z}$. These will be used when we consider the action of transfer operators $\cL^t$ on components of functions that have frequency around $\omega$ in the flow direction. 
We will take them  so that their sizes are proportional to $\konst \pomega^{-1/2}$ (resp. $1$) in the  directions transversal (resp. parallel) to the flow. Also we will let them satisfy a few preferable properties in relation to the flow $f^t$. 

To begin with, we take and fix a finite system of $C^3$ flow box local charts for $f^t$,
\[
\kappa_a:U_a\to B(0,2r_*)\times (-r_*,r_*)\subset \real^3,\quad
\kappa_{a}(m)=(x,y,z)
\]
for $a\in A$ with $\#A<\infty$, where $B(w,r)$ denotes  the open disk in $\real^2$ with radius $r>0$ centered at $w\in \real^2$ and $r_*>0$ is a small real number. We write $\pi:\real^3\to \real^2$ for the projection to the first two components $(x,y)$ in $(x,y,z)\in \real^3$.

\begin{Remark}
By ``flow box local chart", we mean that $(\kappa_a)_* v=\partial_z$ holds for the generating vector field $v$. Since the flow $f^t$ is assumed to be only $C^3$, we can not expect that the flow box local charts $\kappa_a$ is smoother than $C^3$.  
\end{Remark}

Also we take $C^3$ functions $
\rho_a, \;\tilde{\rho}_a:\real^3\to [0,1]$ for $a\in A$
so that 
\begin{itemize}
\item the supports of $\rho_a$ and $\tilde{\rho}_a$ are contained in $B(0,r_*)\times (-r_*,r_*)$,
\item $\{\rho_a\circ \kappa_a\mid  a\in A\}$ is a partition of unity on $M$, that is, $\sum_{a\in A}\rho_a\circ \kappa_a\equiv 1$, and 
\item $\tilde{\rho}_a\equiv 1$ on the support of $\rho_a$.
\end{itemize}
By applying a mollifier along the flow, we can and do assume that $\rho_a$ are $C^\infty$ with respect to the variable $z$ and each of the partial derivatives $\partial_x^{k}\partial_y^{\ell}\partial_z^{m}\rho_{a}$ and $\partial_x^{k}\partial_y^{\ell}\partial_z^{m}\tilde{\rho}_{a}$ with $k+\ell\le 3$ are continuous and therefore bounded. 

Next, depending on $\omega \in \integer$, we construct a finer system of local charts and a partition of unity subordinate to it. The construction of the local charts is done in two steps as follows. 
For the first step, we take a finite subset $N(a,\omega)\subset B(0,r_*)$ and, for each  $n\in N(a, \omega)$, we take its neighborhood $V_{a,\omega,n}\subset B(0,2r_*)$ and a $C^3$ diffeomorphism 
\[
g_{a,\omega,n}:V_{a,\omega,n}\times \real\to D_{a,\omega,n}\times \real \subset B(0,2\konst\pomega^{-1/2})\times \real
\] 
of the form 
\[
g_{a,\omega,n}(x,y,z)=(\hat{g}_{a,\omega,n}(x,y), z+\check{g}_{a,\omega,n}(x,y)).
\]
We suppose that $N(a,\omega)$ contains many points so that  $V_{a,\omega,n}$ for $n\in N(a,\omega)$ cover $B(0,r_*)$.
For the diffeomorphisms $g_{a,\omega,n}$, we may and do assume the following conditions:
\begin{itemize}
\item[(G0)] the $C^3$ norms of  ${g}_{a,\omega, n}$ and those of their inverses are bounded by $C_*$,
\item[(G1)] $g_{a,\omega, n}(n,0)=(0,0,0)$ and the differential $D(g_{a,\omega,n}\circ \kappa_a)$ at $p_{a,\omega,n}:=\kappa_a^{-1}(n,0)$ carries   $E_u(p_{a,\omega,n})$, $E_s(p_{a,\omega,n})$, $E_0(p_{a,\omega,n})$ to the $x$-axis, $y$-axis, $z$-axis respectively,
\item[(G2)]  there exists $\omega_0>0$  such that, if $|\omega|\ge \omega_0$,  we have $D_{a,\omega,n}=B(0,2\konst\pomega^{-1/2})$ and 
\[
g_{a,\omega,n}\circ \kappa_a(w^u_{p_{a,\omega,n}}(\tau))=(\tau,0,0),\qquad 
g_{a,\omega,n}\circ \kappa_a(w^s_{p_{a,\omega,n}}(\tau))=(0,\tau,0)
\]
for $\tau\in [-2\konst \pomega^{-1/2}, 2\konst\pomega^{-1/2}]$,
where $w^s_{p_{a,\omega,n}}(\cdot)$ and $w^u_{p_{a,\omega,n}}(\cdot)$ are the intrinsic parametrization of the stable and unstable manifolds introduced in Section \ref{ss:intrinsic_metric}.
\item[(G3)] $(g_{a,\omega,n}\circ \kappa_a)_*m=c_{a,\omega,n}\cdot dx dy dz$ for some constant $c_{a,\omega,n}$. 
\end{itemize}
\begin{Remark} \label{rem:caon}
The constant $c_{a,\omega,n}$ in (G3) is determined by the angles between the subspaces $E_0(p_{a,\omega,n})$, $E_s(p_{a,\omega,n})$ and $E_u(p_{a,\omega,n})$. Clearly we have $C_*^{-1}\le |c_{a,\omega,n}|\le C_*$. 
\end{Remark}
 
For the second step, we first recall the argument in Subsection \ref{ss:approxNI} and, especially, the definition of the quantity $\torsion^s(q,\delta)$ in \eqref{def:tors}.
We define a non-linear diffeomorphism
\[
b_{a,\omega,n}:\real^3\to \real^3,\qquad
b_{a,\omega,n}(x,y,z)=(x,y,\,z+\beta(a,\omega,n)\cdot xy)
\]
where  
\begin{equation}\label{def:beta}
\beta(a,\omega,n)=
c_{a,\omega,n}\cdot \torsion^s\left(p_{a,\omega,n}, \konst \pomega^{-1/2}\right)\quad \text{with $p_{a,\omega,n}=\kappa_a^{-1}(n,0)$}
\end{equation}
if $|\omega|\ge \omega_0$ and $\beta(a,\omega,n)=0$ otherwise.
Note that 
\begin{equation}\label{eq:beta}
|\beta(a,\omega,n)|<C_*\langle\log \langle \omega\rangle\rangle\quad \text{for any $a\in A$, $\omega\in \integer$, $n\in \mathcal{N}(a,\omega)$}
\end{equation}
from \eqref{eq:torsion_bound} in Lemma \ref{lm:Tor}. 
Then we set $U_{a,\omega,n}=U_a\cap \kappa^{-1}_a(V_{a,\omega,n}\times \real)$ and regard  
\[
\kappa_{a,\omega,n}:=b_{a,\omega,n}\circ g_{a,\omega,n}\circ \kappa_a:U_{a,\omega,n}\to U'_{a,\omega,n}\subset D_{a,\omega,n}\times \real,\qquad\text{for}\quad
 (a,n)\in A\times N(a,\omega)
\]
as the system of local charts on $M$ defined for $\omega\in \integer$. These are flow box local charts and satisfies  $\kappa_{a,\omega,n}(p_{a,\omega,n})=(0,0,0)$ and 
\[
\kappa_{a,\omega,n}(w^u_{p_{a,\omega,n}}(\tau))=(\tau,0,0),\quad 
\kappa_{a,\omega,n}(w^s_{p_{a,\omega,n}}(\tau))=(0,\tau,0)\;\;\text{for $\tau\in [-2\konst \pomega^{-1/2}, 2\konst\pomega^{-1/2}]$.}
\]
Also we have $
(\kappa_{a,\omega,n}^{-1})^*m=c_{a,\omega,n}\cdot  dx dy dz$
from (G2) and (G3) above. 

\begin{Remark}\label{rem:b}
The motivation to compose $b_{a,\omega,n}$ in the second step can be explained as follows. Recall the argument in Subsection \ref{ss:approxNI} and, in particular, the definitions of the sections $\gamma^\dag_{p,(-\delta,\delta)}$ and $\hat{\gamma}^\dag_{p,(-\delta,\delta)}$.
 Suppose that $|\omega|\ge \omega_0$.  
From (G2), we can express $\gamma^\dag_{p,(-\delta,\delta)}$ and $\hat{\gamma}^\dag_{p,(-\delta,\delta)}$ for $p=p_{a,\omega,n}$ and $\delta=\konst\pomega^{-1/2}$ (viewed in the local chart $\kappa_{a,\omega,n}$)  as 
\[
(\kappa_{a,\omega,n}^{-1})^*\circ \gamma^\dag_{p,(-\delta,\delta)}(w^u_{p}(\tau))=(0,\varphi(\tau),1),\quad
(\kappa_{a,\omega,n}^{-1})^*\circ \gamma^\dag_{p,(-\delta,\delta)}(w^s_{p}(\tau))=(\hat{\varphi}(\tau),0,1)
\] 
with $C^2$ functions $\varphi,\hat{\varphi}:(-\delta,\delta)\to \real$. Also note that we have
\[
(\kappa_{a,\omega,n}^{-1})^*\circ \gamma^\perp_{p,(-\delta,\delta)}(p_{a,\omega,n})=(0,c_{a,\omega,n},1),\quad
(\kappa_{a,\omega,n}^{-1})^*\circ \hat{\gamma}^\perp_{p,(-\delta,\delta)}(p_{a,\omega,n})=(c_{a,\omega,n},0,1).
\] 
If we did not have the post-composition of $b_{a,\omega,n}$, we would have
\[
|\varphi'(\tau)-c_{a,\omega,n}\cdot 
\torsion^s(p,\konst\pomega^{-1/2}) |<C_*,\quad 
|\hat{\varphi}'(\tau)-c_{a,\omega,n}\cdot \torsion^u(p,\konst\pomega^{-1/2}) |<C_*
\]
because the $C^2$ norms of $\gamma^{0}_{p,(-\delta,\delta)}$ and $\hat{\gamma}^{0}_{p,(-\delta,\delta)}$ are bounded by a uniform constant $C_*$. 
Hence, with  the post-composition of $b_{a,\omega,n}$, we have  
\begin{equation}\label{eq:varphi}
|\varphi'(\tau)|<C_*\quad \text{and}\quad |\hat{\varphi}'(\tau)-c_{a,\omega,n}\cdot \Delta(p,\konst\pomega^{-1/2}) |<C_*.
\end{equation}
That is to say, by the post-composition of $b_{a,\omega,n}$, 
we stabilize the rotation of the stable subbundle $E_s$ along $W^u_{(-\delta,\delta)}(p)$. Consequently the unstable subbundle $E_u$ will look rotating along $W^s_{(-\delta,\delta)}(p)$ by a rate  proportional to $\Delta(p,\konst\langle \omega\rangle^{-1/2})$.  
\end{Remark}

We next construct a partition of unity associated to the systems of local charts $\{ \kappa_{a,\omega,n}\}_{a,n}$ for $\omega\in \integer$. 
Let $\varrho_{a,\omega,n}, \tilde{\varrho}_{a,\omega,n}:\real^2\to [0,1]$ for  $n\in N(a,\omega)$  be $C^3$ functions such that
\begin{enumerate}
\item  $\supp \varrho_{a,\omega,n}\subset\supp \tilde{\varrho}_{a,\omega,n}\subset  V_{a,\omega,n}$,
\item $\sum_{n\in N(a,\omega)}\varrho_{a,\omega,n}=1$ on $B(0,r_*)$,
and $\tilde{\varrho}_{a,\omega,n}\equiv 1$ on $\supp \varrho_{a,\omega,n}$,
\item  $\max\{\|\partial^{\alpha}\varrho_{a,\omega,n}\|_\infty,\|\partial^{\alpha}\tilde{\varrho}_{a,\omega,n}\|_\infty\}\le C_{*}(\alpha)(\konst^{-1} \pomega^{1/2})^{|\alpha|}$ for any multi-index $\alpha$ with $|\alpha|\le 3$, where $C_*(\alpha)$ does not depend on $\omega$.  
\end{enumerate}
For each $\omega\in \mathbb{Z}$, we consider the family of functions
\[
\rho_{a,\omega,n}=(\rho_a\cdot \varrho_{a,\omega,n})\circ g_{a,\omega,n}^{-1}\circ b_{a,\omega,n}^{-1},\quad
\tilde{\rho}_{a,\omega,n}=(\tilde{\rho}_a\cdot \tilde{\varrho}_{a,\omega,n})\circ g_{a,\omega,n}^{-1}\circ b_{a,\omega,n}^{-1}
\]
defined for $a\in A$ and $n\in N(a,\omega)$, 
where we regard $\varrho_{a,\omega,n}$ as a function on $\real^3$ by reading  $\varrho_{a,\omega,n}(x,y,z)=\varrho_{a,\omega,n}(x,y)$. The set of functions $\rho_{a,\omega,n}\circ \kappa_{a,\omega,n}$ for  $a\in A$ and $n\in N(a,\omega)$ is a partition of unity subordinate to the system of local charts $\{\kappa_{a,\omega,n}\}$ and the function $\tilde{\rho}_{a,\omega,n}\circ \kappa_{a,\omega,n}$ takes the constant value $1$ on $\supp(\rho_{a,\omega,n}\circ \kappa_{a,\omega,n})$. 
\begin{Remark}\label{Rem:intersection_mulitpliicity_rho} 
Since we have wide choices in the definitions above, 
we may and do put the following additional assumptions without loss of generality:
\begin{itemize}
\item  For each $\omega\in \integer$, the intersection multiplicity of the $8\konst \pomega^{-1/2}$-neighborhoods of the  subsets $
\pi (\supp \tilde{\rho}_{a,\omega,n})\subset \real^2$ for $n\in N(a,\omega)$ 
is bounded by an absolute constant.
\item The diameters of the supports $\supp \tilde{\varrho}_{a,\omega,n}$ are bounded uniformly by any given constant.
\end{itemize}
For the second condition, we will specify the bound later in the course of the argument. 
\end{Remark}

\subsection{The central bundle $E_0^*$ viewed in the local charts}
In this subsection, we consider how the central subbundle $E_0^*=(E_s\oplus E_u)^{\perp}$ looks in the local charts $\kappa_{a,\omega,n}$. Since $E^*_0$ is invariant with respect to the flow $f^t$ and since $\kappa_{a,\omega,n}$ is a flow box local chart, the subspace $E^*_0$ viewed in those local charts does not depend on the variable $z$. That is to say, there is a unique continuous mapping
\begin{equation}\label{eq:eaob}
e_{a,\omega,n}:D_{a,\omega, n}\to \real^2, \quad e_{a,\omega,n}(w)=(\theta_{a,\omega,n}^u(w), \theta_{a,\omega,n}^s(w))
\end{equation}
such that  
\[
(D\kappa_{a,\omega,n})^*_p(e_{a,\omega,n}(w),1)\in E_0^*(p)\quad \text{when}\quad \kappa_{a,\omega,n}(p)=(w,z)\text{ and } p\in U_{a,\omega,n}.
\]
From the assumption (G2) on the choice of $g_{a,\omega,n}$, we have
\[
\theta^u_{a,\omega,n}(\tau,0)=\theta^s_{a,\omega,n}(0,\tau)=0\quad\text{for }\quad\tau\in [-2\konst \pomega^{-1/2}, 2\konst\pomega^{-1/2}].
\]
By slight abuse of notation, we will sometimes regard the functions $e_{a,\omega,n}$, $\theta^u_{a,\omega,n}$ and $\theta^s_{a,\omega,n}$ above as  functions on $\real^{3}$ by letting $e_{a,\omega,n}(x,y,z)=e_{a,\omega,n}(x,y)$ and so on. 
\begin{Remark}\label{rem:eaon}
The function $e_{a,\omega,n}$ is not smooth in general, but  satisfies
\begin{equation}\label{eq:e_var}
\|e_{a,\omega,n}(w')-e_{a,\omega,n}(w)\|\le C_* \|w'-w\|\cdot  
 \big\langle\log \|w'-w\|\big\rangle
\quad\text{for }w,w'\in D_{a,\omega, n}
\end{equation}
 from \eqref{eq:irregularity_distribution} and \eqref{eq:torsion_bound}. 
In particular, 
\begin{equation}\label{eq:e_var_2}
\|e_{a,\omega,n}(w)\|=\|e_{a,\omega,n}(w)-e_{a,\omega,n}(0)\|\le 
 C_*\konst \pomega^{-1/2} 
\log \pomega \quad\text{for }w\in D_{a,\omega, n}.
\end{equation}
\end{Remark}

The next lemma gives estimates on the variation of the functions $\theta^{\sigma}_{a,\omega,n}(w)$, $\sigma=s,u$. The former statement is a consequence of the construction of the local charts (see Remark \ref{rem:b}), and  the latter is that of the definition of the $s$-templates and the non-integrability condition $(NI)_{\rho}$. We take and fix a constant $0<\theta_*<1/2$.
(One may suppose $\theta_*=1/4$, say.) 

\begin{Lemma}\label{lm:e_approx} 
There exists a constant $\omega_\sharp>\omega_0$ depending on $t_\sharp$ such that, for any  $a\in A$, $\omega\in \integer$, $n\in \mathcal{N}(a,\omega)$ with $|\omega|\ge \omega_\sharp$, the following holds true. For a point $w\in B(0,\konst\pomega^{-1/2})$ and $h\in [\pomega^{-1+\theta_*}, \,\konst \pomega^{-1/2}]$, we consider the segments
\[
\ell,\hat{\ell}:(-h,h)\to \real^2,\quad \ell(\tau)=w+(\tau, 0),\quad \hat{\ell}(\tau)=w+(0,\tau)
\]
parallel to the $x$-axis and $y$-axis respectively. 
Then, for $-h\le \tau\le h$, we have
\begin{align}
\label{eq:tv3}
&\left|\theta^{u}_{a,\omega,n}\circ \hat{\ell}(\tau)-\theta^{u}_{a,\omega,n}\circ \hat{\ell}(0)-c_{a,\omega,n} \Delta(p_{a,\omega,n},\konst\pomega^{-1/2})\cdot \tau\right|<C_* h  \langle\log (\konst\langle \omega\rangle^{-1/2}/h)\rangle, \\
\label{eq:tv4}
&
\left|\theta^{s}_{a,\omega,n}\circ \hat{\ell}(\tau)-\theta^{s}_{a,\omega,n}\circ \hat{\ell}(0)\right|<C_* \konst^2 \pomega^{-1} \log \pomega
\end{align}
and similarly 
\begin{align}
&\left|\theta^{s}_{a,\omega,n}\circ {\ell}(\tau)-\theta^{s}_{a,\omega,n}\circ {\ell}(0)\right|<C_* h \cdot \langle \log (\konst\langle \omega\rangle^{-1/2}/h)\rangle,\label{eq:tv5}\\
&
\left|(\theta^{u}_{a,\omega,n}\circ {\ell}(\tau)-\theta^{u}_{a,\omega,n}\circ {\ell}(0))\right|<C_* \konst^2 \pomega^{-1} \log \pomega.\label{eq:tv6}
\end{align}
If the non-integrability condition $(NI)_\rho$ holds, we have, for sufficiently large $b_0>0$, that
\begin{equation}\label{eq:tv2}
\frac{1}{2h}\left|\int_{-h}^h \exp\left(ib h^{-1} \left(\theta^s_{a,\omega,n}(\ell(\tau))+\alpha \tau\right)\right) d\tau \right|<
b^{-\rho/2} 
\end{equation} 
for any $b_0\le b\le \konst$, $\alpha\in \real$ and any $h$, $w$ as above.
\end{Lemma}
\begin{proof} 
We prove the claims \eqref{eq:tv3} and \eqref{eq:tv4}.  
Let $q=\kappa_{a,\omega,n}^{-1}(w,0)$.
We first see that how the segments $\ell$ and $\hat{\ell}$ approximate local unstable and stable manifolds on the local charts respectively.  
From the definition of local charts $\kappa_{a,\omega,n}$, \eqref{eq:irregularity_distribution} and \eqref{eq:beta}, we see that 
\begin{equation}
\label{eq:aprox_l2}
\|( \kappa_{a,\omega,n}\circ w^s_q)'(\tau)-(0,1,0)\|\le C_* \konst\pomega^{-1/2} \log \pomega\quad\text{for }-h<\tau<h.
\end{equation}
Further, from the fact noted in Remark \ref{rem:eseu}, we actually have better estimate if we consider the image by the projection $\pi:\real^3\to \real^2$, $\pi(x,y,z)=(x,y)$ along the $z$-axis:
\begin{equation}
\label{eq:aprox_l2_pi}
\|\pi\circ ( \kappa_{a,\omega,n}\circ w^s_q)'(\tau)-(0,1)\|\le C_* \konst\pomega^{-1/2}\quad\text{for }-h<\tau<h.
\end{equation}
For the second derivative, we can prove the following estimate:
\begin{equation}
\label{eq:wudd}
\|(\kappa_{a,\omega,n}\circ w^s_q)''(\tau)\|< C_* \konst\pomega^{-1/2} \log \pomega\quad \text{ for $-h<\tau<h$.}
\end{equation}
Indeed,  at the origin,  we have $(\kappa_{a,\omega,n}\circ w^s_{p})''(0)=0$ with $p=p_{a,\omega,n}$ from the choice of the local chart $\kappa_{a,\omega,n}$.
Hence, if we disregard the composition of $b_{a,\omega,n}$ in the construction of the local chart $\kappa_{a,\omega,n}$, we obtain \eqref{eq:wudd} by a simple estimate\footnote{This estimate is not completely obvious but parallel to that behind \eqref{eq:irregularity_distribution} and straightforward. } on continuity of the $2$-jets of the stable manifolds. 
Then, by using \eqref{eq:beta} and \eqref{eq:aprox_l2_pi}, we can check that the estimate remains valid when
we restore the composition of $b_{a,\omega,n}$. 
Note that the estimates \eqref{eq:aprox_l2}, \eqref{eq:aprox_l2_pi} and \eqref{eq:wudd} above imply respectively 
\begin{align}
&\|\kappa_{a,\omega,n}\circ w^s_q(\tau)-\hat{\ell}(\tau)\|
\le C_* \konst\pomega^{-1/2} \log \pomega\cdot \tau, 
\\
&\|\pi \circ \kappa_{a,\omega,n}\circ w^s_q(\tau)-\pi\circ \hat{\ell}(\tau)\|
\le C_* \konst\pomega^{-1/2}\cdot \tau,\quad \text{and}\label{eq:wudd3}\\
&\|(\kappa_{a,\omega,n}\circ w^s_q)'(\tau)-(\kappa_{a,\omega,n}\circ w^s_q)'(0)\|
\le C_* \konst\pomega^{-1/2}\log \pomega\cdot \tau \label{eq:wudd2}
\end{align}
for $-h<\tau<h$. 
Putting \eqref{eq:wudd3} in  \eqref{eq:e_var}, we find that
\begin{equation}\label{eq:diff_l_w}
\|e_{a,\omega,n}(\hat{\ell}(\tau))-e_{a,\omega,n}(\kappa_{a,\omega,n}\circ w^s_q(\tau))\|\le C_* \konst^2 \pomega^{-1} \log \pomega\quad\text{for }-h<\tau<h.
\end{equation}
The right hand side above is so small that the claims  \eqref{eq:tv3} and \eqref{eq:tv4} follow if we prove them with the segment $\hat{\ell}$ replaced by the curve $ \kappa_{a,\omega,n}\circ w^s_q$.

Let $e_x$ be the vector field on $U_{a,\omega,n}\subset M$ defined by $e_x=(\kappa_{a,\omega,n}^{-1})_*(\partial_x)$. From the definition of $\hat{\gamma}^\dag_{q,(-h,h)}$ and the relation \eqref{eq:recur}, it holds, for $-h<\tau<h$, 
\begin{align*}
\|\gamma^u_{q,(-h,h)}(w^s_q(\tau))-{\hat{\gamma}}^\dag_{q,(-h,h)}(w^s_q(\tau))\| &\le C_* h\cdot \|\gamma^u_{q',(-1,1)}(w^s_{q'}(h^{-1}\tau))-{\hat{\gamma}}^0_{q',(-1,1)}(w^s_{q'}(h^{-1}\tau))\|\\
& <C_*h
\end{align*}
where we take $t>0$ so that $\|Df^{-t}_q|_{E_s}\|=1/h$ and set $q'=f^{-t}(q)$. 
Hence 
\begin{align}\label{eq:thetau}
\theta^u_{a,\omega,n}(\kappa_{a,\omega,n}&(w^s_q(\tau)))
=\langle \gamma^u_{q,(-h,h)}(w^s_q(\tau)),\; e_x(w^s_q(\tau))\rangle
\\
&=\langle \hat{\gamma}^\dag_{q,(-h,h)}(w^s_q(\tau)), \;e_x(w^s_q(\tau))\rangle+\mathcal{O}_*(h)\notag\\
&=
\left\langle \torsion^u(q,h)\cdot \tau \cdot\hat{\gamma}^{\perp}_{q,(-h,h)}(w^s_q(\tau))+\hat{\gamma}^{0}_{q,(-h,h)}(w^s_q(\tau))\;, \; e_x(w^s_q(\tau))\right\rangle+\mathcal{O}_*(h)\notag
\end{align}
for $-h<\tau<h$. 
For the term $\langle \hat{\gamma}^0_{q,(-h,h)}(w^s_q(\tau)), e_x(w^s_q(\tau))\rangle$ on the last line, we show 
\begin{equation}\label{eq:thetau2}
\langle \hat{\gamma}^0_{q,(-h,h)}(w^s_q(\tau)), e_x(w^s_q(\tau))\rangle
-
\langle \hat{\gamma}^0_{q,(-h,h)}(w^s_q(0)), e_x(w^s_q(0))\rangle
=\beta(a,\omega,n)\tau+\mathcal{O}_*(h)
\end{equation}
for $-h<\tau<h$. To this end, we first disregard the composition of $b_{a,\omega,n}$ in the construction of the local chart $\kappa_{a,\omega,n}$ and check
the inequality \eqref{eq:thetau2} without the term $\beta(a,\omega,n) \tau$ on the right-hand side. This is rather obvious because  $\hat{\gamma}^0_{q,(-h,h)}$ is a restriction of $\hat{\gamma}^0_{q,(-1,1)}$ and hence bounded uniformly in the $C^2$ norm. Then it is clear that the term $\beta(a,\omega,n) \tau$ appears on the right-hand side when we restore the composition of $b_{a,\omega,n}$. 

For the other terms on the last line of  \eqref{eq:thetau}, we have that 
\begin{equation}\label{eq:approx_gamma_perp}
\langle \hat{\gamma}^{\perp}_{q,(-h,h)}(w^s_q(\tau)),e_x(w^s_q(\tau))\rangle=c_{a,\omega,n}+\mathcal{O}_* (\konst\pomega^{-1/2})\quad \text{for $-h<\tau<h$}
\end{equation}
from \eqref{eq:aprox_l2_pi}, and also that
\begin{equation}\label{eq:ToruVar}
|\torsion^u(q,h)-\torsion^u(p_{a,\omega,n},\konst\pomega^{-1/2})|<C_*\log \langle \konst\pomega^{-1/2}/h\rangle
\end{equation}
from Lemma \ref{lm:Tor}. 
Therefore we obtain the first claim \eqref{eq:tv3} with $\hat{\ell}$ replaced by $\kappa_{a,\omega,n}\circ w^s_q$, if we rewrite  the left-hand side of \eqref{eq:tv3} using \eqref{eq:thetau} and then apply the estimates  \eqref{eq:thetau2}, \eqref{eq:approx_gamma_perp} and \eqref{eq:ToruVar} above. 

We next prove \eqref{eq:tv4} with $\hat{\ell}$ replaced by $\kappa_{a,\omega,n}\circ w^s_q$. Note that we have the relation
\begin{equation}\label{eq:orth}
(\theta^u(\kappa_{a,\omega,n}\circ w^s_{q}(\tau)), \theta^s(\kappa_{a,\omega,n}\circ w^s_{q}(\tau)), 1)\cdot (\kappa_{a,\omega,n}\circ w^s_{q})'(\tau) \equiv 0\quad \text{for }-h<\tau<h
\end{equation}
from the definition \eqref{eq:eaob}.
It follows from \eqref{eq:e_var_2} and \eqref{eq:aprox_l2_pi} that
\[
|(\theta^u(\kappa_{a,\omega,n}\circ w^s_{q}(\tau)), \theta^s(\kappa_{a,\omega,n}\circ w^s_{q}(\tau)))\cdot (\pi \circ (\kappa_{a,\omega,n}\circ w^s_{q})'(\tau)-(0,1))|\le C_* \konst^2 \pomega^{-1} \log \pomega.
\]
This together with \eqref{eq:orth} gives
\[
|\theta^s(\kappa_{a,\omega,n}\circ w^s_{q}(\tau)) + (0,0,1)\cdot (\kappa_{a,\omega,n}\circ w^s_{q})'(\tau) |\le C_* \konst^2 \pomega^{-1} \log \pomega).
\]
But, for the second term on the left-hand side, we have 
\[
|(0,0,1)\cdot (\kappa_{a,\omega,n}\circ w^s_{q})'(\tau) -(0,0,1)\cdot (\kappa_{a,\omega,n}\circ w^s_{q})'(0)|\le C_* \konst^2 \pomega^{-1} \log \pomega
\]
from \eqref{eq:wudd2}. Therefore the required estimate \eqref{eq:tv4} follows.   

We have finished the proofs of \eqref{eq:tv3} and \eqref{eq:tv4}. We omit the proofs of \eqref{eq:tv5} and \eqref{eq:tv6} because they are obtained by translating the argument above to the time-reversed flow $f^{-t}$ through obvious correspondences. Let us proceed to the proof of the last claim \eqref{eq:tv2}. 
From \eqref{eq:diff_l_w} for the time-reversed flow $f^{-t}$, it is enough to prove the claim \eqref{eq:tv2} with $\ell(\tau)$ in it replaced by $\kappa_{a,\omega,n}\circ w^u_{p}(\tau)$ with $p=p_{a,\omega,n}$. (Notice that  $b\le \konst$ from the assumption.) From the non-integrability condition $(NI)_\rho$, there exists some $b_0>0$ such that the estimate \eqref{eq:tv} holds for any $\psi\in \template$, $\alpha\in \real$ and $b\ge b_0$. Then, by \eqref{eq:psis}, we have    
\begin{equation}\label{eq:conseq_ni_1}
\frac{1}{2h}\left|\int_{-h}^h \exp(ibh^{-1}(\psi^s_{q,(-h,h)}(\tau)+\alpha \tau)) d\tau \right| <b^{-\rho}\qquad\text{for $q\in M$, $\alpha\in \real$ and $b\ge b_0$.}
\end{equation}
Below we consider the relation between $\psi^s_{q,(-h,h)}(\tau)$ and $\theta^s_{a,\omega,n}(\kappa_{a,\omega,n}\circ w^u_{p}(\tau))$. 
Similarly to the argument in the proof of former claims, with setting $e_y=(\kappa_{a,\omega,n}^{-1})_*(\partial_y)$, we find
\begin{align*}
\theta^s_{a,\omega,n}(\kappa_{a,\omega,n}(w^u_q(\tau)))
&=\langle \gamma^s_{q,(-h,h)}(w^u_q(\tau)), e_y(w^u_q(\tau))\rangle\\
&=\psi^s_{q,(-h,h)}(\tau)\cdot 
\langle {\gamma}^{\perp}_{q,(-h,h)}(w^u_q(\tau)), e_y(w^u_q(\tau))\rangle
+
\langle {\gamma}^0_{q,(-h,h)}(w^u_q(\tau)), e_y(w^u_q(\tau))\rangle
\end{align*}
for $-h<\tau<h$. Further we write the first term on the last line as
\begin{align}\label{eq:psis}
&\psi^s_{q,(-h,h)}(0)\cdot 
\langle {\gamma}^{\perp}_{q,(-h,h)}(w^u_q(\tau)), e_y(w^u_q(\tau))\rangle\notag\\
&\qquad +(\psi^s_{q,(-h,h)}(\tau)-\psi^s_{q,(-h,h)}(0))\cdot 
\langle {\gamma}^{\perp}_{q,(-h,h)}(w^u_q(\tau)), e_y(w^u_q(\tau))\rangle.\notag
\end{align}
Notice that the terms $\langle {\gamma}^0_{q,(-h,h)}(w^u_q(\tau)), e_y(w^u_q(\tau))\rangle$ and $\psi^s_{q,(-h,h)}(0)\cdot 
\langle {\gamma}^{\perp}_{q,(-h,h)}(w^u_q(\tau)), e_y(w^u_q(\tau))\rangle$ are of class $C^2$ with respect to $\tau$ (at least) and their second derivatives are bounded by $C_*$, so that we can approximate the sum of them by an affine function, say $\alpha\tau+\beta$, with an error term bounded by $C_* \konst^2 \pomega^{-1}$. For the remaining term on the last line, we have
\[
|\psi^s_{q,(-h,h)}(\tau)-\psi^s_{q,(-h,h)}(0)|\le C_* \konst \pomega^{-1/2} \log \pomega 
\]
from \eqref{eq:regularity_psi_s} and also
\[
\langle {\gamma}^{\perp}_{q,(-h,h)}(w^u_q(\tau)),e_y(w^u_q(\tau))\rangle=c_{a,\omega,n}+\mathcal{O}_* (\konst\pomega^{-1/2})
\]
correspondingly to \eqref{eq:approx_gamma_perp}.
Hence, with some $\alpha', \beta'\in \real$,  it holds   
\[
|\theta^s_{a,\omega,n}(\kappa_{a,\omega,n}\circ w^u_{p}(\tau))-c_{a,\omega,n}\cdot \psi^s_{q,(-h,h)}(\tau)-\alpha'\tau-\beta'|<C_*\konst^2 \pomega^{-1}\log \pomega
\]  
for $-h<\tau<h$. 
Since we assume $b_0\le b\le \konst$ in the claim \eqref{eq:tv2}, we may suppose that the right-hand side above is much smaller than $b^{-1-\rho} h> \konst^{-1-\rho} \pomega^{-1+\theta_*}$ by letting $\omega_\sharp$ be larger if necessary. Therefore we obtain the claim  \eqref{eq:tv2} from \eqref{eq:conseq_ni_1} with $b$ replaced by $c_{a,\omega,n}\cdot  b$, noting that $C_*^{-1}<|c_{a,\omega,n}|<C_*$ from Remark \ref{rem:caon}. 
\end{proof}

\begin{Remark}\label{rem:ext_eawn}
\textcolor{\revisionColor}{
From Lemma \ref{lm:e_approx} and from the definition of $\konst$ in \eqref{eq:konst_def}, we have 
\begin{equation}\label{eq:e_var2}
\|\bDelta_{a,\omega,n}^{-1} (e_{a,\omega,n}(w')- e_{a,\omega,n}(w))\|\le C_* t_\sharp^2 \cdot \langle\omega\rangle^{-1/2}\cdot \langle \langle\omega\rangle^{1/2}\|w'-w\|\rangle
\end{equation}
where $\bDelta_{a,\omega,n}$ is the $2\times 2$ matrix defined by
\begin{equation}\label{eq:bDelta}
\bDelta_{a,\omega,n}=\begin{pmatrix}
\Delta_{a,\omega,n}&0\\
0& 1
\end{pmatrix},\qquad 
\Delta_{a,\omega,n}:=
\langle \Delta(p_{a,\omega,n},\,\konst \pomega^{-1/2})\rangle.
\end{equation}}
(Recall that $\Delta(q,\delta)$ is the approximate non-integrablity defined in Definition \ref{def:approx_ni}.)
For a technical reason to be explained in Remark \ref{Re:ext}, we will need to extend the functions $e_{a,\omega,n}(\cdot)$ to $\real^2$. 
Though the choice of the extension is rather arbitrary, we will suppose that the extension is continuous and satisfies  \eqref{eq:e_var2}  for all $w,w'\in \real^2$. 
\end{Remark}

\subsection{The flow $f^t$ viewed in the local charts}
\label{ss:ftlocal}
In this subsection, we  consider how the flow $f^t$ looks in the local charts $\kappa_{a,\omega,n}$. 
Below we consider $(a,\omega,n)$ and $(a',\omega',n')$ with $a,a'\in A$, $\omega,\omega'\in \integer$ and $n\in \mathcal{N}(a,\omega)$, $n'\in \mathcal{N}(a',\omega')$ and  suppose  
\begin{equation}\label{eq:set}
t_\sharp\le t\le 2t_\sharp\quad\text{and}\quad U=U_{a,\omega,n}\cap f^{-t}(U_{a',\omega',n'})\neq \emptyset.
\end{equation}
Then the flow $f^t$ viewed in the local charts $\kappa_{a,\omega,n}$ and $\kappa_{a',\omega',n'}$ is
\[
f:=\kappa_{a',\omega',n'}\circ f^t\circ \kappa_{a,\omega,n}^{-1}:\kappa_{a,\omega,n}(U)\to \kappa_{a',\omega',n'}(f(U)).
\]
Since $\kappa_{a,\omega,n}$ are flow box local charts,  this diffeomorphism $f$ is written in the form
\begin{equation}\label{eq:expression_f}
f(x,y,z)=(\hat{f}(x,y), z+\check{f}(x,y))
\end{equation}
and therefore we may extend it naturally to
\[
f:V\times \real\to V'\times \real,\qquad\text{with setting $V:=\pi\circ \kappa_{a,\omega,n}(U)$,  \;\;$V':=\pi\circ \kappa_{a',\omega,n}(f(U))$} 
\]
where $\pi:\real^3\to \real^2$ denotes the projection to the first two components.
Before proceeding further, beware that we dropped dependence on $t_\sharp \le t\le 2t_\sharp$ and  $a$, $\omega$, $n$, $a'$, $\omega'$, $n'$ from the notation above. We use this simplified notation only in the following part of this subsection.

Letting $\chi_*$ in \eqref{eq:hyperbolicity} be slightly smaller and  choosing the local charts $\kappa_a$ a little more carefully (see Remark \ref{Rem:intersection_mulitpliicity_rho}), we may and do assume that the diffeomorphism $f$ given as above is uniformly hyperbolic in  the sense that
\begin{equation}\label{eq:fhyp1}
Df_p^*(\mathbf{C}(2))\subset\mathbf{C}(1/2)\quad \text{for $p\in V\times \real$}
\end{equation}
where $\mathbf{C}(\theta)=\{(\xi_x,\xi_y,0)\in \real^3\mid |\xi_y|\le \theta|\xi_x|\}$ and that 
\begin{equation}\label{eq:fhyp2}
\|Df^*_p(v)\|\ge e^{\chi_* t_\sharp} \|v\| \quad \text{if $v\in \mathbf{C}(2)$\quad and \quad}
\|(Df^{-1})^*_{f(p)}(v)\|\ge e^{\chi_* t_\sharp} \|v\| \quad \text{if $v\notin \mathbf{C}(1/2)$.}
\end{equation}
For the higher order derivatives, we will use a crude estimate
\begin{equation}\label{eq:crude_estimate_df}
\|D^k f\|_\infty\le C_* \exp(C_*t_\sharp) \cdot (\log \max\{\pomega, \langle \omega'\rangle\})^2\quad\text{for $k=2, 3$}
\end{equation}
where the last factor appears as a consequence of the composition of $b_{a,\omega,n}$ and $b_{a',\omega',n'}$ in the definition of the local charts $\kappa_{a,\omega,n}$ and $\kappa_{a',\omega',n'}$. 

When $|\omega|$ and $|\omega'|$ are large, the domain $V$ of $f$ is small in the directions transversal to the flow and hence  $f$ will be well approximated by its linearization. In the next lemma, we give a statement along this idea. Let $A:\real^3\to \real^3$ be the mapping defined by
%\footnote{Since we will consider the case where $|\omega|$ is large, we may and do suppose that $f$ is actually defined much larger domain than $V$ containing the origin $(0,0,0)$.} 
\begin{equation}\label{eq:A}
A(x,y,z)=
\big(\lambda x,\tilde{\lambda} y, z+\varpi \cdot (x,y)+\beta xy\big)+f(0,0,0).
\end{equation}
where 
\begin{equation}\label{eq:b}
\lambda=\pm \|D\hat{f}_{\mathbf{0}}(\partial_x)\|,\quad \tilde{\lambda}=\pm\|D\hat{f}_{\mathbf{0}}(\partial_y)\|, \quad \varpi=\left(\partial_x \check{f}(0,0), \partial_y \check{f}(0,0)\right), \quad   
\beta=\partial_{xy} \check{f}(0,0)
\end{equation}
and the signs of $\lambda$ and $\tilde{\lambda}$ are chosen so that $A$ approximates $f$ better at the origin. 
Then we write  the diffeomorphism $f$ as the composition
\begin{equation}\label{eq:AG}
f=A\circ  G\qquad\text{with setting $G=A^{-1}\circ f$. }
\end{equation}
The diffeomorphism $G$ is written in the form 
\begin{equation}\label{eq:Gf}
G(x,y,z)=\big(\hat{G}(x,y), z+\check{G}(x,y)\big).
\end{equation}
In the next lemma, we let  $0<\theta_*<1/2$ be the constant taken just before  the statement of Lemma \ref{lm:e_approx} and let the constant $\omega_\sharp$ taken  in Lemma \ref{lm:e_approx} be larger if necessary.  
\begin{Lemma}\label{lm:abc}
For any $t_\sharp\le t\le 2t_\sharp$ and any combination of $(a,\omega,n)$ and $(a',\omega',n')$  satisfying $|\omega|>\omega_\sharp$ and $1/2\le | \omega'|/|\omega|\le 2$ , we have the following estimates:\\
\noindent
For the diffeomorphism $G:V\times \real \to \hat{G}(V)\times \real$ defined in \eqref{eq:AG}, we have $G(0)=0$ and 
\begin{align}
&\|\mathrm{Id}-D{G}\|_{\infty}< \pomega^{-1/2+\theta_*}\quad\text{and}\quad
\|D^k G\|_\infty< \pomega^{\theta_*}\quad \text{for } k=2,3,
\label{eq:G1}
\intertext{and also}
&\|D\check{G}\|_{\infty}< \pomega^{-1+\theta_*}\quad\text{and}\quad\|D^2\check{G}\|_{\infty}< \pomega^{-1/2+\theta_*}.
\label{eq:G2}
\end{align}
\noindent
For the diffeomorphism $A$ defined in \eqref{eq:A}, we have
\begin{equation}\label{eq:bB}
\|\varpi\|<\pomega^{-1/2+\theta_*}\quad\text{and}\quad
|\beta|<C_* t_\sharp.
\end{equation}
Further,
\begin{equation}\label{eq:est_on_hatf}
\|D\check{f}\|_{\infty}<\pomega^{-1/2+\theta_*}\quad\text{and} \quad\|D^2\check{f}\|_{\infty}<C_* t_\sharp.
\end{equation}
\end{Lemma}
\begin{Remark}\label{Re:aboutG}
If we use the local chart $\kappa_{a,\omega,n}$ and consider the scale $\pomega^{-1}$ (resp.~$\pomega^{-1/2}$) in the $z$ direction (resp.~the $xy$ direction) and if $\omega'=\omega$, the claims \eqref{eq:G1} and \eqref{eq:G2} in the last lemma implies that  the rescaled map of $G$, 
\[
(\tilde{w},\tilde{z})\mapsto \left(\pomega^{1/2}\cdot \hat{G}\big(\pomega^{-1/2}\tilde{w},\pomega^{-1}\tilde{z}\big), \tilde{z}+\pomega \cdot \check{G}(\pomega^{-1/2}\tilde{w})\right),
\]
 tends to identity in $C^3$ sense as $|\omega|\to \infty$, uniformly in  $a$, $n$ and $t_\sharp \le t\le 2t_\sharp$. 
This tells roughly that the non-linearity of $G$ will be negligible when $|\omega|$ is large. Note that this would not be true if we did not put the nonlinear factor $\beta xy$ in $A$. 
We will see that the decomposition \eqref{eq:AG} enables us to concentrate on the essential part $A$ of $f$ by separating the negligible non-linearity in the factor $G$.  
\end{Remark}
\begin{proof}
Since we may choose large $\omega_\sharp$ depending on $t_\sharp$ and $\theta_*$ and since we are assuming $1/2\le | \omega'|/|\omega|\le 2$,  some of the claims are obtained easily by  bounding some factors depending on $t_\sharp$ by $\omega^{\epsilon}_\sharp$ with small $\epsilon>0$.  For instance, the latter claim of \eqref{eq:G1} follows immediately from \eqref{eq:crude_estimate_df}. 
In the following, we argue in parallel for every value of $\theta_*\in (0,1/2)$. Thus, when we refer previous estimates, we will actually refer them for different (or slightly smaller) values of $\theta_*$. 
 
By the estimates \eqref{eq:aprox_l2} and \eqref{eq:wudd} in the beginning of the proof of Lemma \ref{lm:e_approx} and by the corresponding ones for the time reversed flow $f^{-t}$, we obtain that, for $q\in U$,
\begin{align*}
&\|(\kappa_{a,\omega,n}\circ w^u_{q})'(0)-(1,0,0)\|< \pomega^{-1/2+\theta_*},\quad \|(\kappa_{a,\omega,n}\circ w^s_{q})'(0)-(0,1,0)\|<\pomega^{-1/2+\theta_*}.
\intertext{and also}
&\|(\kappa_{a,\omega,n}\circ w^u_{q})''(0)\|< \pomega^{-1/2+\theta_*},\quad
\|(\kappa_{a,\omega,n}\circ w^s_{q})''(0)\|< \pomega^{-1/2+\theta_*}.
\end{align*}
Of course we have the same estimates for the local chart $\kappa_{a',\omega',n'}$.  
Hence, from the fact that $f^t$ preserves the (un)stable manifolds, we obtain 
\[
\|D\check{f}\|_\infty<\pomega^{-1/2+\theta_*},\quad \|\partial_{xx} \check{f}\|_\infty< \pomega^{-1/2+\theta_*},\qquad
\|\partial_{yy}\check{f}\|_\infty < \pomega^{-1/2+\theta_*}
\] 
and also the former claim of \eqref{eq:bB} on $\varpi$. From these and the definitions of $A$ and $G$, it is not difficult to see  
\[
\|(DG)_0-\mathrm{Id}\|<\pomega^{-1/2+\theta_*}\quad\text{and}\quad \|(D\check{G})_0\|<\pomega^{-1+\theta_*}, \quad \|(D^2\check{G})_0\|<\pomega^{-1/2+\theta_*}.
\]
These estimates at the origin and the latter claim of \eqref{eq:G1} that we mentioned in the beginning yield\footnote{Note that, as we mentioned in the beginning, we argue in parallel for different values of $\theta_*$, so that we may and do suppose that the previous estimates are valid for any $\theta_*>0$.}
\[
\|DG-\mathrm{Id}\|_\infty<\pomega^{-1/2+\theta_*},\quad 
\|D\check{G}\|_{\infty}<\pomega^{-1+\theta_*},\quad 
\|D^2\check{G}\|_{\infty}<\pomega^{-1/2+\theta_*}.
\]
To prove the latter claim of \eqref{eq:bB} on $\beta$, we recall the construction of the local chart $\kappa_{a,\omega,n}$. If we ignore the compositions of  $b_{a,\omega,n}$ and $b_{a',\omega',n'}$ in the local charts $\kappa_{a,\omega,n}$ and $\kappa_{a',\omega',n'}$, we can get this claim just by simple application of the chain rule. Then, restoring the compositions of  $b_{a,\omega,n}$ and $b_{a',\omega',n'}$, 
we find the additional term 
\[
(\lambda \tilde{\lambda})\cdot\beta(a',\omega',n')-\beta(a,\omega,n)
= c_{a,\omega,n}\cdot \left(\frac{\beta(a',\omega',n')}{c_{a',\omega',n'}}-
\frac{\beta(a,\omega,n)}{c_{a,\omega,n}}\right)
\]
with an error term bounded by $C_*$. 
By the definition \eqref{def:beta} and Lemma \ref{lm:Tor}, this is bounded by $C_* t_\sharp$, so that we obtain the required estimate. The latter claim of \eqref{eq:est_on_hatf} then follows. \end{proof}

We finish this subsection with the following crude estimate. 
\begin{Lemma}\label{lm:rhot}
If $0\le t\le 2t_{\sharp}$, we have, for integers $k,\ell,m\ge 0$ with $k+\ell\le 3$, that
\[
\|\partial_x^k\partial_y^\ell\partial_z^m (\rho_{a,\omega,n}\cdot (\rho_{a',\omega',n'}\circ f))\|_{\infty}\le C_* \cdot (e^{C_* t_\sharp}\cdot  \konst^{-1}\max\{\pomega, \langle \omega'\rangle\}^{1/2})^{(k+\ell)}.
\]
\end{Lemma}
We omit the proof since it is straightforward. We just note that, in the case $k=\ell=0$, the right hand side in the estimate  above is just $C_*$ and does not depend on the choice of $t_\sharp$. This is because the function $\varrho_{a,\omega,n}$ in the definition of $\rho_{a,\omega,n}$ is the functions of $(x,y)$ and does not depend on $z$.  (Recall that $\rho_a$ are $C^\infty$ in the variable $z$ and we have $\partial_z f\equiv (0,0,1)$.)

%%%%%%%%%%%%%%
\section{The anisotropic Sobolev space}
\label{sec:anisoH}
As the next step towards the proof of Theorem \ref{th:exp}, we introduce the Hilbert space  $\mathcal{H}$, called the anisotropic Sobolev space, and consider the action of the transfer operator $\cL^t$ on it. The argument in this subsection is a modification of that in the previous papers \cite{1311.4932, MR2995886}. 

\subsection{Partial Bargmann transform}\label{ss:pBt}
The purpose of the following construction is to consider the action of the transfer operator $\cL^t$ in the frequency space (that is, through Fourier transform on local charts). But, one because the direction of $E_0^*$ depends on the base point sensitively, we need to consider the action in the real space simultaneously. The partial Bargmann transform, introduced below, meets these demands. We refer \cite[Sec.~4-5]{MR2995886}, \cite[Sec.~4.2-3]{1311.4932} and \cite[Sec.~3-4]{FaureTsujii12} and the references therein for more detailed accounts on the (partial) Bargmann transform (or more generally on wave packet transforms).  

For $(w,\xi,\eta)\in \real^{2+2+1}$ with $w,\xi\in \real^2$ and $\eta\in \real$, we define $
\phi_{w,\xi,\eta}:\real^{3}\to \complex$ 
by 
\[
\phi_{w,\xi,\eta}(w',z')=2^{-3/2} \pi^{-2}\cdot \langle\eta\rangle^{1/2}\cdot \exp
\left(i\eta \cdot z'+i\xi \cdot  (w'-(w/2))-\langle \eta\rangle\cdot  \|w'-w\|^2/2 \right).
\] 
The partial Bargmann transform $
\pBargmann:L^2(\real^{2+1})\to L^2(\real^{2+2+1})$
is defined by
\begin{equation}\label{eq:volume}
\pBargmann u(w,\xi,\eta)=\int \overline{\phi_{w,\xi,\eta}(w',z')}\cdot u(w',z') \, dw'dz'.
\end{equation}
\begin{Remark}
In the above and also henceforth, we write  $\real^{2+1}$ (resp. $\real^{2+2+1}$) for the Euclidean space of dimension $3$ (resp. $5$) equipped with the standard coordinate $(w,z)=(x,y,z)$ (resp. $(w,\xi,\eta)=(x,y,\xi_x,\xi_y,\eta)$) where $w=(x,y)\in \real^2$ and $\xi=(\xi_x,\xi_y)\in \real^2$. We regard $\xi=(\xi_x,\xi_y)$ and $\eta$ as the dual variables of $w=(x,y)$ and $z$ respectively. 
\end{Remark}
The $L^2$-adjoint $
\pBargmann^*: L^2(\real^{2+2+1}) \to L^2(\real^{2+1})$
of the partial Bargmann transform $\pBargmann$ is 
\begin{equation}\label{eq:pBargmannAdj}
\pBargmann^* v(w',z')= \int \phi_{w,\xi,\eta}(w',z')\cdot v(w,\xi,\eta)  
dw d\xi d\eta.
\end{equation}

\begin{Lemma}\cite[Proposition 5.1]{MR2995886}\label{lm:pBargmannL2}
The partial Bargmann transform $\pBargmann$ is an $L^2$-isometric injection and $\pBargmann^*$ is a $L^2$-bounded operator such that $\pBargmann^*\circ \pBargmann =\mathrm{Id}$. 
The composition  
\begin{equation}\label{eq:pBargmannP}
\pBargmannP:=\pBargmann\circ \pBargmann^*:L^2(\real^{2+2+1})\to L^2(\real^{2+2+1})
\end{equation}
is the $L^2$ orthogonal projection onto the image of $\pBargmann$. The (Schwartz) kernel of $\pBargmannP$ is written in the form
\[
K_{\pBargmannP}(w,\xi,\eta;w',\xi',\eta')=\delta(\eta'-\eta)
\exp\left(\frac{i(\xi w'-\xi' w)}{2} -\frac{\langle \eta\rangle \|w-w'\|^2}{4}-\frac{\langle \eta\rangle^{-1} \|\xi-\xi'\|^2}{4}\right).
\]
\end{Lemma}

%\begin{Remark}The relation $\pBargmann^*\circ \pBargmann =\mathrm{Id}$ implies that any function $u\in L^2(\real^{2+1})$ is expressed as an integration of the functions $\phi_{w,\xi,\eta}$ in such a way as 
%\[
%u(w,z)=\frac{1}{2\pi}\int \phi_{w',\xi',\eta'}(w,z)\cdot  \pBargmann u(w',\xi',\eta')   \,
%dw' d\xi' d\eta'.
%\]
%\end{Remark}

\subsection{Decomposition of functions in the phase space}
\label{ss:dc}
We introduce a few $C^\infty$ partitions of unity. Let $\chi:\real\to [0,1]$ be a $C^\infty$ function such that $\chi(s)=1$ if $|s|\le 1$ and $\chi(s)=0$ if $|s|\ge 3/2$, which we have already introduced in Subsection \ref{ss:perturbation_family}.
\begin{enumerate}
\item a partition of unity on the projective space:
$\{\chi_{\sigma}:\mathbb{P}^1\to [0,1]\mid \sigma=+,-\}$ such that
\[
\chi_+([(x,y)])=
\begin{cases}1,&\text{if $|x|\ge 2|y|$;}\\
0,&\text{if $|y|\ge 2|x|$,}
\end{cases}\qquad \text{and }\qquad \chi_-([(x,y)])=1-\chi_+([(x,y)]).
\]
\item a periodic partition of unity on the real line $\real$:\; $\{q_\omega:\real\to [0,1]\mid \omega\in \mathbb{Z}\}$ such that 
\[  
\mathrm{supp}\, q_{\omega}\subset [\omega-1,\omega+1], \quad 
q_\omega(s)=q_0(s-\omega),
\]
\item a Littlewood-Paley type partition of unity:\; $\{\chi_m:\real^2\to [0,1]\mid m\in \mathbb{Z}, m\ge 0\}$ defined by 
\[
\chi_m(w)=
\begin{cases}
\chi(\|w\|),&\quad \text{if $m=0$;}\\
\chi(e^{-m}\|w\|)-\chi(e^{-m+1}\|w\|),&\quad \text{if $m>0$.}
\end{cases}
\]
\end{enumerate}
We define also the (anisotropic) partition of unity
$\{\psi_m:\real^2\to [0,1]\mid m\in \mathbb{Z}\}$ by  
\[
\psi_{m}(x,y)=\chi_{\mathrm{sgn}(m)}([(x,y)])\cdot \chi_{|m|}(x,y)
\]  
where we ignore the first factor on the right-hand side when $m=0$. 

We next introduce partitions of unity on the phase space $\real^{2+2+1}$. 
For $a\in A$, $\omega\in \mathbb{Z}$, $n\in N(a,\omega)$ and $m\in \integer$, we define the function $\psi_{a,\omega,n,m}:\real^{2+2+1}\to [0,1]$ by
\begin{equation}\label{eq:sumq}
\psi_{a,\omega,n,m}(w,\xi,\eta)=
q_{\omega}(\eta)\cdot 
\psi_m\left(\pomega^{-1/2}\cdot \textcolor{\revisionColor}{\bDelta_{a,\omega,n}^{-1}}
 (\xi-\eta\cdot  e_{a,\omega,n}(w))\right)
\end{equation}
where $\bDelta_{a,\omega,n}$ is the $2\times 2$ matrix defined in \eqref{eq:bDelta}.
Then we have, for each $a\in A$ and $n\in \mathcal{N}(a,\omega)$, that
\[
\sum_{m}\psi_{a,\omega,n,m}(w,\xi,\eta)=q_\omega(\eta)\quad\text{and hence}\quad 
\sum_\omega\sum_{m}\psi_{a,\omega,n,m}(w,\xi,\eta)\equiv 1.
\]  
\begin{Remark}
Note the factor $\bDelta_{a,\omega,n}^{-1}$ in the definition \eqref{eq:sumq}, which did not appear in \cite{MR2652469,MR2995886, 1311.4932} when we studied contact Anosov flows using a parallel method.
We put this factor because, as we observed in Lemma \ref{lm:e_approx},  the direction of $E_0^*$ viewed in the local chart $\kappa_{a,\omega,n}$, varies with resect to its base point  at a rate proportional to $\Delta(p_{a,\omega,n}, \konst \pomega^{-1/2})$, which is not uniform in $a$, $\omega$, $n$ and can be as large as $\mathcal{O}_*(\log \langle \omega\rangle)$ in absolute value.  
\end{Remark}

\begin{Remark}\label{Rem:tilde_psi}
For the argument in the proofs in the next section, we define the (enveloping) family of functions 
$\tilde{\psi}_{a,\omega,n,m}:\real^{2+2+1}\to [0,1]$ by 
\begin{equation}\label{eq:tilde_sumq}
\tilde{\psi}_{a,\omega,n,m}(w,\xi,\eta)=
\tilde{q}_{\omega+1}(\eta)\cdot 
\tilde{\psi}_m\left(\pomega^{-1/2}\cdot A_m\cdot \textcolor{\revisionColor}{\bDelta_{a,\omega,n}^{-1}}
 (\xi-\eta\cdot  e_{a,\omega,n}(w))\right)
\end{equation}
where
\begin{equation}\label{eq:tildeq}
\tilde{q}_{\omega}={q}_{\omega-1}+{q}_{\omega}+{q}_{\omega+1},\qquad
\tilde{\psi}_m={\psi}_{m-1}+{\psi}_m+{\psi}_{m+1}
\end{equation}
and
\[
A_m(\xi_x,\xi_y)=\begin{cases}(2\xi_x, \xi_y/2),&\text{if $m>0$;}\\
(\xi_x, \xi_y),&\text{if $m=0$;}\\
(\xi_x/2, 2\xi_y),&\text{if $m>0$.}
\end{cases}
\]
From the definition, we have $\tilde{\psi}_{a,\omega,n,m}=1$ on the support of $\psi_{a,\omega,n,m}$. 
\end{Remark}

For each $C^r$ function $u$ on $M$, we define a family of functions $\hat{u}_{a,\omega,n,m}:\real^{2+2+1}\to \complex$ for $a\in A$, $\omega\in \mathbb{Z}$, $n\in N(a,\omega)$ and $m\in \mathbb{Z}$, by    
\[
\hat{u}_{a,\omega,n,m}(w,\xi,\eta)=
{\psi}_{a,\omega,n,m}(w,\xi,\eta)\cdot \pBargmann(\rho_{a,\omega,n}\cdot (u\circ \kappa_{a,\omega,n}^{-1}))(w,\xi,\eta).
\]
We regard this correspondence $u\mapsto (u_{a,\omega,n,m})$ as an operator
\[
\bI:C^\infty(M)\to \prod_{a,\omega,n,m}C^\infty_0(\supp \psi_{a,\omega,n,m}),\quad \bI(u)=(\hat{u}_{a,\omega,n,m})_{a\in A, \omega\in \mathbb{Z}, n\in N(a,\omega), m\in \mathbb{Z}}.
\]

\begin{Remark}\label{Re:ext}
Since $\pBargmann(\rho_{a,\omega,n}\cdot (u\circ \kappa_{a,\omega,n}^{-1}))$ is real-analytic, its support is $\real^2\oplus\real^3$ unless $u_{a,\omega,n}\equiv 0$.
Thus, in order that the definitions above make good sense, we have to extend the mapping $e_{a,\omega,n}$ from $D_{a,\omega,n}$ to $\real^2$. (Recall Remark \ref{rem:eaon}.)  
Note however that the functions $\pBargmann(\rho_{a,\omega,n}\cdot (u\circ \kappa_{a,\omega,n}^{-1}))$  decays extremely  fast on the outside of $\supp \rho_{a,\omega,n}\times \real^3$, so that we can basically neglect its part on the outside of $D_{a,\omega,n}\times \real^3$. 
\end{Remark}

The next lemma tells that the operator $\bI^*:\bigoplus_{a,\omega,n,m}C^\infty_0(\supp \psi_{a,\omega,n,m})\to C^3(M)$,
\[
\bI^*((u_{a,\omega,n,m})_{a\in A, \omega\in \mathbb{Z}, n\in N(a,\omega), m\in \mathbb{Z}})=\sum_{a,\omega,n,m} \left(\tilde{\rho}_{a,\omega,n}\cdot \pBargmann^*u_{a,\omega,n,m}\right)\circ \kappa_{a,\omega,n}, 
\]
gives a construction reverse to the decomposition in $\bI$. 
\begin{Lemma}
\label{lm:ii} $\bI^*\circ \bI=\mathrm{Id}$ on $C^\infty(M)$.
\end{Lemma}
\begin{proof} The claim is not trivial but can be checked by simple computations. Just note that, since the function $\varrho_{a,\omega,n}$ in the construction of  $\rho_{a,\omega,n}$ in Subsection \ref{ss:lc} does not depend on the variable $z$ and since $\pBargmann \circ  \mathcal{M}(q_{\omega})\circ \pBargmann^*$ is a convolution operator that involves only the $z$-variable, we have the commutative relation 
\[
(\pBargmann \circ  \mathcal{M}(q_{\omega})\circ \pBargmann^*)\circ \mathcal{M}(\varrho_{a,\omega,n})=
\mathcal{M}(\varrho_{a,\omega,n})\circ( \pBargmann \circ  \mathcal{M}(q_{\omega})\circ \pBargmann^*)
\]
where $\mathcal{M}(\varphi)$ denotes the multiplication operator by $\varphi$. 
We refer \cite[Lemma 6.5]{1311.4932} for the detail. 
\end{proof}

We can now define the Hilbert space $\mathcal{H}$ of distributions. We henceforth fix $\alpha_0\in (0,1/6)$. To simplify the notation, we set
\[
\cJ=\{(a,\omega, n,m)\mid a\in A, \omega\in \mathbb{Z}, n\in N(a,\omega), m\in \mathbb{Z}\}
\] 
and refer the components of $\bj=(a,\omega, n,m)\in \cJ$ as $a(\mathbf{j})=a$, $\omega(\bj)=\omega$ and so on. (Of course, $\cJ$ here is different from  that in Section \ref{sec:pr_od}.) Also, for $\bj=(a,\omega, n,m)\in \cJ$, we set 
\begin{equation}\label{eq:convention}
\kappa_{\bj}:=\kappa_{a,\omega,n},\quad 
\rho_{\bj}:=\rho_{a,\omega,n}, \quad\psi_{\bj}:=\psi_{a,\omega,n,m}, \quad \Delta_{\bj}:= \Delta_{a,\omega,n}, \quad \bDelta_{\bj}:= \bDelta_{a,\omega,n}
\end{equation}
and so on. It will be useful to remember that the components $a$ and $n$ are related to the position, $\omega$ to the frequency in the flow direction and $m$ to the frequency in the directions transversal to the subspace $E_0^*$. 

\begin{Definition}
We define $\bbH$ as the Hilbert space obtained as the completion of the direct sum $\bigoplus_{\bj\in \cJ} L^2(\supp \psi_{\bj})$ with respect to the norm
\begin{equation}\label{eq:def_normH}
\|(u_{\bj})_{\bj\in \cJ}\|_{\bbH}=\left(\sum_{\bj\in \cJ} e^{\alpha_0 \cdot m(\bj)}\|u_{\bj}\|_{L^2}^2\right).
\end{equation} 
We define $\mathcal{H}$ as the Hilbert space of distributions on $M$ that is obtained as the completion of 
$C^\infty(M)$ with respect to the norm $\|u\|_{\mathcal{H}}=\|\bI(u)\|_{\bbH}$. Then we have\footnote{For this relation, we refer \cite[Ch.1]{NicolaBook} and  \cite[Appendix A]{Taylor}. }
\begin{equation}\label{eq:H_inclusion}
C^{\alpha_0}(M)\subset H^{\alpha_0}(M)\subset \mathcal{H} \subset H^{-\alpha_0}(M)\subset(C^{\alpha_0}(M))'
\end{equation}
where $H^{r}(M)$ denotes the Sobolev space of order $r$.   
By definition, the operator $\bI$ extends to an isometric injection $\bI:\mathcal{H}\to \bbH$. 
\end{Definition}

We define the operator $\bbL^t$ formally
 by $\bbL^t=\bI\circ \cL^t\circ \bI^*$, so that the following diagram commutes:
\[
\begin{CD}
\bbH@>{\bbL^t}>> \bbH\\
@A{\bI}AA @A{\bI}AA\\
\mathcal{H}@>{\cL^t}>> \mathcal{H}
\end{CD}
\]
\begin{Remark}
At this moment, we only know that the operator $\bbL^t$ is defined  
as an operator from $\bigoplus_{\bj}C^\infty_0(\supp \psi_{\bj})$ to 
$\prod_{\bj}C^\infty_0(\supp \psi_{\bj})$. We will see that it extends naturally to a bounded operator on $\bbH$ and consequently that $\mathcal{L}^t$ extends to a bounded operator on $\mathcal{H}$ when $t\ge t_\sharp$.
\end{Remark}

\section{Proof of Theorem \ref{th:exp}}
\label{sec:proof1}
We henceforth assume that $f^t\in \fF^3_A$ satisfies  the non-integrability condition $(NI)_\rho$ for some $\rho>0$ and suppose $t_{\sharp}\le t\le 2t_{\sharp}$. 
Most part of the argument below is devoted to show that $f^t$ itself is exponentially mixing. 
In the last subsection, we complete the proof of Theorem \ref{th:exp} by examining dependence of the argument on the flow $f^t$. For this last part of the argument, we emphasize at this moment that, for the proof of exponential mixing for $f^t$, we actually need the estimate \eqref{eq:tv} in the non-integrability condition $(NI)_\rho$ only for $b$ in some bounded range. (See Remark \ref{rem:perturb}.) 
This is crucial when we prove stability of exponential mixing. 

In the following, we suppose that $\omega_\sharp>\omega_0$ is the constant in Lemma \ref{lm:abc}, but will let it be larger if necessary. 
We will also introduce a large constant $m_{\sharp}>0$ depending on $t_\sharp$ and $\omega_\sharp$. Below we will  ignore some absolute constants, such as $2\pi$,  that appear in Fourier transform and partial Bargmann transform, since they are not essential at all in our argument.

\subsection{Estimates on the components of $\bbL^t$}
Below we will use the notations prepared in the last section, especially \eqref{eq:convention}. 
We write $\bbL^t_{\bj\to\bj'}:C^\infty_0(\supp \psi_{\bj})\to C^\infty_0(\supp \psi_{\bj'})$ for the component of $\bbL^t$ that sends the $\bj$-component to the $\bj'$-component. It is written as \begin{align}
&\bbL_{\bj\to \bj'}^t u=\mathcal{M}({\psi}_{\bj'})\circ\pBargmann\circ \cL^t_{\bj\to\bj'}\circ  \pBargmann^*  u \label{def:bbL}
\intertext{
where $\cL^t_{\bj\to \bj'}$ is the transfer operator on the local charts defined by}
&\cL^t_{\bj\to \bj'} v=({\rho}^t_{\bj\to \bj'} \cdot v) \circ (f^t_{\bj\to \bj'})^{-1}
\end{align}
with setting
\[
f^t_{\bj\to \bj'}=\kappa_{\bj'}\circ f^t\circ \kappa_{\bj}^{-1}
\quad\text{and}\quad
\rho^t_{\bj\to \bj'}=(\rho_{\bj'}\circ f^t_{\bj\to \bj'})\cdot \tilde{\rho}_{\bj}.
\]
As we noted in Subsection~\ref{ss:ftlocal}, the diffeomorphism $f^t_{\bj\to \bj'}$ is written 
\[
f^t_{\bj\to \bj'}(x,y,z)=(\hat{f}^t_{\bj\to \bj'}(x,y), z+\check{f}^t_{\bj\to \bj'}(x,y)).
\]
This extends naturally to
\[
f^t_{\bj\to \bj'}:V^t_{\bj\to \bj'}\times \real\to \widetilde{V}^t_{\bj\to \bj'}\times \real
\]
where, with setting $U_{\bj}:=U_{a(\bj), \omega(\bj), n(\bj)}$, we define 
\[
V^t_{\bj\to \bj'}=\pi\circ \kappa_{\bj}(U_{\bj}\cap f^{-t}(U_{\bj'}))\quad\text{and}\quad
\widetilde{V}^t_{\bj\to \bj'}=\hat{f}^t_{\bj\to \bj'}({V}^t_{\bj\to \bj'}).
\]
Since the differential of $f^{t}_{\bj\to \bj'}$ at a point $(w,z)\in V^t_{\bj\to \bj'}$ does not depend on the variable $z$, we will write $(Df^t_{\bj\to \bj'})_w$ for it. 
The natural action of $f^t_{\bj\to \bj'}$ on the cotangent bundle is written 
\[
(Df^t_{\bj\to \bj'})^*:\widetilde{V}^t_{\bj\to \bj'}\times \real^3\to V^t_{\bj\to \bj'}\times \real^3,\quad 
(Df^t_{\bj\to \bj'})^*(w,\xi,\eta)=(w', (Df^t_{\bj\to \bj'})^*_{w'}(\xi,\eta))
\]
where $w'=(\hat{f}^t_{\bj\to \bj'})^{-1}(w)$. 

From the definitions of the partial Bargmann transform $\pBargmannP$ and its adjoint $\pBargmannP^*$ given in Subsection~\ref{ss:pBt}, 
the operator $\bbL_{\bj\to \bj'}^t$ is written as an integral operator with smooth rapidly decaying kernel
\begin{equation}\label{eq:kernel}
K(w,\xi,\eta;w',\xi',\eta')={\psi}_{\bj'}(w',\xi',\eta')\cdot \int  
({\rho}_{\bj\to \bj'}\cdot \phi_{w,\xi,\eta})\circ (f^t_{\bj\to \bj'})^{-1}(z)\cdot  \overline{\phi_{w',\xi',\eta'}(z)}\, 
dz,
\end{equation}
so that it is a compact  operator from $L^2(\supp \psi_{\bj})$ to $L^2(\supp \psi_{\bj'})$. 

We take a large constant $m_\sharp>0$, which will be specified in the course of the argument, and let $K:\bbH\to \bbH$ be the part of the operator $\bbL^t$ that consists of the components $\bbL_{\bj\to \bj'}^t$ with 
\begin{equation}
\label{ass:momega}
\max\{|\omega(\bj)|,|\omega(\bj')|\}\le \omega_\sharp\quad\text{and}\quad \max\{|m(\bj)|,|m(\bj')|\}\le m_\sharp.
\end{equation}
This operator $K$ consists of finitely many non-vanishing components and therefore compact regardless of the choice of $\omega_\sharp$ and $m_\sharp$. 
Let $\Pi_{\omega}:\bbH\to \bbH$ be the projection operator that extracts the components with $\omega(\bj)=\omega$.
We are going to prove the following proposition.
\begin{Proposition}\label{pp:main_est}
There exists a constant $c>0$ (independent of the choice of $t_\sharp$) such that 
\[
\|\Pi_{\omega'}\circ (\bbL^t-K) \circ \Pi_{\omega}:\bbH\to \bbH\|\le \exp(-c t)\cdot  \langle \omega'-\omega\rangle^{-1}\quad
\text{for $\omega,\omega'\in \integer$ and $t_\sharp\le t\le 2t_\sharp$.}
\]
\end{Proposition}
This proposition implies that $f^t$ is exponentially mixing. Indeed, from the proposition, we have $\|\bbL^t-K:\bbH\to \bbH\|<e^{-(c/2)t}$ for $t_\sharp\le t\le 2t_\sharp$, by letting $t_\sharp$ be larger if necessary. Since $K$ is compact as we noted above, the essential spectral radius of $\bbL^t$ is bounded by $e^{-(c/2)t}$ and so is that of $\cL^t:\mathcal{H}\to \mathcal{H}$. Since $f^t$ is mixing\footnote{It is easy to see that $(NI)_\rho$ implies joint non-integrability of the stable and unstable foliations and hence $f^t$ is stably mixing. See the argument in the proof of Proposition \ref{pp:main_est}.},  there is a unique eigenvalue $1$ on the region $|z|\ge 1$, which is simple and the corresponding spectral projector is the averaging with respect to the volume $\vol$; The other part of the spectrum is contained in the region $|z|<e^{-c't}<1$ for some $c'>0$. Therefore, letting $\mathcal{H}_0=\{ u\in \mathcal{H}\mid \int u\, d\vol=0\}$, we have 
\[
\|\cL^t:\mathcal{H}_0\to \mathcal{H}_0\|\le Ce^{-c't}\quad\text{for $t\ge 0$.}
\]   
This and \eqref{eq:H_inclusion} give the required decay estimate:
\[
\left|\int \varphi \cdot(\psi\circ f^t) \, d\vol\right|
=\left|\int  \psi \cdot \mathcal{L}^t \varphi\, d\vol\right|
\le \|\psi\|_{\mathcal{H}'}\cdot \|\mathcal{L}^t\varphi\|_{\mathcal{H}}
\le Ce^{-c't}\cdot \|\psi\|_{C^{\alpha_0}(M)}\cdot \|\varphi\|_{C^{\alpha_0}(M)} 
\]
for $\varphi, \psi\in C^{\alpha_0}(M)$ with $\int \varphi \,d\vol =0$.

\subsection{Estimates on components $\bbL^t_{\bj\to \bj'}$} 
\label{ss:est_comp}
In this subsection, we present a few statements on the components $\bbL^t_{\bj\to \bj'}$ \emph{with respect to the $L^2$ norm} and deduce Proposition \ref{pp:main_est} from them. The proofs of the estimates (precisely, Lemma \ref{lm:basic_components}, Lemma \ref{lm:hook} and Proposition \ref{prop:central_components}) are deferred to the subsections that follow. Since our task is the proof of Proposition \ref{pp:main_est}, we will disregard the components $\bbL^t_{\bj\to \bj'}$ (or suppose $\bbL^t_{\bj\to \bj'}=0$) for which  \eqref{ass:momega} holds.  

We begin with a few simple estimates and then proceed to more involved ones. The most important statement is Proposition \ref{prop:central_components}, in which we give a consequence of the non-integrability condition $(NI)_\rho$. 
First of all, we note that 
\begin{equation}
\|\bbL^t_{\bj\to \bj'}\|_{L^2}\le 1\quad \text{ for any $t\ge 0$ and $\bj,\bj'\in \cJ$.} 
\end{equation}
This is obvious because neither of $\pBargmann$, $\pBargmann^*$ and $\cL^t$ increases the $L^2$ norm. 

Observe that the operator $\bbL^t_{\bj\to \bj'}$ is localized in the space from the expression \eqref{eq:kernel} of its kernel. 
In order to give  quantitative estimates related to this observation, we introduce a few definitions. Let $\pi:\real^{2+1}\to \real^2$ and $\tilde{\pi}:\real^{2+2+1}\to \real^2$ be the projections to the first two components, that is, we set  $\pi(w,z)=w$, $\tilde{\pi}(w,\xi,\eta)= w$. 
In order to cut off the tail part of the operator $\bbL^t_{\bj\to \bj'}$, we introduce the $C^\infty$ function 
\[
\mathcal{X}_{\bj\to \bj'}:\real^2\to [0,1]\quad \text{ (\;resp. }\mathcal{X}'_{\bj\to \bj'}:\real^2\to [0,1]\;)
\]
so that it takes the constant value $1$ on the $\konst^{1/2}\langle \omega(\bj)\rangle^{-1/2}$-neighborhood of $\pi(\supp \rho_{\bj\to \bj'})$ (resp. $\pi(f^t_{\bj\to \bj'}(\supp \rho_{\bj\to \bj'}))$),  
while it is supported in the  $2\konst^{1/2}\langle \omega(\bj)\rangle^{-1/2}$-neighborhood of the same subset. 
Further we may and do suppose that 
\begin{equation}\label{eq:Dalpha}
\|D^\alpha \mathcal{X}_{\bj\to \bj'}\|
\le C_*(\alpha) \cdot (\konst^{-1/2}\langle \omega(\bj)\rangle^{1/2})^{|\alpha|},\quad
\|D^\alpha \mathcal{X}'_{\bj\to \bj'}\|
\le C_*(\alpha)\cdot  (\konst^{-1/2}\langle \omega(\bj)\rangle^{1/2})^{|\alpha|}
\end{equation}
for any multi-index $\alpha$.
For brevity, we will write $\mathcal{X}_{\bj\to \bj'}$ and $\mathcal{X}'_{\bj\to \bj'}$ also for the functions $\mathcal{X}_{\bj\to \bj'}\circ \tilde{\pi}$ and $\mathcal{X}'_{\bj\to \bj'}\circ \tilde{\pi}$ on $\real^{2+2+1}$, abusing the notation slightly. With this convention, we  define 
\begin{equation}\label{eq:defhatbbL}
\hat{\bbL}^t_{\bj\to \bj'}: L^2(\supp \psi_{\bj})\to L^2(\supp \psi_{\bj'}), \quad
\hat{\bbL}^t_{\bj\to \bj'}u=\mathcal{X}'_{\bj\to \bj'}\cdot \bbL^t_{\bj\to \bj'}(\mathcal{X}_{\bj\to \bj'}\cdot u).
\end{equation}
\begin{Remark}\label{Rem:intersection_multiplicity}
From Remark \ref{Rem:intersection_mulitpliicity_rho}, we may assume that, for any $\omega,\omega'\in \integer$ and $m,m'\in \integer$ and for each $\bj\in \cJ$ with $\omega(\bj)=\omega$ (resp. $\bj'\in \cJ$ with $\omega(\bj')=\omega'$), the intersection multiplicity of 
\[
\{ \supp \mathcal{X}_{\bj\to \bj'}\mid \bj'\in \cJ, \omega(\bj')=\omega', m(\bj')=m'\}\;\;
\text{( resp.}
\{ \supp \mathcal{X}'_{\bj\to \bj'}\mid \bj\in \cJ, \omega(\bj)=\omega, m(\bj)=m\}
\text{)}
\]
is bounded by an absolute constant. 
\end{Remark}

From the observation on $\bbL^t_{\bj\to \bj'}$ mentioned above, it
 is not difficult to get the estimate 
\begin{equation}\label{eq:bbLjj_approx}
\|\bbL^t_{\bj\to \bj'}-\hat{\bbL}^t_{\bj\to \bj'} \|_{L^2}\le C_*(\nu) \cdot \konst^{-\nu} \quad \text{ for $t_\sharp\le t\le 2 t_\sharp$ and $\bj,\bj'\in \cJ$} 
\end{equation}
for arbitrarily large $\nu>0$.
Actually the next lemma gives a little more precise estimates, making use of the fact that the flow $f^t$ viewed in our local charts is just a translation in each of the flow lines. (The proof is not difficult but deferred to Subsection \ref{ss:pf_lemmas}.)
\begin{Lemma}\label{lm:basic_components}
For any $\nu>0$, there exists a constant $C_*({\nu})>0$ such that 
\begin{align*}
&\|\bbL^t_{\bj\to \bj'}\|_{L^2}\le C_*({\nu})\cdot  \langle \omega(\bj)- \omega(\bj')\rangle^{-\nu}
\intertext{
and further}
&\|\bbL^t_{\bj\to \bj'}-\hat{\bbL}^t_{\bj\to \bj'}\|_{L^2}\le C_{*}(\nu)\cdot \konst^{-\nu}\cdot  \langle \omega(\bj)- \omega(\bj')\rangle^{-\nu}
\end{align*}
for  $t_\sharp \le t\le 2t_\sharp$ and $\bj,\bj'\in \cJ$.
\end{Lemma} 
%We omit the proof of this lemma because the required estimates are obtained without difficulty if one notes that the partial Bargmann transform $\pBargmann$ is just the Fourier transform in the $z$ variable combined with the Bargmann transform\cite{MR1872698} in $w=(x,y)$ with some scaling depending on the frequency in $z$. (We refer \cite[Proof of Lemma 9.8]{1311.4932} for more detail, where a slightly different estimates are proved but the argument is completely parallel.)

Next we give estimates on $\bbL^t$ obtained as consequences of the hyperbolic properties \eqref{eq:fhyp1} and \eqref{eq:fhyp2} of the flow~$f^t$. We first introduce the following definition. This definition is motivated by  a simple geometric observation on the  position of $(Df^t_{\bj\to \bj'})^*(\supp \tilde{\psi}_{\bj'})$ relative to  $\supp \tilde{\psi}_{\bj}$ in the phase space. (See Remark \ref{rem:hookarrow} below.)
\begin{Definition}\label{def:arrow} 
For two pairs $(m,\omega)$ and $(m',\omega')$ of integers and a positive real number $t>0$, we write 
$(m,\omega)\hookrightarrow^t (m',\omega')$  if either 
\textcolor{\revisionColor}{
\begin{enumerate}
\item $m\ge 0$ and $m'\le  0$, or
\item $m>0$, $m'>0$ and $m'
\le m -[\chi_*t/2]+|\log (\langle \omega'\rangle/\langle\omega\rangle)|+10$, or
\item $m<0$, $m'<0$ and $m'
\le m -[\chi_*t/2]+|\log (\langle \omega'\rangle/\langle\omega\rangle)|+10$.
\end{enumerate}
}
We write$(m,\omega)\not\hookrightarrow^t (m',\omega')$ otherwise. 
We write  $\bj \hookrightarrow^t\bj'$ (resp. $\bj \not\hookrightarrow^t\bj'$) for  $(\bj,\bj')\in \cJ\times \cJ$ if and only if $(m(\bj),\omega(\bj)) \hookrightarrow^t  (m(\bj'), \omega(\bj'))$ (resp. $(m(\bj),\omega(\bj)) \not\hookrightarrow^t  (m(\bj'), \omega(\bj'))$).
\end{Definition}

\begin{Remark}\label{rem:hookarrow}
The definition above is given so that subsets $(Df^t_{\bj\to \bj'})^*(\supp \tilde{\psi}_{\bj'})$ and $\supp \tilde{\psi}_{\bj}$ are separated in the case $\bj \not\hookrightarrow^t\bj'$. The terms  $|\log (\langle \omega'\rangle/\langle\omega\rangle)|$ in the conditions (2) and (3) above are put in order to deal with technical problems in the case where the ratio between $\langle \omega(\bj')\rangle$ and $\langle \omega(\bj)\rangle$ is not close to $1$. (But such technical problems will turn out to be far from essential.) 
At this moment, we ask the readers to observe that disjointness between $(Df^t_{\bj\to \bj'})^*(\supp \tilde{\psi}_{\bj'})$ and $\supp \tilde{\psi}_{\bj}$ would follow from the condition 
$\bj \not\hookrightarrow^t\bj'$ by simple geometric argument,  if we assumed $1/2\le \langle \omega(\bj')\rangle/\langle\omega(\bj)\rangle\le 2$ and  ignored the factor $\bDelta_{\bj}$ and the variation of $e_{\bj}$ in the definition of the function $\tilde{\psi}_{\bj}$.
In the next lemma, we give a related more quantitative estimate.
\end{Remark}

We henceforth consider two small constants 
\[
0<\delta_*<\rho_*
\]
whose choices are independent of $t_\sharp$, $\omega_\sharp$, $m_\sharp$ and made later in Lemma \ref{prop:central_components}. In the lemma below (and henceforth), the constants $t_\sharp$ and $\omega_\sharp$ are suppose to be taken according to the choice of $\delta_*$ and $\rho_*$. 
Let us recall the definition of $\bDelta_{\bj}$ from \eqref{eq:bDelta}.
\begin{Lemma}\label{lm:hook} There exists a constant $C_*>0$ such that, if\/ $\bj,\bj'\in \cJ$ satisfy $\bj \not\hookrightarrow^t \bj'$ for $t_\sharp\le t\le 2t_\sharp$ and if we have in addition that  
\begin{equation}\label{eq:m_not_small}
|m(\bj)|\ge  \delta_* t_\sharp \quad\text{ (resp. $|m(\bj')|\ge  \delta_* t_\sharp$)}
\end{equation}
%and
%\begin{equation}\label{eq:omegaomegad}
%\max\{|\omega(\bj)|,|\omega(\bj')|\}\ge \omega_\sharp/2,\qquad | \omega(\bj')-\omega(\bj)| \le  |\omega(\bj)|^{1/2},\end{equation}
then we have 
\begin{equation}\label{eq:hook}
\langle \langle \omega(\bj)\rangle^{1/2}\|w'-w\|\rangle^2 \cdot
\left\langle \langle \omega(\bj)\rangle^{-1/2}\|\textcolor{\revisionColor}{\bDelta_{\bj}^{-1}}(\xi'-\xi)\|\right\rangle \ge C_*^{-1} e^{|m(\bj)| }
\end{equation}
for  $(w,\xi,\eta)\in \supp \tilde{\psi}_{\bj}$ and $(w',\xi',\eta')\in (Df^t_{\bj\to \bj'})^*(\supp \tilde{\psi}_{\bj'})$ with $w'\in V^t_{\bj\to \bj'}$ ( resp. 
\begin{equation}\label{eq:hookrev}
\langle \langle \omega(\bj')\rangle^{1/2}\|w'-w\|\rangle^2 \cdot
\left\langle \langle \omega(\bj')\rangle^{-1/2}\|\textcolor{\revisionColor}{\bDelta_{\bj'}^{-1}}(\xi'-\xi)\|\right\rangle \ge C_*^{-1} e^{|m(\bj')| }
\end{equation}
for $(w,\xi,\eta)\in ((Df^t_{\bj\to \bj'})^*)^{-1}(\supp \tilde{\psi}_{\bj})$ and $(w',\xi',\eta')\in \supp \tilde{\psi}_{\bj'}$ with $w\in f^t_{\bj\to \bj'}(V^t_{\bj\to \bj'})$.)

\end{Lemma}
\begin{proof}
We are going to prove the former claim \eqref{eq:hook}. The other claim \eqref{eq:hookrev} is proved in a parallel manner replacing $f^t_{\bj\to \bj'}$ by its inverse.
First of all, note that the ratio between 
$\Delta_{\bj'}\ge 1$ and $ \Delta_{\bj}\ge 1$ is bounded by \textcolor{\revisionColor}{$C_*(t_\sharp+|\log \langle \omega(\bj')\rangle/\langle\omega(\bj)\rangle|)$} because so is $|\Delta_{\bj'}-\Delta_{\bj}|$ from  Lemma \ref{lm:Tor}. Also note that,  from the assumption \eqref{eq:m_not_small}, we may and do suppose 
\begin{equation}\label{eq:mmlarge}
e^{|m(\bj)|}\ge  e^{\delta_* t_\sharp}\gg t_\sharp^4.
\end{equation}

To fix ideas, let us start with considering the case $w=w'$. 
In this case the points $(w,\xi,\eta)$ and $(w',\xi',\eta')$ belongs to the same cotangent space $\{w\}\times \real^3$ and we have 
\[
(Df^t_{\bj\to \bj'})^*(e_{\bj'}(\hat{f}^t_{\bj\to \bj'}(w)))=e_{\bj}(w).
\]
By geometric consideration using \eqref{eq:fhyp1} and \eqref{eq:fhyp2}, it is easy to see that images of the subsets 
\[
\supp \tilde{\psi}_{\bj} \cap (\{w\}\times \real^3)
\quad \text{and}\quad 
(Df^t_{\bj\to \bj'})^*(\supp \tilde{\psi}_{\bj'}) \cap (\{w\}\times \real^3)
\]
projected to the $\xi$-plane along the direction of $e_{\bj}(w)$ and \textcolor{\revisionColor}{mapped by $\bDelta_{\bj}$} are separated by the distance 
$C_*^{-1} e^{|m(\bj)|}\cdot \langle \omega(\bj)\rangle^{1/2}$
at least and hence the claim \eqref{eq:hook} holds in this case. (To see this, we first consider the simple case where $1/2\le \langle \omega(\bj')\rangle/\langle\omega(\bj)\rangle\le 2$ and then note that,  in the other case, \textcolor{\revisionColor}{the term $|\log (\langle \omega'\rangle/\langle\omega\rangle)|$ in Definition \ref{def:arrow} and also} the estimate on the ratio $\Delta_{\bj'}/\Delta_{\bj}$ mentioned above help.)
Next we extend this estimate to the case $w'\neq w$. To this end, we have to consider the difference between $\eta\cdot e_{\bj}(w)$ and $\eta\cdot e_{\bj}(w')$.
If  $w$ and $w'$ are so close that $
\langle \omega(\bj)\rangle^{1/2}\|w'-w\| \le c_* e^{|m(\bj)|/2}$ with sufficiently small  $c_*>0$, 
we have  
\[
  |\textcolor{\revisionColor}{\bDelta_{\bj}^{-1}} (\eta \cdot e_{\bj}(w')-\eta\cdot e_{\bj}(w))|\ll   e^{|m(\bj)|}\cdot \langle \omega(\bj)\rangle^{1/2}
\]
from \eqref{eq:e_var2} and \eqref{eq:mmlarge} and, therefore,  \eqref{eq:hook} remains valid.  
Otherwise, the required estimate  \eqref{eq:hook} is trivial because of the first factor on its left-hand side.  
\end{proof}

The next lemma is a consequence of the observation made in the last lemma. The proof will be given in Subsection \ref{lm:hyp_component}.
\begin{Lemma}\label{lm:hyp_component}
For any $\nu>0$, there exists a constant $C_*({\nu})>0$, which is  independent of $t_\sharp$,  such that, for $t_{\sharp}\le t\le 2t_{\sharp}$ and $\bj,\bj'\in \cJ$ satisfying $\bj\not \hookrightarrow^t \bj'$ and \eqref{eq:m_not_small}, we have 
\begin{equation}\label{claim:hyp_component1}
\|\bbL^t_{\bj\to\bj'}\|_{L^2}\le C_{*}({\nu})\cdot  e^{-\max\{|m(\bj)|, |m(\bj')|\}/2}\cdot \langle \omega(\bj')-\omega(\bj)\rangle^{-\nu}
\end{equation}
and further
\[
\|\bbL^t_{\bj\to\bj'}-\hat{\bbL}^t_{\bj\to \bj'}\|_{L^2}\le C_{*}({\nu})\cdot  e^{-\max\{|m(\bj)|, |m(\bj')|\}/2}\cdot\konst^{-\nu}\cdot  \langle \omega(\bj)- \omega(\bj')\rangle^{-\nu}
.
\]
\end{Lemma}

The next proposition is the key to the proof of Proposition \ref{pp:main_est}, which gives an estimate on the components $\bbL^t_{\bj\to \bj'}$ for which $\max\{|m(\bj)|,|m(\bj')|\}$ is  relatively small. The proof of this proposition is the main ingredient of this section and will be given in the next subsection. 

\begin{Proposition}\label{prop:central_components}
There exist constants $0<\delta_*<\rho_*$ such that, if we let the constants  $t_\sharp$ and $\omega_\sharp$ be large depending on $\delta_*$, then, for all \/ $\bj,\bj'\in \cJ$ satisfying 
\begin{equation}\label{eq:ass_prop_central}
\max\{|m(\bj)|, |m(\bj')|\}\le \delta_* t_\sharp,\quad |\omega(\bj)|\ge \omega_\sharp/2 \quad\text{ and }\quad | \omega(\bj')-\omega(\bj)|\le \exp(\rho_* t_\sharp/10),
\end{equation}
we have 
\[
\|\hat{\bbL}^t_{\bj\to \bj'}\|_{L^2}\le  \exp(-\rho_* t_\sharp)\quad\text{ for }t_{\sharp}\le t\le 2t_{\sharp}.
\]
\end{Proposition} 
Now we deduce Proposition~\ref{pp:main_est} from the estimates on the norms of the components $\bbL^t_{\bj\to \bj'}$ given  Lemma~\ref{lm:basic_components}, Lemma \ref{lm:hyp_component} and Proposition \ref{prop:central_components}.
\begin{proof}[Proof of Proposition \ref{pp:main_est}] 
Let $0<\delta_*<\rho_*$ be those constants in Proposition \ref{prop:central_components}. Note that we may and do suppose that $\delta_*$ is much smaller than $\rho_*$, because the claims of Proposition \ref{prop:central_components} remains valid when we let $\delta_*$ be smaller (and Lemma \ref{lm:hook} holds for any choice $\delta_*$). 
Below we suppose $t_\sharp\le t\le 2t_\sharp$ and proceed with the assumption  
\begin{equation}\label{eq:omegaomegadash}
|\omega|\ge \omega_\sharp/2\quad\text{and}\quad |\omega'-\omega|<\exp(\alpha_0 \delta_* t_\sharp/10)
\end{equation}
for $\omega,\omega'\in \integer$, where $\alpha_0\in (0,1/6)$ is that in the definition of the anisotropic Sobolev space~$\mathcal{H}$.  
The argument in the remaining case is much simpler and will be mentioned at the end. 

Let us take $m,m'\in \integer$ and consider the components $\bbL^t_{\bj\to \bj'}$ for $\bj,\bj'\in \cJ$ satisfying 
\begin{equation}\label{eq:condmomega}
\omega(\bj)=\omega, \quad m(\bj)=m,\quad \omega(\bj')=\omega',\quad m(\bj')=m'.
\end{equation}
Note that, for each $\bj$ (resp. $\bj'\in \cJ$), the cardinality of the set 
\begin{equation*}
\{ \bj'\in \cJ \mid \rho^t_{\bj\to \bj'}\neq 0, \omega(\bj')=\omega', m(\bj')=m'\}\quad\text{(resp. }\{ \bj\in \cJ \mid \rho^t_{\bj\to \bj'}\neq 0, \omega(\bj)=\omega, m(\bj)=m\}\text{)}
\end{equation*}
 may be large ({\it i.e.} grow exponentially with respect to $t$). This causes a problem when we sum the estimates for the components $\bbL^t_{\bj\to \bj'}$ for $\bj,\bj'\in \cJ$. 
Our idea to do with this problem is as follows:
\begin{itemize}
\item[(A)] if we consider the operator $\hat{\bbL}^t_{\bj\to \bj'}$ defined in \eqref{eq:defhatbbL} instead of ${\bbL}^t_{\bj\to \bj'}$, we will not have this problem by virtue of the assumption noted in Remark \ref{Rem:intersection_multiplicity}, and 
\item[(B)] the norm $\|\hat{\bbL}^t_{\bj\to \bj'}-{\bbL}^t_{\bj\to \bj'}\|_{L^2}$ of the difference is very small and dominates the cardinalities of the sets above,  which is bounded by $C_*\exp(C_* t_\sharp)\ll \konst$. 
\end{itemize} 
Below we consider the following three cases for the combination  $(m,m')\in \integer^2$, but note that we continue to assume \eqref{eq:omegaomegadash} for $(\omega,\omega')\in \integer^2$:
\begin{itemize}
\item[(i)] those $(m,m')$ satisfying $\max\{|m|, |m'|\}\le \delta_*t_\sharp$, 
\item[(ii)] those $(m,m')$ not in (i), but satisfies $(m,\omega) \not\hookrightarrow^t (m',\omega')$,  or
\item[(iii)] those $(m,m')$ not either in (i) and (ii).
\end{itemize}
We first consider the case (i). If we consider $\hat{\bbL}^t_{\bj\to \bj'}$ in the places of ${\bbL}^t_{\bj\to \bj'}$, then, by Proposition \ref{prop:central_components} and the idea (A) mentioned above, the operator norm (with respect to the norm on $\mathbb{H}$) of the totality of components satisfying \eqref{eq:condmomega} is bounded by $C_* e^{2\alpha_0 \delta_* t_\sharp}\cdot e^{-\rho_* t}$, where the first factor $ e^{2\alpha_0 \delta_* t_\sharp}$ comes from the weight in the definition of $\mathbb{H}$. For the differences between $\hat{\bbL}^t_{\bj\to \bj'}$ and ${\bbL}^t_{\bj\to \bj'}$, we apply the second claim of Lemma~\ref{lm:basic_components} and, by the idea (B), find that the last estimate remains valid when we replace $\hat{\bbL}^t_{\bj\to \bj'}$ by ${\bbL}^t_{\bj\to \bj'}$. 

Next we consider the case (ii). If we consider $\hat{\bbL}^t_{\bj\to \bj'}$ in the places of ${\bbL}^t_{\bj\to \bj'}$, then, from the first claim of  Lemma \ref{lm:hyp_component} and the idea (A), the operator norm of the totality of components satisfying \eqref{eq:condmomega} is bounded by $C_* e^{-((1/2)-2\alpha_0) \max\{|m|,|m'|\}}$. 
We apply the second claim of Lemma \ref{lm:hyp_component} 
for the differences between $\hat{\bbL}^t_{\bj\to \bj'}$ and ${\bbL}^t_{\bj\to \bj'}$ and, by the idea (B), find that the last estimate remains valid when we restore ${\bbL}^t_{\bj\to \bj'}$.

Finally we consider the case (iii). In this case, the weight in the definition on the Hilbert space $\mathbb{H}$ plays its roll. Suppose that ${\bbL}^t_{\bj\to \bj'}$ are replaced by $\hat{\bbL}^t_{\bj\to \bj'}$. Then we can apply the first claim of Lemma \ref{lm:basic_components} to each of the components and,
by the idea (A), show that the operator norm of the totality of components satisfying \eqref{eq:condmomega} is bounded by $C_*(\nu) e^{\alpha_0 (m'-m)}\langle \omega'-\omega\rangle^{-\nu}$. Then, applying the second claim of Lemma~\ref{lm:basic_components}, we check that this estimate remains valid when we restore ${\bbL}^t_{\bj\to \bj'}$.
Note that, since
$(m,\omega) \hookrightarrow^t (m',\omega')$ and $\max\{|m|, |m'|\}> \delta_*t_\sharp$ in this case, the factor  $e^{\alpha_0 (m'-m)}$ is bounded by 
$C_* \max\{ e^{-\alpha_0 (\chi_*/2)t_\sharp }\textcolor{\revisionColor}{\langle \omega'-\omega\rangle}, e^{-\alpha_0 \delta_* t_\sharp}\}$.

Collecting the estimates in the cases (i), (ii) and (iii) above and taking sum with respect to the combinations $(m,m')\in \integer^2$, we see that the operator norm of $\Pi_{\omega'}\circ (\bbL^t-K) \circ \Pi_{\omega}$  on $\mathbb{H}$ is bounded by 
\[
C_* \max\{\;\delta_* t_\sharp \cdot e^{(2\alpha_0\delta_*-\rho_*) t_\sharp},\; e^{-((1/2)-2\alpha_0)\delta_* t_\sharp},\; \delta_* t_\sharp\cdot  e^{-\alpha_0\delta_* t_\sharp},\;e^{-\alpha_0(\chi_*/2) t_\sharp}\;\}.
\]
Since we are assuming that $|\omega'-\omega|<\exp(\alpha_0 \delta_* t_\sharp/10)$, this implies the conclusion of Proposition \ref{pp:main_est}, provided that  $\delta_*$ is sufficiently small and $t_\sharp$ is sufficiently large. 

In the case where the assumption \eqref{eq:omegaomegadash} does not hold, the proof is parallel to the argument above but it becomes much simpler. Indeed,
\begin{itemize}
\item  In the case where $|\omega|< \omega_{\sharp}/2$ and $|\omega'|\le \omega_{\sharp}$, we may assume $\max\{|m|,|m'|\} \ge m_{\sharp}$ since  we subtract the compact part $K$ from $\bbL^t$. Since we  can choose large $m_{\sharp}$ depending on $t_\sharp$ and $\omega_\sharp$, we need not consider the case (i). Then the proof goes as well as the argument above for the cases (ii) and (iii). 
\item  In the remaining case, we may suppose that the factors $\langle \omega'-\omega\rangle^{-\nu}$ that appear in the claims of Lemma \ref{lm:basic_components} and Lemma~\ref{lm:hyp_component} are  small enough by letting $\omega_\sharp$ and  $\nu$ be large. Therefore we can go through the argument above with much cruder estimates. (In the case (i), we use Lemma \ref{lm:basic_components} instead of Proposition \ref{prop:central_components}.)
\end{itemize}
We therefore obtain the conclusion of  Proposition \ref{pp:main_est}.
%In particular, we do not have to invoke Proposition \ref{prop:central_components} in either of these remaining cases. 
\end{proof}
In the following subsections, we prove Lemma \ref{lm:basic_components}, Lemma \ref{lm:hyp_component} and Proposition \ref{prop:central_components}. 
We present the proof of Proposition \ref{prop:central_components} first.
\subsection{Proof of Proposition \ref{prop:central_components}}
Let us consider the operator $\hat{\bbL}^t_{\bj\to \bj'}$ for $\bj,\bj'\in \cJ$ satisfying \eqref{eq:ass_prop_central} and set $\omega=\omega(\bj)$, $\omega'=\omega(\bj)$, $m=m(\bj)$ and $m'=m(\bj)$ for brevity. Recall the argument in Subsection \ref{ss:ftlocal}, we express the diffeomorphism $f^t_{\bj\to \bj'}$ as 
\[
f^t_{\bj\to \bj'}=A^t_{\bj\to \bj'}\circ G^t_{\bj\to \bj'}
\] 
in parallel to \eqref{eq:AG}. Accordingly we write $\hat{\bbL}^t_{\bj\to\bj'}$ as
\begin{equation}\label{eq:bbLexp}
\hat{\bbL}^t_{\bj\to\bj'}=\mathcal{M}( \psi_{\bj'}\cdot \mathcal{X}'_{\bj\to \bj'}) \circ \mathbb{A}\circ \mathbb{G}
\end{equation}
where $\mathcal{M}(\varphi)$ denotes the multiplication operator by $\varphi$ and we set 
\[
\mathbb{G}:L^2(\supp \psi_{\bj})\to L^2(\real^{2+2+1}),\quad
\mathbb{G}u= \pBargmann \left(\left(\rho^t_{\bj\to \bj'}\cdot \pBargmann^* (\mathcal{X}_{\bj\to \bj'}\cdot u)\right)\circ (G^t_{\bj\to \bj'})^{-1}\right)
\]
and 
\[
 \mathbb{A}:L^2(\real^{2+2+1})\to L^2(\real^{2+2+1}),\quad
\mathbb{A}u=\pBargmann((\pBargmann^* u)\circ (A^t_{\bj\to \bj'})^{-1}).
\]
Below we disregard the part $\mathbb{G}$ for a while and concentrate on the part $\mathcal{M}(\psi_{\bj'}\cdot \mathcal{X}'_{\bj\to \bj'}) \circ\mathbb{A}$. In the last part of this proof,  we will show that the pre-composition of $\mathbb{G}$ is negligible.

The operator $\mathbb{A}$ is a rather simple integral operator and its kernel can be presented explicitly by computing Gaussian integrals. (See \cite[Ch.~3]{Martinez}.) But our main concern in the argument below is actually the variations of the functions $\psi_{\bj}$ and $\tilde{\psi}_{\bj}$ with respect to the space variable $w$. Recall from \eqref{eq:sumq} and \eqref{eq:tilde_sumq} that their definitions involve the mapping $e_{\bj}(w):e_{a(\bj),\omega(\bj),n(\bj)}(w)$.
To proceed, we introduce the functions
\begin{align}\label{eq:Psibj0}
&\Psi^0_{\bj}:\real^{2+2+1}\to [0,1], \quad \Psi^0_{\bj}(w,\xi,\eta)=  \chi\bigg(4^{-1}e^{-\delta_* t_\sharp}|\omega|^{-1/2}\cdot \|\textcolor{\revisionColor}{\bDelta_{\bj}^{-1}}(\xi-\eta\cdot  e_{\bj}(w))\|\bigg)
\intertext{and}
\label{eq:Psibj}
&\Psi_{\bj}:\real^{2+2+1}\to [0,1], \quad \Psi_{\bj}(w,\xi,\eta)= {q}_{\omega(\bj)}(\eta)\cdot \Psi_{\bj}^0(w,\xi,\eta)
\end{align}
where ${q}_{\omega}$ is the function defined in \eqref{eq:tildeq}. 
We prove the following lemma as  the main step of the proof of Proposition \ref{prop:central_components}.
\begin{Lemma}
\label{lem:main_est} Under the assumptions as above, we have 
\begin{equation}\label{eq:claim_on_A}
\|\mathcal{M}(\Psi_{\bj'}\cdot \mathcal{X}'_{\bj\to \bj'})\circ{\mathbb{A}}:L^2(\supp (\Psi^0_{\bj}\cdot  \mathcal{X}_{\bj\to \bj'}))\to L^2(\supp (\Psi_{\bj'}\cdot \mathcal{X}'_{\bj\to \bj'}))\|\le e^{-\rho_* t_\sharp}.
\end{equation}
\end{Lemma} 
In the following, we prove this lemma by showing that the operator norm of\footnote{We suppose that the image of this operator is restricted to $\supp (\Psi_{\bj}\cdot \mathcal{X}_{\bj\to \bj'})$.}  
\begin{equation}\label{eq:opAtA}
 {\mathbb{A}}^*\circ \mathcal{M}(\Psi_{\bj'}\cdot\mathcal{X}'_{\bj\to \bj'})^2\circ  {\mathbb{A}}:L^2(\supp (\Psi_{\bj}\cdot \mathcal{X}_{\bj\to \bj'}))\to L^2(\supp (\Psi_{\bj}\cdot \mathcal{X}_{\bj\to \bj'}))
\end{equation}
is bounded by $e^{-2\rho_* t_\sharp}$. 
Let us recall the expression \eqref{eq:A} of the diffeomorphism $A^t_{\bj\to \bj'}$. Below we suppose $A^t_{\bj\to \bj'}(0)=0$ by shifting the coordinates, hence
\[
A^t_{\bj\to \bj'}(x,y,z)=
\big(\lambda x,\tilde{\lambda} y, z+\varpi\cdot (x,y)+\beta xy\big)
\] 
where $\lambda$, $\tilde{\lambda}$ and $\varpi$ are those given in \eqref{eq:b} for $f=f^t_{\bj\to \bj'}$.
The inverse of $A^t_{\bj\to \bj'}$ is then written
\begin{equation}\label{eq:A_inv}
(A^t_{\bj\to \bj'})^{-1}(x,y,z)=
\bigg(\Lambda^{-1}\begin{pmatrix}x\\y\end{pmatrix}, z-\varpi\cdot \Lambda^{-1}\begin{pmatrix}x\\y\end{pmatrix}-\sigma(x,y)\bigg)
\end{equation}
where 
\[
\Lambda=\begin{pmatrix}\lambda&0\\
0&\tilde{\lambda}
\end{pmatrix}
\quad\text{and}\quad
\sigma(x,y)=\beta \lambda^{-1} \tilde{\lambda}^{-1} xy.
\]
We write the operator $\mathbb{A}$ as an integral operator 
\[
\mathbb{A}u(w'',\xi'',\eta)=\int K_{\mathbb{A}}(w,\xi;w'',\xi'';\eta)\, u(w,\xi,\eta)\, dw d\xi 
\]
with the kernel
\[
K_{\mathbb{A}}(w'',\xi'';w,\xi;\eta)=  e^{-i\xi w/2-i\xi'' w''/2}
\cdot  k_{\mathbb{A}}(w'',\xi'';w,\xi;\eta) 
\]
where\footnote{Note that the right-hand side of \eqref{eq:Keta} does not depend on $z$. We separated the term $e^{-i\xi w/2-i\xi'' w''/2}$ in order to simplify the expressions below.}
\begin{equation}\label{eq:Keta}
k_{\mathbb{A}}(w'',\xi'';w,\xi;\eta)=e^{i\xi w/2+i\xi'' w''/2}\cdot \int \overline{{\phi}_{w'',\xi'',\eta}(\tilde{w},z)}\cdot {\phi}_{w,\xi,\eta}((A^t_{\bj\to \bj'})^{-1}(\tilde{w},z))\, d\tilde{w}.
\end{equation}
Using the expression \eqref{eq:A_inv} of $(A^t_{\bj\to \bj'})^{-1}$  and changing the variable $\tilde{w}$ to $\tilde{w}+w''$, we rewrite the last expression as  
\begin{align*}
k_{\mathbb{A}}&(w,\xi;w'',\xi'';\eta)\\
& =\langle \eta\rangle  \int d\tilde{w}\; 
\exp(i(\xi (\Lambda^{-1}(\tilde{w}+w''))-\xi'' \tilde{w}-\eta \varpi \cdot \Lambda^{-1}(\tilde{w}+w'')-\eta\cdot  \sigma(\tilde{w}+w'')))\\
&\qquad\qquad\qquad\cdot \exp(-
\langle \eta\rangle \|\Lambda^{-1}(\tilde{w}+w'')-w\|^2/2
-\langle \eta\rangle \|\tilde{w}\|^2/2).\notag
\end{align*}
Then we can write the operator \eqref{eq:opAtA} as 
\begin{equation}\label{eq:AtA}
({\mathbb{A}}^*\circ\mathcal{M}( \Psi_{\bj'}\cdot \mathcal{X}'_{\bj\to \bj'})^2 \circ   {\mathbb{A}}) u(w',\xi',\eta)=
\int e^{-i\xi w/2+i\xi' w'/2}\cdot \mathcal{K}(w',\xi';w,\xi;\eta)\, u(w,\xi,\eta)\, dw d\xi
\end{equation}
where, introducing the variable $\zeta=\xi''-\eta e_{\bj'}(w'')$,  we set 
\begin{multline}\label{eq:KAtA}
\mathcal{K}(w',\xi';w,\xi;\eta)= q_{\omega'}(\eta)^2 \cdot  \int d\zeta \int dw''\cdot \chi\left(\textcolor{\revisionColor}{4^{-1}}e^{-\delta_* t_\sharp}|\omega|^{-1/2}\cdot  \|\textcolor{\revisionColor}{\bDelta_{\bj'}^{-1}}\zeta\|\right)^2\cdot \mathcal{X}'_{\bj\to \bj'}(w'')^2\\
\times 
k_{\mathbb{A}}(w'',\zeta+\eta e_{\bj'}(w'');w,\xi;\eta)\cdot 
\overline{k_{\mathbb{A}}(w'',\zeta+\eta e_{\bj'}(w'');w',\xi';\eta)}.
\end{multline}
In the integral on the right-hand side of \eqref{eq:KAtA}, we are going to compute the integration with respect to the variable $w''=(x'',y'')$.  
If we write the integral by putting \eqref{eq:Keta} and extract the parts that involve the variable $w''=(x'',y'')$, then, setting
\[
e_{\bj'}(w):=(\theta^u(w),\theta^s(w)).
\] 
we find
\begin{align}\label{eq:int_xdd}
&I(w,\xi;w',\xi';\tilde{w},\tilde{w}';\eta):= \int dx'' dy'' \cdot \mathcal{X}'_{\bj\to \bj'}(x'',y'')^2
\\
&\qquad \cdot
\exp(-i\eta (\tilde{y}-\tilde{y}')\cdot\theta^s(x'',y'')+i(\xi_x-\xi'_x)\lambda^{-1}x''-i\beta \lambda^{-1} \tilde{\lambda}^{-1} \eta (\tilde{y}-\tilde{y}')x'')\notag\\
&\qquad\cdot \exp(-i\eta (\tilde{x}-\tilde{x}')\cdot \theta^u(x'',y''))\times \exp(i(\xi_y-\xi'_y)\tilde{\lambda}^{-1}y''-i\eta\beta\lambda^{-1}\tilde{\lambda}^{-1}(\tilde{x}-\tilde{x}')y'')\notag
\\
&\qquad \cdot \exp(
-\langle\eta\rangle|\lambda^{-1}(\tilde{x}+x'')-x|^2/2
-\langle\eta\rangle|\lambda^{-1}(\tilde{x}'+x'')-x'|^2/2)\notag\\
&\qquad \cdot \exp(
-\langle\eta\rangle|\tilde{\lambda}^{-1}(\tilde{y}+y'')-y|^2/2
-\langle\eta\rangle|\tilde{\lambda}^{-1}(\tilde{y}'+y'')-y'|^2/2)\notag
\end{align}
and \eqref{eq:KAtA} is written as
\begin{align}\label{eq:Kexpression_by_I}
\mathcal{K}&(w',\xi';w,\xi;\eta)=q_{\omega'}(\eta)^2 \cdot \langle \eta\rangle^{2}\cdot\int d\zeta d\tilde{w} d\tilde{w}' \cdot I(w,\xi;w',\xi';\tilde{w},\tilde{w}';\eta)\\
&\cdot \chi\left(\textcolor{\revisionColor}{4^{-1}}e^{-\delta_* t_\sharp}|\omega|^{-1/2}\cdot  \|\textcolor{\revisionColor}{\bDelta_{\bj'}^{-1}}\zeta\|\right)^2\cdot \exp(-\langle \eta\rangle(\|\tilde{w}\|^2+\|\tilde{w}'\|^2)/2)\notag\\
&\cdot \exp\Bigg(i(\xi  \Lambda^{-1}\tilde{w}-\xi' \Lambda^{-1}\tilde{w}'-\zeta(\tilde{w}-\tilde{w}')-\eta \varpi \Lambda^{-1}(\tilde{w}-\tilde{w}')-\eta(\sigma(\tilde{w})-\sigma(\tilde{w}')))\bigg)\notag
\end{align}

In the following, we consider the following two cases separately:
\begin{center}
(I)  $\Delta_{\bj}< 
e^{3\delta_* t_\sharp}$, \qquad
(II) $\Delta_{\bj}\ge 
e^{3\delta_* t_\sharp}$.
\end{center}
In the case (I), we will use the non-integrability condition $(NI)_\rho$ to deduce the required estimate.  
In the case (II), we will use the fact that the approximate  non-integrability $\Delta_{\bj}=\Delta(p,\konst \langle \omega(\bj)\rangle)$ is sufficiently large.  In the following, we suppose that 
\begin{equation}
\label{eq:wxieta}
(w,\xi,\eta), (w',\xi',\eta)\in \supp ({\Psi}_{\bj}\cdot  \mathcal{X}_{\bj\to \bj'}).
\end{equation} 
Note that this implies in particular that $
\|w\|$ and $\|w'\|$ are bounded by $C_* \konst \langle \omega\rangle$.

\subsubsection*{{\bf Case (I)}} 
We consider the integration in \eqref{eq:int_xdd} with respect to the variable $x''$. 
Notice that the factor on the second line of the right-hand side of \eqref{eq:int_xdd} is of the form to which we can apply \eqref{eq:tv2} in Lemma~\ref{lm:e_approx} with setting  
\begin{equation}\label{eq:setting}
h=\lambda^{1/2} |\omega|^{-1/2},\quad 
b=-h \cdot \eta(\tilde{y}-\tilde{y}'), \quad 
\alpha=\frac{(\xi_x-\xi'_x)\lambda^{-1}-\beta \lambda^{-1}\tilde{\lambda}^{-1}\eta (\tilde{y}-\tilde{y}')}{bh^{-1}}
\end{equation}
provided that $b>b_0$.
For the remaining part on the right-hand side of \eqref{eq:int_xdd}, observe that
\begin{itemize}
\item the factor on the third line is almost constant\footnote{We may suppose that  $|\tilde{x}|$ and $|\tilde{x}'|$ are bounded by $\konst^2\pomega^{-1/2}$ (say) because otherwise the factors on the fourth and fifth lines are very small.} as a function of $x''$ in the scale $|\omega|^{-1/2}$ from  \eqref{eq:tv6} in Lemma \ref{lm:e_approx}.
\item the factor on the fourth line is also almost constant in $x''$ viewed in the scale $|\omega|^{-1/2}$, precisely, 
its derivative with respect to $x''$ is bounded by $\lambda^{-1}\langle \eta\rangle^{1/2}\le 2 \lambda^{-1}|\omega|^{1/2}$,
\item the factor on the fifth line does not depend on $x''$.
\end{itemize}
These observations motivate us to divide the domain of integration with respect to $x''$ ({\it i.e.} the real line $\real$)  into intervals with length $2h=2\lambda^{1/2}|\omega|^{-1/2}$ and apply  \eqref{eq:tv2} in Lemma~\ref{lm:e_approx} to the integral \eqref{eq:int_xdd} with respect to $x''$ on each of those intervals with the setting \eqref{eq:setting}.  We approximate the remaining parts ({\it i.e.} the factor $\mathcal{X}'_{\bj\to \bj'}(x'',y'')^2$ on the first line and those on the third to fifth lines) on the right-hand side of \eqref{eq:int_xdd} by their averages on the interval and, to ensure the assumption of  Lemma~\ref{lm:e_approx}, we suppose
\begin{equation}\label{eq:cond_b}
b_0\le |b|=|h\cdot \eta (\tilde{y}-\tilde{y'})|=\lambda^{1/2}  |\omega|^{-1/2}\cdot |\eta  (\tilde{y}-\tilde{y}')| \le \konst.
\end{equation}
Then we see that the integral \eqref{eq:int_xdd} with resect to  $x''$ on each of the intervals of length $2h=2\lambda^{1/2}|\omega|^{-1/2}$ is bounded by 
\[
C_*(b_0)\cdot 2h \cdot \left(\langle \lambda^{1/2}  |\omega|^{1/2} (\tilde{y}-\tilde{y}')\rangle^{-\rho/2}+\lambda^{-1/2}\right).
\]
Hence, evaluating the factors on the fourth and fifth lines of \eqref{eq:int_xdd}, we conclude
\begin{align}\label{eq:estimate_I}
|I(w,\xi&;w',\xi';\tilde{w},\tilde{w}';\eta)|\\
& \le 
C_*(\nu,b_0) \cdot |\omega|^{-1}
\cdot (\langle \lambda^{1/2}|\omega|^{1/2}|\tilde{y}-\tilde{y}'|\rangle^{-\rho/2}+\lambda^{-1/2})\notag\\
&\quad \cdot \langle |\omega|^{1/2} |\lambda^{-1}(\tilde{x}-\tilde{x}')- (x-x')| \rangle^{-\nu}\cdot \langle |\omega|^{1/2} |\tilde{\lambda}^{-1}(\tilde{y}-\tilde{y}')- (y-y')| \rangle^{-\nu}\notag
\end{align} 
for arbitrarily large $\nu$, under the condition \eqref{eq:cond_b}. 

We put the last estimate  into the expression \eqref{eq:Kexpression_by_I} of the kernel $\mathcal{K}(\cdot)$. And, for the integration with respect to $\zeta$, we apply a simple estimate about Fourier transform to see 
\begin{equation}\label{eq:zeta_int}
\left|\int \chi\left(\textcolor{\revisionColor}{4^{-1}}|\omega|^{-1/2}\cdot e^{-\delta_* t_\sharp}\cdot \|\textcolor{\revisionColor}{\bDelta_{\bj'}^{-1}}\zeta\|\right)^2 \exp(-i\zeta(\tilde{w}-\tilde{w'})) d\zeta\right|
\le  \frac{C_*(\nu)\cdot e^{2\delta_* t_\sharp}\cdot \textcolor{\revisionColor}{\Delta_{\bj'}}\cdot |\omega|}
{\langle e^{\delta_* t_\sharp}|\omega|^{1/2}\|\textcolor{\revisionColor}{\bDelta_{\bj'}}(\tilde{w}'-\tilde{w})\| \rangle^{\nu}}
\end{equation}
for arbitrarily large $\nu>0$. 
Therefore we see that \eqref{eq:Kexpression_by_I}  is bounded by 
\begin{align}\label{int:ww}
C_*(\nu,b_0) \cdot& q_{\omega'}(\eta)^2\cdot  e^{2\delta_* t_\sharp}\cdot \Delta_{\bj'}\cdot |\omega|^2\cdot  \int  d\tilde{w} d\tilde{w}' \;   \\
&\times  (\langle \lambda^{1/2}|\omega|^{1/2}|\tilde{y}-\tilde{y}'|\rangle^{-\rho/2}+\lambda^{-1/2}) \cdot 
 \langle |\omega|^{1/2} \|\Lambda^{-1}(\tilde{w}'-\tilde{w})-(w'-w)\| \rangle^{-\nu} \notag\\
 &
\times  \langle e^{\delta_* t_\sharp}|\omega|^{1/2}\|\textcolor{\revisionColor}{\bDelta_{\bj'}}(\tilde{w}'-\tilde{w})\| \rangle^{-\nu} \cdot \langle |\omega|^{1/2}\|\tilde{w}\|\rangle^{-\nu}  \langle |\omega|^{1/2}\|\tilde{w}'\|\rangle^{-\nu}\notag
\end{align}
in absolute value. Notice that, in the last claim, we actually  had to restrict the domains of integrations in \eqref{eq:Kexpression_by_I} and \eqref{int:ww} by the condition \eqref{eq:cond_b}. 
However, since the factor  $\exp(-\langle \eta\rangle(\|\tilde{w}\|^2+\|\tilde{w}'\|^2)/2)$ on the second line of \eqref{eq:Kexpression_by_I} is very small in the case where the right inequality of \eqref{eq:cond_b} fails, the claim remains valid without such restrictions. 

Inspecting the integral \eqref{int:ww} with respect to $\tilde{w}$ and $\tilde{w}'$ above, we conclude  
\begin{align}\label{eq:estK1}
&|\mathcal{K}(w',\xi';w,\xi;\eta)|
\le 
\frac{C_*(\nu,b_0) \cdot q_{\omega'}(\eta)^2 \cdot \lambda^{-1}\cdot    (\langle \lambda^{-1/2}|\omega|^{1/2} |y-y'| \rangle^{-\rho/2}+\lambda^{-1/2}) }
{\langle |\omega|^{1/2} \|(\Lambda^{-2}+1)^{-1/2}(w-w')\|\rangle^{\nu}}.
\end{align}
Finally note that  \eqref{eq:claim_on_A} is an operator on $L^2(\supp (\Psi_{\bj}\cdot \chi_{\bj\to \bj'}))$ and that the $2d$-dimensional Lebesgue measure of $\supp \Psi_{\bj}\cap (\{w\}\times \real^2\times \{\eta\})$ for $w\in \real^2$ and $\eta\in \supp q_{\omega}$ is bounded by $C_*e^{2\delta_* t_\sharp}\Delta_{\bj}\cdot |\omega|\le C_*e^{8\delta_* t_\sharp} |\omega|$. (Note that we have the last inequality since we are considering the case (I).) Hence, by simple estimate using  Schur test\footnote{See \cite[p.50]{Martinez} or the wikipedia page of ``Schur test''. We will use this lemma for a few times below.}, we conclude  that the operator norm of \eqref{eq:opAtA} is bounded by 
\begin{equation}\label{eq:last_est}
C_*(b_0)\cdot  e^{8\delta_* t_\sharp}\cdot \lambda^{-\rho/4}\le
C_*(b_0) \cdot e^{-\rho \chi_* t_\sharp/6} \le C_*(b_0) \cdot e^{-2\rho_* t_{\sharp}}
\end{equation}
provided that we let the constants $0<\delta_*<\rho_*$ be sufficiently small. This gives the required estimate \eqref{eq:claim_on_A} as we noted in the beginning. 

\begin{Remark}\label{rem:perturb}
In the argument above for the case (I), we actually used the estimate \eqref{eq:tv2} only for $b$ (which was the same as that in \eqref{eq:tv} in the non-integrability condition $(NI)_\rho$) in a bounded interval contained in
 $[b_0,\konst]$.  And we will not use the non-integrability condition $(NI)_\rho$ in the argument for the case (II) below.
\end{Remark}

\begin{Remark}
It is natural to think that we may be able to apply the argument above also to the case (II), one because the case (II) should be simpler and one because the estimate \eqref{eq:tv2} holds for any $\alpha$. But the author found a technical problem in such argument and had to argue about the case (II) separately. The problem will be presented in the argument for the case (II) below, especially, in Footnote \ref{fn:1}. 
\end{Remark}

\subsubsection*{{\bf Case (II)}} 
We now consider the case (II) where the approximate non-integrability $\Delta_{\bj}$ is large. Let us begin with a preliminary discussion. Observe that,  from  \eqref{eq:tv3} and \eqref{eq:tv6} in Lemma \ref{lm:e_approx}, the unstable subspace $E_u$ viewed in the local chart~$\kappa_{\bj}$ twists along segments parallel to the $y$-axis by the rate proportional to $\Delta_{\bj}\ge 
e^{3\delta_* t_\sharp}\gg 1$, while it is almost constant on segments parallel to the $x$-axis. Hence, recalling  the definition \eqref{eq:Psibj} of  ${\Psi}_{\bj}$, we see that 
\begin{equation}\label{eq:diffxi}
|\xi_x-\xi'_x|\ge C_*^{-1}\Delta_{\bj}\cdot |\omega|\cdot |y-y'|\gg 1
\end{equation}
for $(x,y,\xi_x,\xi_y,\eta), (x',y',\xi'_x,\xi'_y,\eta)\in \supp {\Psi}_{\bj}\cap U_{\bj\to \bj'}$, 
provided $|y'-y|\ge C_* e^{\delta_*t_\sharp}\pomega^{-1/2}$. 
This motivate us to regard the integral with respect to $x''$ in \eqref{eq:int_xdd} as an oscillatory integral with the oscillating factor $\exp(i(\xi_x-\xi'_x)\lambda^{-1}x'')$. 
To estimate that oscillatory integral, we need the following formula of \emph{regularized integration by parts} because the function $e_{\bj'}(w'')=(\theta^u(w''), \theta^s(w''))$ is not differentiable. (The proof is obtained by a simple computation.) 
\begin{Lemma}[{\cite[{p.137}]{MR2313087}}] \label{lm:reg_int_part} Let $\rho:\real\to \real$ be a $C^\infty$ function supported on $[-1,1]$ such that $\int \rho(s) ds =1$. 
Let $\vartheta\in C^{2}(\real)$ and $g\in C_c^{0}(\real)$.  If $\vartheta'(s)\neq 0$ on a neighborhood of $\supp g$, then we have, for sufficiently small $\varepsilon>0$, that 
\begin{equation}\label{eq:ibp}
\int e^{i \vartheta(s)} g(s) ds=i \int e^{i \vartheta(s)}\cdot (g_{\varepsilon}/\vartheta')'(s) ds+\int e^{i \vartheta(s)} (g(s)-g_{\varepsilon}(s)) ds
\end{equation}
where $g_{\varepsilon}=\rho_{\varepsilon}*g$ and $\rho_{\varepsilon}(s)=\varepsilon^{-1}\rho(\varepsilon^{-1}s)$. 
\end{Lemma}

Now we start the proof in the case (II). 
As in the case (I), we estimate the integral \eqref{eq:int_xdd}. 
To fix ideas, we first proceed with the assumption that  
\begin{equation}\label{cond:yy}
|\omega|^{1/2} |y-y'|\ge  \lambda^{1/2}\quad \text{for $w=(x,y)$ and $w'=(x',y')$}
\end{equation}
and also   
\begin{equation}\label{cond:yy0}
\frac{1}{2}\le \frac{\tilde{\lambda}^{-1}|\tilde{y}-\tilde{y}'|}{|y-y'|}\le 2\quad \text{for $\tilde{w}=(\tilde{x},\tilde{y})$ and $\tilde{w}'=(\tilde{x}',\tilde{y}')$}.
\end{equation}
We are going to apply   
Lemma \ref{lm:reg_int_part} to the integral \eqref{eq:int_xdd} with respect to $x''$,  with setting 
\[
 \vartheta(x'')=(\xi_x-\xi'_x)\lambda^{-1} x'', \qquad \varepsilon=\frac{\lambda}{\Delta_{\bj} \cdot  |\omega|^{1/2} |{y}-{y}'|}\cdot |\omega|^{-1/2}.
\]
In order to evaluate the result of integration by parts, we prepare a few simple estimates. 
First we note that, from \eqref{eq:wxieta} and \eqref{cond:yy}, we can estimate $\varepsilon$ from above and below as\footnote{For the left-most inequality, recall that we have $|\omega|>\omega_\sharp$ and take large $\omega_\sharp$ depending on $t_\sharp$.
}
\[
|\omega|^{-1+\theta_*}\ll C_*^{-1}\lambda\cdot  \Delta_{\bj}^{-1} \cdot  \konst^{-1}|\omega|^{-1/2}\le \varepsilon\le \lambda^{1/2} \Delta_{\bj}^{-1} |\omega|^{-1/2}\ll \konst |\omega|^{-1/2}.
\] 
Hence we can apply  \eqref{eq:tv5} of Lemma \ref{lm:e_approx} with setting $h=\varepsilon$ and get   
\begin{align*}
|\theta^s(x''+\varepsilon \tau,y'')-\theta^s(x'',y'')|
\le C_* \varepsilon \log (\konst \Delta_{\bj}) \quad \text{for $\tau\in [-1,1]$.} 
\end{align*}
This together with \eqref{cond:yy0} and  the fact $C_*^{-1}<|\lambda \tilde{\lambda}|<C_*$ yield that, for $\tau\in [-1,1]$, 
\begin{align*}
&|\exp(i\eta (\tilde{y}-\tilde{y'}) \theta^s(x''+\varepsilon \tau,y''))-\exp(i\eta (\tilde{y}-\tilde{y'})\theta^s(x'',y''))|\\
&\qquad \le C_* \langle  \omega\rangle \cdot\tilde{\lambda} |y-y'|\cdot  \varepsilon  \log (\konst \Delta_{\bj})\le C_* \Delta_{\bj}^{-1}
\cdot \log (\konst \Delta_{\bj})<C_*\Delta_{\bj}^{-1/2}.
\end{align*}
Similarly we get, from \eqref{eq:tv6} in Lemma \ref{lm:e_approx}, that, for $\tau\in [-1,1]$,
\begin{align*}
&|\exp(i\eta (\tilde{x}-\tilde{x'}) \theta^u(x''+\varepsilon \tau,y''))-\exp(i\eta (\tilde{x}-\tilde{x'})\theta^u(x'',y''))|
\\
&\qquad\qquad < C_* \cdot \konst^3\pomega^{-1/2} \cdot \log |\omega| < C_*\Delta_{\bj}^{-1/2}.
\end{align*}
From \eqref{cond:yy0} and \eqref{eq:bB}, we have also
\[
|\beta \lambda^{-1}\tilde{\lambda}^{-1} \eta (\tilde{y}-\tilde{y}')\cdot \varepsilon|\le C_* t_\sharp\cdot \Delta_{\bj}^{-1}< C_*\Delta_{\bj}^{-1/2}.
\]
On the other hand, from \eqref{eq:wxieta}, \eqref{eq:diffxi} and \eqref{cond:yy}, we have
\[
\lambda^{-1}|\xi_x-\xi'_x|\ge C_*^{-1}\lambda^{-1} \Delta_{\bj}\cdot |\omega|\cdot |y-y'|\ge C_*^{-1} \lambda^{-1/2}\Delta_{\bj}\cdot |\omega|^{1/2}.
\]

Now we can apply Lemma \ref{lm:reg_int_part} to the integral in \eqref{eq:int_xdd} with respect to the variable $x''$. 
We write $g(x'')=g(x'';y''; w,\xi;w',\xi';\tilde{w}, \tilde{w}';\eta)$ for the part of the integrand of \eqref{eq:int_xdd} other than the factor $e^{i\vartheta(x'')}$. 
Note that  the factors on the last two lines of \eqref{eq:int_xdd} vary relatively slowly with respect to $x''$ as well as the factor $ \mathcal{X}'_{\bj\to \bj'}(x'',y'')^2$ on the first line. Hence, if we put
\[
h_\nu(w'')=h_\nu(w'';w,w',\tilde{w},\tilde{w}')=\langle |\omega|^{1/2}|\Lambda^{-1}(\tilde{w}+w'')-w|\rangle^{-\nu}\cdot  
\langle |\omega|^{1/2}|\Lambda^{-1}(\tilde{w}'+w'')-w'|\rangle^{-\nu}
\] 
for $\nu>0$ and let $t_\sharp$ and $\omega_\sharp$ be larger if necessary, the resulting terms are estimated as   
\[
\|g_{\varepsilon}-g\|_\infty \le  C_*(\nu)  \Delta_{\bj}^{-1/2}\cdot h_\nu(w'')
\]
and
\[
\left\|\left(\frac{g_{\varepsilon}}{\vartheta'}\right)'\right\|_{\infty}=\left\|\frac{g'_{\varepsilon}}{\vartheta'}\right\|_{\infty}=C_*(\nu)  \Delta_{\bj}^{-1/2}\cdot 
\frac{\varepsilon^{-1} }{\lambda^{-1}|\xi_x-\xi'_x|}\cdot  h_\nu(w'')
\le C_*(\nu)  \Delta_{\bj}^{-1/2}\cdot h_\nu(w'')
\]
for arbitrarily large $\nu>0$. Therefore, integrating the result with respect to $w''$,   we obtain 
\begin{equation}\label{eq:Iest0}
|I(w,\xi;w',\xi';\tilde{w},\tilde{w}';\eta)| \le 
\frac{C_*(\nu)\cdot  \Delta_{\bj}^{-1/2}\cdot |\omega|^{-1} }
{\langle |\omega|^{1/2} |\Lambda^{-1}(\tilde{w}-\tilde{w}')- (w-w')| \rangle^{\nu}}
\end{equation}
for arbitrarily large $\nu>0$, under the assumptions \eqref{cond:yy} and \eqref{cond:yy0}.
%In the case where we drop the condition  \eqref{cond:yy0} while keep \eqref{cond:yy}, we have the same estimate \eqref{eq:Iest0} but without the term $\Delta_{\bj}^{-1/2}$ by plain estimate. 
%Note that, in this case, the factor  $\exp(-\langle \eta\rangle|\tilde{w}|^2/2- \langle \eta\rangle|\tilde{w}'|^2/2)$ in \eqref{eq:Kexpression_by_I} is small. Therefore, putting the estimates on $I(w,\xi;w',\xi';\tilde{w},\tilde{w}';\eta)$ into \eqref{eq:Kexpression_by_I},  we obtain  
%\begin{align}\label{eq:Kalphabeta0}|\mathcal{K}(w',\xi';w,\xi;\eta)|\le &C_*(\nu) \cdot q_{\omega}(\eta)^2 \cdot\Delta_{\bj}^{-1/2} 
%\cdot \lambda^{-1} \langle |\omega|^{1/2} |(\Lambda^{-2}+1)^{-1/2}(w-w')| \rangle^{-\nu} \end{align}
%under the assumption \eqref{cond:yy}.

The last estimate \eqref{eq:Iest0} corresponds to \eqref{eq:estimate_I} in the case (I). 
But we can not conclude the required estimate from \eqref{eq:Iest0} in the case (II) because $\Delta_{\bj}$ may be large and hence we can not follow the argument in the  last part of the proof in the case (I) using the Schur test.
In order to resolve this problem\footnote{\label{fn:1}The precise point of the problem is that the non-integrability condition $(NI)_\rho$ does not tell how the quantity on the right hand side of \eqref{eq:tv} depends on $\alpha$. Below we make use of the fact that we can get \eqref{eq:tv} by (regularized) integration by parts when $|\alpha|$ is sufficiently large and then we can see how the result depend on $\alpha$. This argument is valid only in the case (II).  }, we actually need a little  more information on the result of (regularized) integration by parts. 
As the result of (regularized) integration by parts, the integral $I(w,\xi;w',\xi';\tilde{w},\tilde{w}';\eta)$  is written in the form
\begin{equation}\label{exp_I}
\int \exp(i(\xi-\xi')\Lambda^{-1} w'') \cdot I_0(w'';w,\xi_x;w',\xi'_x;\tilde{w},\tilde{w}';\eta) dw''
\end{equation}
where, correspondingly to the two terms on right-hand side of \eqref{eq:ibp}, $I_0(\cdot)$ is of the form 
\begin{align*}
I_0(w'';w,\xi_x;w',\xi'_x;\tilde{w},\tilde{w}';\eta)=
&\frac{I_1(w'';w,w';\tilde{w},\tilde{w}';\eta)}{\lambda^{-1}(\xi_x-\xi'_x)}+I_2(w'';w,w';\tilde{w},\tilde{w}';\eta).
\end{align*}
From the argument above, we already have  
\begin{align*}
&|I_0(w'';w,\xi_x;w',\xi'_x;\tilde{w},\tilde{w}';\eta)|\le C_*(\nu) \Delta_{\bj}^{-1/2}\cdot h_\nu(w''),\\
&|I_1(w'';w,w';\tilde{w},\tilde{w}';\eta)|\le C_*(\nu) \Delta_{\bj}^{-1/2}\cdot \varepsilon^{-1}\cdot h_\nu(w''),\quad\text{ and}\\
&|I_2(w'';w,w';\tilde{w},\tilde{w}';\eta)|\le C_*(\nu)\Delta_{\bj}^{-1/2}\cdot h_\nu(w'').
\end{align*}
Note that the dependence of $I_0(\cdot)$ on the variables $\xi$ and $\xi'$ is rather simple. Indeed, using \eqref{eq:diffxi} and  the assumption \eqref{cond:yy}, it is easy to see  
\begin{align}\label{eq:estimage_I2}
&|\partial^k_{\xi_x}\partial^{k'}_{\xi'_x}
I_0(w'';w,\xi_x;w',\xi'_x;\tilde{w},\tilde{w}';\eta)|
\le 
C_{*}(\nu,k,k')\cdot   (\lambda^{1/2}\Delta_{\bj} |\omega|^{1/2} )^{-k-k'} h_{\nu}(w'')
\end{align} 
for any $k,k'\ge 0$ with $(k,k')\neq (0,0)$. This together with the estimate on $I_0(\cdot)$ above gives 
\begin{align}\label{eq:estimage_I3}
|\partial^k_{\xi_x}\partial^{k'}_{\xi'_x} 
I_0(w'';w&,\xi_x;w',\xi'_x;\tilde{w},\tilde{w}';\eta)|
\le 
\frac{C_{*}(\nu,k,k') \cdot \max\{\lambda^{-1/4}, \Delta_{\bj}^{-1/2}\}\cdot h_{\nu}(w'') }
{ (e^{\delta_* t_\sharp}\Delta_{\bj} |\omega|^{1/2} )^{k+k'}} 
\end{align} 
for any $k,k'\ge 0$.

\begin{Remark}
\label{rem:except_case}
Notices that the estimates above are obtained under the assumptions \eqref{cond:yy} and \eqref{cond:yy0}. For the proof of the next lemma we note that, in the case where the condition \eqref{cond:yy0} fails while \eqref{cond:yy} remains valid, we have the estimate \eqref{eq:Iest0} with $\Delta_{\bj}^{-1/2}$ replaced by $\lambda^{-1}$, that is, 
\begin{equation}\label{eq:Iest0-bis}
|I(w,\xi;w',\xi';\tilde{w},\tilde{w}';\eta)| \le 
\frac{C_*(\nu)\cdot \lambda^{-1}\cdot |\omega|^{-1} }
{\langle |\omega|^{1/2} |\Lambda^{-1}(\tilde{w}-\tilde{w}')- (w-w')| \rangle^{\nu}}.
\end{equation}
Indeed it is easy to see that simple estimates \emph{without} using integration by parts yield this estimate \eqref{eq:Iest0-bis} without the factor $\lambda^{-1}$ (or the estimate \eqref{eq:Iest0} without the factor $\Delta_{\bj}^{-1/2}$). But, if the condition \eqref{cond:yy0} fails while \eqref{cond:yy} remains valid, we have
\[
|(\tilde{\lambda}^{-1} (\tilde{y}+y'')-y)-(\tilde{\lambda}^{-1} (\tilde{y}'+y'')-y')|\ge |\tilde{\lambda}^{-1}(\tilde{y}-\tilde{y}')-(y-y')|\ge |y-y'|/2\ge |\omega|^{-1/2} \lambda^{1/2}
\]
and therefore we can get the estimate \eqref{eq:Iest0-bis}, making use of arbitrariness of $\nu$. Let us note also that, in such argument without using integration by parts, we may express \eqref{eq:int_xdd} in the form \eqref{exp_I} with the function $I_0(\cdot)$ independent of the variables $\xi$ and $\xi'$. 
\end{Remark}

The next lemma is our remedy to the problem in the case (II) mentioned above. Recall the definition of the function $\mathcal{X}_{\bj\to \bj'}$ and consider another $C^\infty$ function $\widetilde{\mathcal{X}}_{\bj\to \bj'}:\real^2\to [0,1]$ that takes value $1$ on  the $4\konst^{1/2}\langle \omega(\bj)\rangle^{-1/2}$-neighborhood of the subset $\tilde{\pi}(\supp \rho_{\bj\to \bj'})$ and is supported in the  $8\konst^{1/2}\langle \omega(\bj)\rangle^{-1/2}$-neighborhood of that subset. Clearly the estimate parallel to \eqref{eq:Dalpha} holds for this function.  
Let $\widetilde{\Psi}_{\bj}:\real^{2+2+1}\to [0,1]$ be the function defined by 
\[
\widetilde{\Psi}^0_{\bj}(w,\xi,\eta)=  \chi\bigg(8^{-1}e^{-\delta_* t_\sharp} |\omega|^{-1/2} \|\textcolor{\revisionColor}{\bDelta_{\bj}^{-1}}(\xi-\eta\cdot  e_{\bj}(w))\|\bigg)\cdot \widetilde{\mathcal{X}}_{\bj\to \bj'}(w).
\]
Notice that this function differs from ${\Psi}_{\bj}$ in \eqref{eq:Psibj} only in the factor $8^{-1}$ and the multiplication by the function $\widetilde{\mathcal{X}}_{\bj\to \bj'}(w)$. 
\begin{Lemma}\label{lm:opAtA} We have 
\[
\|\pBargmann^*\circ \widetilde{\Psi}^0_{\bj}\circ \mathbb{A}^*\circ \mathcal{M}(\Psi_{\bj'}\cdot \mathcal{X}'_{\bj\to \bj'} )^2\circ   \mathbb{A}\circ \widetilde{\Psi}^0_{\bj} \circ \pBargmann:L^2(\real^{2+1})\to L^2(\real^{2+1})\|\le C_*
\max\{\lambda^{-1/8},\Delta_{\bj}^{-1/4}\}.
\]
\end{Lemma}
Before proving this lemma, let us see that Lemma \ref{lem:main_est} follows from it. 
Letting $\mathbf{1}_{\bj\to \bj'}$ be the multiplication by the characteristic function of the support of $\Psi^0_{\bj}\cdot \mathcal{X}_{\bj\to \bj'}$ and writing,
\[
D=(1-\widetilde{\Psi}^0_{\bj}\circ \pBargmann\circ \pBargmann^*)\circ \mathbf{1}_{\bj\to \bj'}
\quad\text{and}\quad
D^*=\mathbf{1}_{\bj\to \bj'} \circ (1- \pBargmann\circ \pBargmann^*\circ \widetilde{\Psi}^0_{\bj}),
\] 
we see
\begin{align*}%\label{eq:diff_1}
&\mathbf{1}_{\bj\to \bj'}\circ \mathbb{A}^*\circ \mathcal{M}( \Psi_{\bj'}\cdot \mathcal{X}'_{\bj\to \bj'})^2\circ  \mathbb{A} \circ \mathbf{1}_{\bj\to \bj'}\\
&\quad\qquad-\mathbf{1}_{\bj\to \bj'}\circ
\pBargmann\circ \bigg[\pBargmann^*\circ \widetilde{\Psi}^0_{\bj}\circ \mathbb{A}^*\circ \mathcal{M}(\Psi_{\bj'}\cdot \mathcal{X}'_{\bj\to \bj'})^2\circ  \mathbb{A}\circ \widetilde{\Psi}^0_{\bj}\circ \pBargmann\bigg]\circ \pBargmann^*\circ \mathbf{1}_{\bj\to \bj'}\\
&= D^*\circ \mathbb{A}^*\circ \mathcal{M}(\Psi_{\bj'}\cdot \mathcal{X}'_{\bj\to \bj'})^2\circ \mathbb{A} \circ \mathbf{1}_{\bj\to \bj'}\\
&\qquad \qquad +\mathbf{1}_{\bj\to \bj'}\circ\pBargmann\circ \pBargmann^*\circ \widetilde{\Psi}^0_{\bj} \circ \mathbb{A}^*\circ \mathcal{M}(\Psi_{\bj'}\cdot \mathcal{X}'_{\bj\to \bj'} )^2\circ \mathbb{A}\circ D.
\end{align*} 
From the estimate on the kernel of $\pBargmannP=\pBargmann\circ \pBargmann^*$ in  Lemma \ref{lm:pBargmannL2}, it is clear that the operator norms of $D$ and $D^*$ with respect to the $L^2$ norm are  bounded by 
\[
C_* \exp (-e^{\delta_* t_\sharp})\ll \lambda^{-1/8}.
\] 
Since the operators $\mathbb{A}$, $\pBargmann$ and $\pBargmann^*$ do not increase the $L^2$ norm, the same estimate holds for the operator norms of the two operators on the right-hand side above. 
Therefore the required estimate \eqref{eq:claim_on_A} follows from Lemma \ref{lm:opAtA}.

\begin{proof}[Proof of Lemma \ref{lm:opAtA}]
The proof is easy once we have the estimate \eqref{eq:estimage_I3} and note Remark \ref{rem:except_case}. 
Let us write the operator under consideration as an integral operator
\[
(\pBargmann^*\circ \widetilde{\Psi}^0_{\bj}\circ \mathbb{A}^*\circ  \mathcal{M}( \Psi_{\bj'}\cdot \mathcal{X}_{\bj\to \bj'} )^2\circ \mathbb{A}\circ \widetilde{\Psi}^0_{\bj}\circ \pBargmann)\,
u(w'_\dag, z')
=\int \widetilde{\mathcal{K}}(w'_\dag, z';w_\dag,z) u(w_\dag,z)\, dw_\dag dz.
\]
From the description \eqref{eq:AtA} of   $\mathbb{A}^*\circ  (\Psi_{\bj'}\cdot \mathcal{X}_{\bj\to \bj'} )^2\circ \mathbb{A}$, we can write the kernel as
\begin{align}\label{eq:wtK}
\widetilde{\mathcal{K}}(w'_\dag, z';w_\dag,z) 
=&\int dwd\xi dw'd\xi' d\eta \cdot e^{-i\xi w/2+i\xi' w'/2}\cdot \mathcal{K}(w',\xi';w,\xi;\eta)\\
&\quad \cdot  \widetilde{\Psi}^0_{\bj}(w,\xi,\eta)\cdot 
\widetilde{\Psi}^0_{\bj}(w',\xi',\eta)\cdot  \overline{\phi_{w,\xi,\eta}(w_\dag,z)}\cdot \phi_{w',\xi',\eta}(w'_\dag,z')\notag
\end{align}
where 
\begin{align}\label{eq:Kexpression_by_I2}
\mathcal{K}(w',&\xi';w,\xi;\eta)={q}_{\omega'}(\eta)^2 \cdot \langle \eta\rangle^{2}\cdot\int d\zeta d\tilde{w} d\tilde{w}' dw''\cdot \exp(-\langle \eta\rangle(\|\tilde{w}\|^2+\|\tilde{w}'\|^2)/2) \\
&\cdot \exp(i\zeta (\tilde{w}-\tilde{w}'))\cdot \chi\left(\textcolor{\revisionColor}{8^{-1}}e^{-\delta_* t_\sharp}|\omega|^{-1/2}\cdot  \|\textcolor{\revisionColor}{\bDelta_{\bj'}^{-1}}\zeta\|\right)^2\cdot \mathcal{X}_{\bj\to \bj'}(w'')^2\notag\\
&\cdot \exp(i(\xi \cdot \Lambda^{-1}(w''+\tilde{w})-\xi' \cdot \Lambda^{-1}(w''+\tilde{w}'))) \cdot  I_0(w'';w,\xi;w',\xi';\tilde{w},\tilde{w}';\eta)\notag\\&
\cdot \exp(-i\eta  \varpi \Lambda^{-1}(\tilde{w}-\tilde{w}')-i\eta(\sigma(\tilde{w})-\sigma(\tilde{w}'))).\notag
\end{align}
For a while we restrict the domain of integration \eqref{eq:wtK}  to the region where the condition \eqref{cond:yy} holds. 
We are going to estimate the integration with respect to the variables $\zeta$, $\xi$, $\xi'$ and $\eta$. 
For the integration with respect to $\zeta$, we apply the plane estimate \eqref{eq:zeta_int} on Fourier transform.
For the integration with respect to the other variables $\xi$, $\xi'$ and $\eta$, we apply similar estimates, making use of the estimates \eqref{eq:Iest0} and \eqref{eq:estimage_I2} if the condition \eqref{cond:yy0} holds and recalling Remark \ref{rem:except_case} otherwise.
Then we find
\begin{align*}
|\widetilde{\mathcal{K}}&(w'_\dag, z';w_\dag,z)|\le C_*(\nu) \max\{ \lambda^{-1/4}, \Delta_{\bj}^{-1/2}\} \cdot \langle z-z'\rangle^{-\nu} \cdot |\omega|^3 \cdot  \int dw dw' dw'' d\tilde{w} d\tilde{w}'  \\
&\quad  \cdot  \langle |\omega|^{1/2}
 \|\tilde{w}\|\rangle^{-\nu}
\cdot  \langle |\omega|^{1/2} \|\tilde{w}'\|\rangle^{-\nu}
\cdot (e^{\delta_* t_\sharp}\Delta_{\bj'}^{1/2}|\omega|^{1/2})^2
\cdot 
\langle e^{\delta_*t_\sharp} |\omega|^{1/2}\|\bDelta_{\bj'}(\tilde{w}-\tilde{w}')\|\rangle^{-\nu}\\
& \quad\cdot \langle e^{\delta_* t_\sharp}|\omega|^{1/2}\|\bDelta_{\bj}(\Lambda^{-1}(w''+\tilde{w})-w_\dag)\|\rangle^{-\nu}
\cdot \langle e^{\delta_* t_\sharp}|\omega|^{1/2}\|\bDelta_{\bj}^{-1}(\Lambda^{-1}(w''+\tilde{w}')-w'_\dag)\|\rangle^{-\nu} \\
&\quad  \cdot (e^{\delta_* t_\sharp}\Delta_{\bj}^{1/2}|\omega|^{1/2})^4 \cdot \langle |\omega|^{1/2}\|w-w_\dag\|\rangle^{-\nu}
\cdot \langle |\omega|^{1/2}\|w'-w'_\dag\|\rangle^{-\nu}
.
\end{align*}
By estimating the integral above with respect to $w$, $w'$, $w''$, we continue
\begin{align*}
|\widetilde{\mathcal{K}}(w'_\dag, z';w_\dag,z)|&\le C_*(\nu) \max\{ \lambda^{-1/4}, \Delta_{\bj}^{-1/2}\} \cdot  \langle z-z'\rangle^{-\nu} \cdot  |\omega|  \int d\tilde{w} d\tilde{w}'  \\
&\quad\cdot  \langle |\omega|^{1/2}
 \|\tilde{w}\|\rangle^{-\nu}
  \langle |\omega|^{1/2} \|\tilde{w}'\|\rangle^{-\nu}
 (e^{\delta_* t_\sharp}\Delta_{\bj'}^{1/2}|\omega|^{1/2})^2  
\langle e^{\delta_*t_\sharp} |\omega|^{1/2}\|\textcolor{\revisionColor}{\bDelta_{\bj'}}(\tilde{w}-\tilde{w}')\|\rangle^{-\nu}\\
& \quad \cdot (e^{\delta_* t_\sharp}\Delta_{\bj}^{1/2}|\omega|^{1/2})^2 \cdot \langle e^{\delta_* t_\sharp}|\omega|^{1/2}\|\Delta_{\bj}(\Lambda^{-1}(\tilde{w}-\tilde{w}')-(w_\dag-w'_\dag)))\|\rangle^{-\nu}.
\end{align*}
Further, changing variables  $(\tilde{w},\tilde{w}')$ to $(\tilde{w}, \tilde{w}'':=\tilde{w}'-\tilde{w})$ and computing the integral above with respect to $\tilde{w}$ and then to $\tilde{w}''$, we reach the estimate
\begin{align}\label{eq:tildeKalphabeta3}
|\widetilde{\mathcal{K}}(w'_\dag, z';w_\dag,z)|&\le C_*(\nu) \max\{ \lambda^{-1/4}, \Delta_{\bj}^{-1/2}\} \cdot  \langle z-z'\rangle^{-\nu}\int  d\tilde{w}''  \\
&\quad \cdot (e^{\delta_* t_\sharp}\Delta_{\bj'}^{1/2}|\omega|^{1/2})^2  \cdot  \langle e^{\delta_*t_\sharp} |\omega|^{1/2}\|\textcolor{\revisionColor}{\bDelta_{\bj'}}\tilde{w}''\|\rangle^{-\nu}\notag\\
&\quad  \cdot (e^{\delta_* t_\sharp}\Delta_{\bj}^{1/2}|\omega|^{1/2})^2 \cdot \langle e^{\delta_* t_\sharp}|\omega|^{1/2}
\|\textcolor{\revisionColor}{\bDelta_{\bj}}(\Lambda^{-1}\tilde{w}''-(w_\dag-w'_\dag)))\|\rangle^{-\nu}\notag\\
&\le C_*(\nu) \cdot t_\sharp^{\nu+2}\cdot \max\{ \lambda^{-1/4}, \Delta_{\bj}^{-1/2}\} \cdot \langle z-z'\rangle^{-\nu} 
\notag
\\
&\quad \cdot \lambda^{-1} (e^{\delta_* t_\sharp}\Delta_{\bj}^{1/2}|\omega|^{1/2})^2 \cdot  \langle e^{\delta_* t_\sharp}|\omega|^{1/2}
\|\textcolor{\revisionColor}{\bDelta_{\bj}}(\Lambda^2+1)^{-1/2}(w_\dag-w'_\dag)\|\rangle^{-\nu}.\notag
\end{align}
In the last inequality, we have used the fact that the ratio between $\Delta_{\bj}$ and $\Delta_{\bj'}$ is bounded by $C_*t_\sharp$, as we noted in the proof of Lemma \ref{lm:hook}.

Now recall that we have restricted the domain of the integration by the condition \eqref{cond:yy} in the argument above. For the integral on the remaining region where \eqref{cond:yy} does not hold, we can argue in parallel without using integration by parts and get 
 the corresponding estimates without the term $\max\{ \lambda^{-1/4}, \Delta_{\bj}^{-1/2}\}$. 
But, since we have the restriction  $|\omega|^{1/2} |y-y'|< \lambda^{1/2}$ in this case,
we restore the additional factor 
$C_* \lambda^{-1/2}$ when we estimate the integral with respect to $w$ and $w'$. Therefore we obtain the estimate \eqref{eq:tildeKalphabeta3} without the restriction of the domain.
Finally we conclude the required estimate  from \eqref{eq:tildeKalphabeta3} by using Schur test, provided that $t_\sharp$ is sufficiently large.  
\end{proof}

We have finished the proof of Lemma \ref{lem:main_est}. 
To complete the proof of Proposition \ref{prop:central_components}, we recall the expression \eqref{eq:bbLexp} of $\hat{\bbL}^t_{\bj\to\bj'}$ and consider the effect of the pre-composition of  $\mathbb{G}$. 
\begin{Remark}\label{rem:smalldiff}
As we will see, the estimates given below are rather crude  and can be obtained in many different ways. But one have to be attentive to the fact that   $f^t_{\bj\to \bj'}$ and $G^t_{\bj\to \bj'}$ are assumed to be only $C^r$ with some $r\ge 3$. 
It makes the argument below a little technical. 
This remark applies also to the proof of Lemma \ref{lm:hyp_component} in the next subsection.
\end{Remark}
Let $\hat{\rho}^t_{\bj\to \bj'}:\real^{2+1}\to \complex$ be the Fourier transform of ${\rho}^t_{\bj\to \bj'}$ solely in the variable $z$: 
\[
\hat{\rho}^t_{\bj\to \bj'}(w,\eta)= 
\int e^{-i\eta z}\cdot  {\rho}^t_{\bj\to \bj'}(w,z) dz.  
\]
In the next lemma, we compare $\mathbb{G}$ with  
the operator
\[
\mathbb{P}:L^2(\supp \psi_{\bj})\to L^2(\real^{2+2+1}), \quad 
\mathbb{P} u(w,\xi,\eta)=\int \hat{\rho}^t_{\bj\to \bj'}(w,\eta-\eta') \cdot 
\pBargmannP (\textcolor{\revisionColor}{\mathcal{X}_{\bj\to \bj'}}\cdot u)(w,\xi,\eta') d\eta'
\]
where $\pBargmannP=\pBargmann\circ \pBargmann^*$ is the Bargmann projection operator in \eqref{eq:pBargmannP}. 
\begin{Lemma}\label{lm:Gneg}
$
\|\mathbb{G}-\mathbb{P}:L^2(\supp \psi_{\bj})\to L^2(\real^{2+2+1})\|\le \konst^{-1/2}$. 
%Further, for the kernel of the operator $\mathbb{P}$, we have, for any $\nu>0$, that
%\[
%|K(w,\xi,\eta;w',\xi',\eta')|\le
%\frac{ C_*(\nu) \cdot \langle\pomega^{1/2} |w-w'|\rangle^{-\nu} \cdot \langle\pomega^{-1/2} |\xi-\xi'|\rangle^{-\nu} }{\langle \eta-\eta'\rangle^{\nu}\cdot \langle\pomega^{1/2} d(w, \pi(\supp \rho^{t}_{\bj\to \bj'}))\rangle^{\nu}}.
%\]
\end{Lemma}
Before proving this lemma, we finish the proof of Proposition \ref{prop:central_components} using it.  Since the operator $\mathbb{A}$ does not increase the $L^2$ norm of functions, the lemma above implies 
\[
\|\mathcal{M}( \psi_{\bj'}\cdot \mathcal{X}'_{\bj\to \bj'}) \circ\mathbb{A}\circ (\mathbb{G}- \mathbb{P}):L^2(\supp \psi_{\bj})\to L^2(\supp \psi_{\bj'})\|\le \konst^{-1/2}\ll e^{-\rho_* t_\sharp}.
\]
Let $\mathbf{1}_{\bj\to \bj'}:\real^{2+2+1}\to [0,1]$ be the characteristic function that we have introduced in (the last part of) the proof of Lemma \ref{lem:main_est} and write
\[
\mathcal{M}( \psi_{\bj'}\cdot \mathcal{X}'_{\bj\to \bj'}) \circ\mathbb{A}\circ \mathbb{P}=\mathcal{M}( \psi_{\bj'}\cdot \mathcal{X}'_{\bj\to \bj'}) \circ\mathbb{A}\circ \mathbf{1}_{\bj\to \bj'}\circ \mathbb{P}+\mathcal{M}( \psi_{\bj'}\cdot \mathcal{X}'_{\bj\to \bj'}) \circ\mathbb{A}\circ (1-\mathbf{1}_{\bj\to \bj'})\circ \mathbb{P}.
\]
From Lemma \ref{lem:main_est} on $\mathbb{A}$, the operator norm of the former part on the right-hand side is bounded by $C_*e^{-\rho_* t_\sharp}$. From the localized property of the kernel of $\mathbb{P}$ as a consequence of Lemma \ref{lm:rhot} and Lemma \ref{lm:pBargmannL2}, we also see that the operator norm of the latter is (much) smaller than $e^{-\rho_* t_\sharp}$. 
We therefore obtain Proposition \ref{prop:central_components}.
\begin{proof}[Proof of Lemma \ref{lm:Gneg}] As an intermediate approximation, we consider the operator
\[
\widetilde{\mathbb{P}}:=\pBargmann\circ \mathcal{M}({\rho}^t_{\bj\to \bj'})\circ \pBargmann^* \textcolor{\revisionColor}{\circ\mathcal{M}(\mathcal{X}_{\bj\to \bj'})}:L^2(\supp \psi_{\bj})\to L^2(\real^{2+2+1})
\]
which is obtained by letting $G^t_{\bj\to \bj'}$ be the identity map in the definition of $\mathbb{G}$. 
The operator norm of $\mathbb{G}-\widetilde{\mathbb{P}}:L^2(\supp \psi_{\bj})\to L^2(\real^{2+2+1})$ coincides with that of
\begin{equation}\label{eq:bgp}
\pBargmann^* \circ (\mathbb{G}-\widetilde{\mathbb{P}}): L^2(\supp \psi_{\bj})\to L^2(\real^{2+1})
\end{equation}
because $\pBargmann^*\circ \pBargmann=\mathrm{Id}$ and $\pBargmann$ is an isometric embedding with respect to the $L^2$ norms. 
We may write this operator as an integral operator
\begin{equation}\label{eq:ker:BGP}
\pBargmann^* \circ (\mathbb{G}-\widetilde{\mathbb{P}}) u(w',z')
=\int K(w',z';w,\xi,\eta) \, u(w,\xi,\eta)\, dw d\xi d\eta.
\end{equation}
From Lemma \ref{lm:abc} and Lemma \ref{lm:rhot}  (see also Remark \ref{Re:aboutG}), the kernel satisfies
\begin{align*}
|K(w',z';w,\xi,\eta)|&\le \bigg|(\rho^t_{\bj\to \bj'}\cdot \phi_{w,\xi,\eta})((G^t_{\bj\to \bj'})^{-1}(w',z'))-(\rho^t_{\bj\to \bj'}\cdot\phi_{w,\xi,\eta})(w',z')\bigg|\\
&\le C_*(\nu) \cdot |\omega|^{-1/2+2\theta_*} \cdot \left(|\omega|^{1/2} \cdot \big\langle|\omega|^{1/2} (w-w')\big\rangle^{-\nu}\right)\cdot \langle z'\rangle^{-\nu}
\end{align*}
and vanishes unless  $(w',z')\in \supp\rho^t_{\bj\to \bj'}$ 
and $(w,\xi,\eta)\in \supp \psi_{\bj}$. Hence we have
\[
\sup_{w,\xi,\eta}\int |K(w',z';w,\xi,\eta)| dw'dz'<C_* |\omega|^{-1+2\theta_*}
\]
and 
\[
\sup_{w',z'}\int_{\supp \psi_{\bj}} |K(w',z';w,\xi,\eta)| dwd\xi d\eta<C_*e^{2|m|} \Delta_{\bj} |\omega|^{2\theta_*}.
\] 
By Schur test, the operator norm  of \eqref{eq:bgp} (and hence that of $\mathbb{G}-\widetilde{\mathbb{P}}$ )  is bounded by 
\begin{equation}\label{eq:boundPG}
C_*e^{|m|}\Delta_{\bj}^{1/2}|\omega|^{-1/2+2\theta_*} \le C_* e^{\delta_* t_\sharp}\cdot \omega_\sharp^{-1/2+3\theta_*}\ll \konst^{-1/2},
\end{equation}
where we used the estimate $\Delta_{\bj}\le C_* \log \langle\omega\rangle$, which follows from Lemma \ref{lm:Tor}, in the left inequality. (For the right inequality, recall that we choose large $\omega_\sharp$ depending on $t_\sharp$.)

Next we consider the difference $\mathbb{P}-\widetilde{\mathbb{P}}$. Its kernel $K'(w,\xi,\eta;w',\xi',\eta')$ is written 
\begin{align*}
&e^{i(\xi w-\xi' w')/2}\langle \eta\rangle^{1/2} \int dw'' \cdot e^{i(\xi-\xi')w''
-\langle \eta\rangle |w'-w''|^2/2}\\
& 
\cdot  \left(\langle \eta'\rangle^{1/2} e^{-\langle \eta'\rangle\|w-w''\|/2}\hat{\rho}^t_{\bj\to \bj'}(w',\eta'-\eta) -\langle \eta\rangle^{1/2} e^{-\langle \eta\rangle\|w-w''\|/2}\hat{\rho}^t_{\bj\to \bj'}(w'',\eta'-\eta)\right)
.
\end{align*}
By integration by parts using the estimate in Lemma \ref{lm:rhot}, we obtain that 
\[
|K'(w,\xi,\eta;w',\xi',\eta')|\le C_*(\nu)\cdot  (e^{C_* t_\sharp}\cdot \konst^{-1})\cdot
\langle |\omega|^{1/2} |w-w'|\rangle^{-\nu}\cdot \langle |\omega|^{-1/2} |\xi-\xi'|\rangle^{-\nu}\cdot 
\langle \eta'-\eta\rangle^{-\nu}
\]
for arbitrarily large $\nu>0$. 
Therefore the operator norm of $\mathbb{P}-\widetilde{\mathbb{P}}:L^2(\supp \psi_{\bj})\to L^2(\real^{2+2+1})$ is bounded by $C_* (e^{C_* t_\sharp}\cdot \konst^{-1})\ll \konst^{-1/2}$. 

From the estimates  on the differences $\mathbb{G}-\widetilde{\mathbb{P}}$ and $\mathbb{P}-\widetilde{\mathbb{P}}$ above, we obtain the conclusion of the lemma, provided that we take sufficiently large $t_\sharp$ and then take sufficiently large $\omega_\sharp$ according to the choice of~$t_\sharp$. 
\end{proof}

\subsection{Proof of Lemma \ref{lm:basic_components} and Lemma \ref{lm:hyp_component}} \label{ss:pf_lemmas}
The proofs of Lemma \ref{lm:basic_components} and Lemma \ref{lm:hyp_component} presented below are based on straightforward estimates on the kernels of the operators using integration by parts. (But recall Remark \ref{rem:smalldiff}.)
Below we will prove the claims for $\bbL^t_{\bj\to \bj'}$ in the conclusions, and those for the difference 
$\bbL^t_{\bj\to \bj'}-\hat{\bbL}^t_{\bj\to \bj'}$ will follow immediately from that proof if we note that the kernel of the operator $\bbL^t_{\bj\to \bj'}$ is localized in the space.

\begin{proof}[Proof of Lemma \ref{lm:basic_components}]
The operator norm of $\bbL^t_{\bj\to \bj'}$ is bounded by that of 
\[
(\pBargmann^*\circ \mathcal{M}(\tilde{q}_{\omega(\bj')})\circ \pBargmann)\circ \cL^t_{\bj\to \bj'}\circ (\pBargmann^*\circ \mathcal{M}(\tilde{q}_{\omega(\bj)})\circ \pBargmann).
\]
because we have $
(\pBargmann^*\circ \mathcal{M}(\tilde{q}_{\omega(\bj')})\circ \pBargmann)\circ \pBargmann^* = \pBargmann^*$ on $L^2(\supp \psi_{\bj})\subset L^2(\supp q_{\omega(\bj)})$. 
As we noted in the proof of Lemma \ref{lm:ii}, the operators on the both sides of $\cL^t_{\bj\to \bj'}$ above are convolution operators that involve solely the variable $z$. 
Thus it is easy to prove the claims
 using the estimate on  $\rho^t_{\bj\to \bj'}$ in Lemma~\ref{lm:rhot} (with the note that followed it) and the fact that the map $f^t_{\bj\to \bj'}$ is just a translation on each of the lines parallel to the $z$-axis. We omit the details of the proof since the argument is simple and will be clear from the next proof where we consider a parallel but more involved situations.
\end{proof}

\begin{proof}[Proof of Lemma \ref{lm:hyp_component}]
Let us set $\omega=\omega(\bj)$, $\omega'=\omega(\bj')$, $m=m(\bj)$ and $m'=m(\bj')$ for brevity. We proceed with the assumption 
\begin{equation}\label{eq:omegaclose}
|\omega'-\omega|\le e^{\max\{|m|,|m'|\}/10}
\end{equation}
because the claims follow from Lemma \ref{lm:basic_components} otherwise. 
Below we consider the operator
\[
\check{\bbL}^t_{\bj\to \bj'}:= (\pBargmann^*\circ \mathcal{M}(\tilde{\psi}_{\bj'})\circ \pBargmann) \circ \cL^t_{\bj\to \bj'}\circ 
(\pBargmann^*\circ \mathcal{M}(\tilde{\psi}_{\bj})\circ \pBargmann): L^2(\real^{2+1})\to L^2(\real^{2+1})
\]
and show that 
\begin{equation}\label{eq:checkL}
\|\check{\bbL}^t_{\bj\to \bj'}: L^2(\real^{2+1})\to L^2(\real^{2+1})\|\le C_*(\nu) \cdot e^{-\max\{m,m'\}/2}\langle \omega'-\omega\rangle^{-\nu}. 
\end{equation}
Let us write the operator $\check{\bbL}^t_{\bj\to \bj'}$ above as an integral operator:
\[
\check{\bbL}^t_{\bj\to \bj'} u(w',z')=\int K(w',z';w,z) u(w,z) dw dz
\]
where the kernel is written explicitly as an integral
\begin{align}\label{eq:Kwz}
K(w',z';w,z)=&\int dw''  dz'' \;   d\tilde{w} d\xi d\eta\;   d\tilde{w}' d\xi' d\eta'  \cdot \rho^t_{\bj\to \bj'}(w'',z'')\\
&\cdot \phi_{\tilde{w}',\xi',\eta'}(w',z') \cdot \tilde{\psi}_{\bj'}(\tilde{w}',\xi',\eta') \cdot\overline{\phi_{\tilde{w}',\xi',\eta'}(f^t_{\bj\to \bj'}(w'',z''))}\notag\\
&\cdot \phi_{\tilde{w},\xi,\eta}(w'',z'') \cdot \tilde{\psi}_{\bj}(\tilde{w},\xi,\eta) \cdot \overline{\phi_{\tilde{w},\xi,\eta}(w,z)}.\notag
\end{align}
We estimate this integral, regarding it as an oscillatory integral with the oscillatory term
\[
\exp\left(i\left((\xi',\eta')\cdot {f}^t_{\bj\to \bj'}(w'',z'')-(\xi,\eta) \cdot (w'',z'')\right)\right).
\] 
More concretely, we apply  the formula of integration by parts to the integral above, once by using the differential operator $\mathcal{D}_1$  below and several times  using $\mathcal{D}_2$ in addition:
\begin{align*}
\mathcal{D}_1=
&\frac{1-i\left((Df^t_{\bj\to \bj'})_{w''}^*(\xi',\eta')-(\xi,\eta)\right)\cdot \partial_{w''}}
{1+\|(Df^t_{\bj\to \bj'})_{w''}^*(\xi',\eta')-(\xi,\eta)\|^2},\quad \mathcal{D}_2=
\frac{1-i(\eta'-\eta)\cdot \partial_{z''}}
{1+\|\eta'-\eta)\|^2}.
\end{align*}
Further we apply integration by parts several times regarding the terms
\[
\exp(i \xi'(w'-\hat{f}^t_{\bj\to \bj'}(w'')))\quad \text{and}\quad 
\exp(i \xi(w''-w))
\]
as the oscillatory term and using 
the differential operators
\[
\mathcal{D}_3=
\frac{1-i \cdot e^{|m'|} \langle \omega'\rangle^{1/2} \textcolor{\revisionColor}{\bDelta_{\bj'}}(w'-\hat{f}^t_{\bj\to \bj'}(w'')) \cdot e^{|m'|} \langle \omega'\rangle^{1/2}\textcolor{\revisionColor}{\bDelta_{\bj'}}\partial_{\xi'}}
{1+\|e^{|m'|} \langle \omega'\rangle^{1/2}\bDelta_{\bj'}(w'-\hat{f}^t_{\bj\to \bj'}(w''))\|^2}
\]
and
\[
\mathcal{D}_4=\frac{1-i \cdot e^{|m|}\langle\omega\rangle^{1/2} \textcolor{\revisionColor}{\bDelta_{\bj}}(w''-w) \cdot e^{|m|} \langle\omega\rangle^{1/2} \bDelta_{\bj}\partial_{\xi}}
{1+\|e^{|m|} \langle\omega\rangle^{1/2} \textcolor{\revisionColor}{\bDelta_{\bj}}(w''-w)\|^2}.
\]
To evaluate the result, we use  the basic estimates \eqref{eq:crude_estimate_df} and \eqref{eq:est_on_hatf} for $f^t_{\bj\to \bj'}$,  Lemma \ref{lm:rhot} for $\rho^t_{\bj\to \bj'}$ and also \eqref{eq:hook} in Lemma \ref{lm:hook}. 
Then we can deduce 
\begin{align}\label{eq:estK}
|K(w,z;w',z')|\le 
&\bigg(C_*(\nu) +C_*(\nu,t_{\sharp})\cdot e^{-|m|}\Delta_{\bj}^{-1}\cdot \langle e^{|m'|} \Delta_{\bj'}\cdot \langle \omega'\rangle^{-1/2}\rangle\bigg) \cdot  e^{-|m|}\cdot  \langle z'-z\rangle^{-\nu} \\
&\cdot \int dw''\cdot 
 e^{2|m'|} \Delta_{\bj'}\cdot  \langle \omega'\rangle \cdot\langle e^{|m'|}\langle \omega'\rangle^{1/2}\|\textcolor{\revisionColor}{\bDelta_{\bj'}}(w'-\hat{f}_{\bj\to \bj'}(w''))\|\rangle^{-\nu}
 \notag\\
 & \qquad \quad\cdot
e^{2|m|} \Delta_{\bj}\cdot  \langle \omega\rangle \cdot\langle e^{|m|} \langle \omega\rangle^{1/2}\cdot \|\textcolor{\revisionColor}{\bDelta_{\bj}}(w''-w)\|\rangle^{-\nu}\notag
\end{align}
for arbitrarily large $\nu$, 
provided that we let the constant  $t_\sharp$ and $\omega_{\sharp}$ be large enough. 
\begin{Remark}
The estimate to get \eqref{eq:estK} is demanding but straightforward and not too difficult. Note that we get the factor $D_{w''}((D_{w''}f^t_{\bj\to \bj'})^*_{w''}(\xi',\eta'))$  in the (first) integration by parts using $\mathcal{D}_1$, which is bounded  as 
\[
\|D_{w''}((D_{w''}f^t_{\bj\to \bj'})^*_{w''}(\xi',\eta'))\|\le C_*(\nu,t_{\sharp})\cdot  \|(\xi',\eta')\|\le C_*(\nu,t_{\sharp})\cdot
\langle e^{m'} \Delta_{\bj'}\cdot \langle \omega'\rangle^{-1/2}\rangle \cdot \langle \omega'\rangle.
\]
The term $C_*(\nu,t_{\sharp})\cdot \langle e^{m'} \Delta_{\bj}\cdot \langle \omega'\rangle^{-1/2}\rangle$ in \eqref{eq:estK} is put in order to  bound the terms related to this factor. The estimates on the other terms are simple. 
\end{Remark}

If we consider the change of variables $w''$ to $w'''=f^t_{\bj\to \bj'}(w'')$ in \eqref{eq:Kwz} and proceed in parallel to  the argument above, replacing $f^t_{\bj\to \bj'}$ by its inverse in some places and  using \eqref{eq:hookrev} in the place of \eqref{eq:hook},  then we reach the following estimate similar to \eqref{eq:estK},
\begin{align}\label{eq:estK2}
|K(w,z;w',z')|\le 
&\bigg(C_*(\nu) +C_*(\nu,t_{\sharp})\cdot e^{-|m'|}\Delta_{\bj'}^{-1}\cdot \langle e^{|m|} \Delta_{\bj}\cdot \langle \omega\rangle^{-1/2}\rangle\bigg) \cdot  e^{-|m'|}\cdot  \langle z'-z\rangle^{-\nu} \\
&\cdot \int dw''\cdot 
 e^{2|m'|} \Delta_{\bj'}\cdot  \langle \omega'\rangle \cdot\langle e^{|m'|}\langle \omega'\rangle^{1/2}\|\textcolor{\revisionColor}{\bDelta_{\bj'}}(w'-\hat{f}_{\bj\to \bj'}(w''))\|\rangle^{-\nu}
 \notag\\
 & \qquad \quad\cdot
e^{2|m|} \Delta_{\bj}\cdot  \langle \omega\rangle \cdot\langle e^{|m|} \langle \omega\rangle^{1/2}\cdot \|\textcolor{\revisionColor}{\bDelta_{\bj}}(w''-w)\|\rangle^{-\nu}.\notag
\end{align}

If we assume the conditions 
\begin{equation}
\label{eq:ass1}
\min\{\omega,\omega'\}\ge \omega_\sharp/2\quad\text{and}\quad 
e^{\max\{|m|,|m'|\}}\le  \langle \omega\rangle^{1/10},
\end{equation}
the ratio between $|\omega|$ and $|\omega'|$ are close to $1$, since we are assuming \eqref{eq:omegaclose},  and the second term $C_*(\nu,t_{\sharp})\cdot e^{-|m|}\Delta_{\bj}^{-1}\cdot \langle e^{|m'|} \Delta_{\bj'}\cdot \langle \omega'\rangle^{-1/2}\rangle$ in the parentheses on the right-hand side 
of (\ref{eq:estK}) and (\ref{eq:estK2}) is bounded by $C_*(\nu)$. Hence, applying either (\ref{eq:estK}) or (\ref{eq:estK2}) according to whether $|m|>|m'|$ or not, the  claim \eqref{eq:checkL} on $\check{\bbL}^t_{\bj\to \bj'}$ follows immediately from Schur test. 

Suppose on the other hand that the condition \eqref{eq:ass1} does not hold. Then we have
\begin{equation}\label{eq:casetwo}
e^{\max\{|m|,|m'|\}}\ge \min\{e^{m_\sharp}, \omega_{\sharp}^{1/10}\}
\end{equation} 
because, if the former condition in \eqref{eq:ass1} fails,  we have $\max\{|m|,|m'|\}\ge m_\sharp$ as we are assuming \eqref{eq:omegaclose} and disregarding  the case where \eqref{ass:momega} holds. Hence, letting the constant $m_{\sharp}$ be large according to $\omega_\sharp$ and $t_\sharp$, we again see that (\ref{eq:estK}) and (\ref{eq:estK2}) yield  \eqref{eq:checkL}.

Finally we deduce the required estimate \eqref{claim:hyp_component1} from \eqref{eq:checkL}. Note that \eqref{eq:checkL} implies   
\[
\|\pBargmannP\circ \mathcal{M}(\tilde{\psi}_{\bj'})\circ 
\bbL^t_{\bj\to \bj'}\circ  \mathcal{M}(\tilde{\psi}_{\bj})\circ \pBargmannP\|_{L^2(\real^{2+2+1})}\le C_*(\nu)\cdot  e^{-\max\{|m|, |m'|\}/2}\cdot \langle \omega'-\omega\rangle^{-\nu}
\]
where $\pBargmannP=\pBargmann\circ \pBargmann^*$ is the (partial) Bargmann projector. 
Our task is to eliminate the terms $\pBargmannP\circ \mathcal{M}(\tilde{\psi}_{\bj'})$ and $ \mathcal{M}(\tilde{\psi}_{\bj})\circ \pBargmannP$ on the both sides of $
\bbL^t_{\bj\to \bj'}$. 
From the description of the kernel of $\pBargmannP$ in Lemma \ref{lm:pBargmannL2}, it is easy to see that  
\[
\|(1-\mathcal{M}(\tilde{\psi}_{\bj}))\circ \pBargmannP:L^2(\supp \psi_{\bj})\to L^2(\real^{2+2+1})\|< \exp(-e^{|m|/2})
\]
and the same estimate with $\bj$ and $m$ replaced by $\bj'$ and $m'$ respectively. Thus obtain the required estimate \eqref{claim:hyp_component1} provided that  $\max\{e^{|m|/2}, e^{|m'|/2}\} \ge \max\{ |m|, |m'|\}$. 
Unfortunately, the last condition fails if the ratio between $|m|$ and $|m'|$ is extremely large. However the problem in such case is  superficial and it is easy to provide a simple remedy for it. Indeed,  the argument above remains valid as far as $\tilde{\psi}_{\bj}$ and $\tilde{\psi}_{\bj'}$ fulfill the conditions that they are sufficiently smooth functions taking constant value $1$ on a (scaled) neighborhood of the supports of $\tilde{\psi}_{\bj}$ and $\tilde{\psi}_{\bj'}$ respectively, and that the conclusion of Lemma \ref{lm:hook} holds for them. Thus, modifying the definitions of  $\tilde{\psi}_{\bj}$ and $\tilde{\psi}_{\bj'}$ appropriately, we can get the conclusion  \eqref{claim:hyp_component1}. (For instance, when $m>0$ and $m'>0$, the problem happens when $m$ is much smaller than $m'$ and the remedy is to enlarge the support of $\tilde{\psi}_{\bj}=\tilde{\psi}_{a(\bj),n(\bj),\omega(\bj), m(\bj)}$ so that its size  is comparable to that of $\tilde{\psi}_{\bj'}$.)
\end{proof}

\subsection{Local uniformity of exponential mixing}\label{ss:localuniformity}
Finally we prove Theorem \ref{th:exp}. 
Let us write $f^t_0$ for the flow $f^t$ that we have considered in  the previous subsections. 
We first show that, if we take a sufficiently small $C^3$ neighborhood $\mathcal{V}$ of $f^t_0$ in $\fF^3_A$, each of the flows in $\mathcal{V}$ is exponentially mixing. 
To this end, we recall the arguments in the previous subsections and check dependence of objects on the flow. 
We can construct the local charts $\kappa_{a,\omega,n}$ and the functions $\rho_{a,\omega,n}$ so that \emph{each} of them depend on the flow continuously in $C^3$ sense. Then we can define the Hilbert space $\mathbb{H}$ and $\mathcal{H}$ and also the operator $\bbL^t:\mathbb{H}\to \mathbb{H}$ in a parallel manner so that each of the components $\bbL^t_{\bj\to \bj'}$ depend on the flow continuously. Recall from Remark \ref{rem:perturb} that, to go through the arguments in the previous subsections, we actually needed the estimate \eqref{eq:tv} only for $b$ in a bounded interval $[b_0, \konst]$. And, letting the neighborhood $\mathcal{V}$ be smaller if necessary, we may assume that this is true for \emph{all} $f^t\in \mathcal{V}$. (Recall Remark \ref{Rem:largealpha} for the case $|\alpha|>b$.) 
Therefore one can check that all the arguments in the previous subsections remains valid for each $f^t\in \mathcal{V}$ and the constants denoted by the symbols with the subscript $*$ and also $t_\sharp$, $\omega_\sharp$, $m_\sharp$ can be taken uniformly. 
That is, \emph{each} of the flows in $\mathcal{V}$ is exponentially mixing. 
 
We next consider uniformity of the constants $c_\alpha$ and $C_\alpha$ in the decay estimate \eqref{eq:cor}. This is not very simple to see because  continuity in the dependence of the local charts $\kappa_{\bj}$ and the operators $\bbL^t_{\bj\to \bj'}$ on the flow in $\mathcal{V}$ is \emph{not} uniform (especially in the limit $|\omega(\bj)|\to \infty$). 
To do with this problem, we use an indirect argument by contradiction. 
Let $\mathcal{H}(\mathbf{f})$ be the anisotropic Sobolev space $\mathcal{H}$ defined for a flow $\mathbf{f}=\{f^t\}\in \mathcal{V}$ and consider its  subspace  $\mathcal{H}_0(\mathbf{f})=\{u\in \mathcal{H}(\mathbf{f})\mid \int u dm=0\}$. Let $\cL^t_{\mathbf{f}}$ be the transfer operator $\cL^t$ defined for $\mathbf{f}\in \mathcal{V}$. 
To obtain the conclusion, it is enough to show, for some $T>0$ and $\delta>0$, that 
\[
\|\cL^T_{\mathbf{f}}:\mathcal{H}_0(\mathbf{f})\to \mathcal{H}_0(\mathbf{f})\|<1-\delta\quad\mbox{for all $\mathbf{f}\in \mathcal{V}$}
\]
provided that $\mathcal{V}$ is a sufficiently small neighborhood  of $\mathbf{f}_0=\{f^t_0\}$.
Suppose that this assertion is not true, that is,  for an arbitrarily large  $T>0$,  we can find a sequence of flows $\mathbf{f}_{k}$ which converges to $\mathbf{f}_0$ in $C^3$ sense and a sequence of distributions $u_{k}\in \mathcal{H}_0(\mathbf{f}_k)$ such that 
\[
\|u_{k}\|_{\mathcal{H}(\mathbf{f}_k)}=1\quad\text{and}\quad
\|\cL^T_{\mathbf{f}_k}u_k\|_{\mathcal{H}(\mathbf{f}_k)}\ge 1-(1/k).
\]
Notice that the conclusion of Proposition \ref{pp:main_est} is valid uniformly for $\mathbf{f}\in \mathcal{V}$, so that the operators $\cL^T_{\mathbf{f}_k}$ contract the high frequency parts of functions ({\it i.e.} the components $u_{\bj}$ with $|\omega(\bj)|\ge \omega_\sharp$ or $|m|\ge m_{\sharp}$)  by a uniform rate $<1$.  Hence\footnote{We actually need additional arguments to provide precise statements and their proofs for the claims in this sentence and the next. But we  leave it to the readers one because the claims should be intuitively clear and one because qualitative estimates that follow from the arguments in the previous sections are enough for the proofs.}, for the assumption above on $u_k$ to be true, the high frequency part of $u_k$ must be relatively small uniformly in $k$. 
 This implies that there exists a subsequence $u_{k(\ell)}$ of $u_k$ which converges to some $u_{\infty}\in \mathcal{H}_0(\mathbf{f}_0)$ satisfying
\[
\|u_{\infty}\|_{\mathcal{H}(\mathbf{f}_0)}=\lim_{\ell\to \infty}\|u_{k(\ell)}\|_{\mathcal{H}(\mathbf{f}_k)}=1, \qquad \|\cL^T_{\mathbf{f}_0}u_\infty\|_{\mathcal{H}(\mathbf{f}_0)}=\lim_{\ell\to \infty}\|\cL^T_{\mathbf{f}_k}u_{k(\ell)}\|_{\mathcal{H}(\mathbf{f}_k)}\ge 1.
\]
Clearly this conclusion for arbitrarily large $T>0$ contradicts what we have proved for $\mathbf{f}_0$. 

\begin{Remark}
The argument in the proof of local uniformity of the constants $c_\alpha$ and $C_\alpha$ in the decay estimate \eqref{eq:cor} is indirect and not very satisfactory. 
It would be much better if we could apply the perturbation theory of transfer operator developed by Keller and Liverani\cite{MR1679080}  (see also \cite{MR2201945}). 
But, for the moment, it seems that the theory in \cite{MR1679080} is not applicable (at least, directly) to our setting because the anisotropic Sobolev space $\mathcal{H}(\mathbf{f})$ depends sensitively on the flow $\mathbf{f}$ through the choice of infinitely many local charts $\kappa_{\bj}$.   
\end{Remark}
\bibliographystyle{plain}
\bibliography{ExpDecay}

\end{document}